\documentclass[a4paper]{scrartcl}
\usepackage[english]{babel}
\usepackage{amsfonts,amsmath}
\usepackage{amsthm}
\usepackage{txfonts}
\usepackage{tikz}
\usepackage{fancyhdr}

\usetikzlibrary{decorations.pathreplacing}
\newenvironment{remark}[1][Remark]{\begin{trivlist}
\item[\hskip \labelsep {\bfseries #1}]}{\end{trivlist}}
\newenvironment{assumptions}[1][Assumptions]{\begin{trivlist}
\item[\hskip \labelsep {\bfseries #1}]}{\end{trivlist}}
\newenvironment{notations}[1][Notations]{\begin{trivlist}
\item[\hskip \labelsep {\bfseries #1}]}{\end{trivlist}}

\newtheorem{theorem}{Theorem}[section]
\newtheorem{lemma}[theorem]{Lemma}
\newtheorem{proposition}[theorem]{Proposition}
\newtheorem{corollary}[theorem]{Corollary}
\newtheorem{definition}[theorem]{Definition}

\newcommand{\Rzwei}{\mathbb{R}^2}
\newcommand{\Rnplus}{\mathbb R^{n+1}}
\newcommand{\N}{\mathbb{N}}

\newcommand{\R}{\mathbb{R}}

\newcommand{\defi}{\coloneqq}
\newcommand{\de}{\text{\,d}}

\newcommand{\heatkernel}{\rho_{x_0,T}}
\newcommand{\dist}{\text{dist}}
\newcommand{\spt}{\mbox{spt}}
\newcommand{\Div}{\mbox{div}}
\newcommand{\nix}{\hspace*{0.2em}\hspace*{-0.2em}}
\newcommand{\ftilde}{\nix^\Sigma\tilde f}
\newcommand{\fsig}{\nix^\Sigma f}
\newcommand{\Sig}{\nix^\Sigma}
\newcommand{\winkel}{\angle_{\Rzwei}}
\newcommand{\ind}{ \mbox{ind}}

\newcommand{\pp}{\partial_p}   
    
\newcommand{\pt}{\partial_t}  
\newcommand{\ps}{\partial_s}       
\newcommand{\ct}{c_t}   
  
\newcommand{\gat}{\gamma_t} 
\newcommand{\intab}{\int_a^b} 
\newcommand{\diam}{\text{diam}}

\newcommand{\nut}{[0,T)}
\newcommand{\tildega}{\tilde\gamma}
\newcommand{\tildegat}{\tilde\gamma_t}

\newcommand{\ab}{[a,b]}
\newcommand{\tn}{t_0}
\newcommand{\tildec}{\tilde c_j}
\newcommand{\gainfty}{\tilde\gamma_\infty}
\newcommand{\psij}{\psi_j}

\newcommand{\Tdel}{\frac{T}{\delta}}

\newcommand{\gaj}{\gamma_j}
\newcommand{\Ij}{I_j}

\newcommand{\phir}{\varphi_{(x_0,T),\lambda}}
\newcommand{\tildephi}{ \tilde\varphi_{(Q_j\lambda)}}

\newcommand{\Minfty}{\tilde M^{\infty}_{\tau}}

\newcommand{\sigmax}{\Sig\kappa_{\text{max}}}
\newcommand{\dt}{\frac{\text{d}}{\text{d}t}}
\newcommand{\diff}{(\kappa-\bar\kappa)}
\newcommand{\tj}{t_j}

\newcommand{\ctn}{c_{t_0}}
\newcommand{\ctnt}{c_{\tilde t_0}}

\title{The area preserving curve shortening flow with Neumann free boundary conditions}

\author{Elena M\"ader-Baumdicker}

\date{}

\renewcommand*\theenumi{\textit{\roman{enumi}}}
\renewcommand*\labelenumi{\theenumi)}

\begin{document}

\maketitle

\begin{abstract}
 We study the area preserving curve shortening flow with Neumann free boundary conditions outside of a convex domain in the Euclidean plane. Under certain conditions on the initial curve the flow does not develop any singularity, and it sub\-con\-ver\-ges smoothly to an arc of a circle sitting outside of the given fixed domain and enclosing the same area as the initial curve.
\end{abstract}

\section{Introduction}

For a family of simple closed plane curves $\gamma=\gamma:\mathbb S^1\times[0,T)\to\Rzwei$, Gage introduced in \cite{Gage} the \emph{area preserving curve shortening flow} as
\begin{align}\label{E1}
 \frac{\partial \gamma }{\partial t} &= \left(\kappa-\frac{\int\kappa\de s}{L}\right) \nu,
\end{align}
where $\nu(\cdot,t)$ is the normal of the curve $\gamma(\cdot,t)$ and $\kappa (\cdot,t)$ is its curvature with respect to $\nu (\cdot,t)$. The length of $\gamma(\cdot,t)$ is denoted by $L$ and $\de s$ denotes integration by arclength. The term $\frac{\int\kappa\de s}{L} \eqqcolon \bar\kappa$ is the average of the curvature. For simple closed curves $\bar\kappa =\frac{2\pi}{L}$ holds.
Gage pointed out that this evolution equation arises as the ``$L^2$-gradient flow'' of the length functional under the constant enclosed area constraint. The term $\bar\kappa$ is the suitable Lagrange parameter for this variational problem. 
Gage proved that a strictly convex simple closed curve which evolves according to (\ref{E1}) remains strictly convex and converges to a circle in the $C^\infty$ metric as $t\to\infty$, see \cite[Theorem~4.1]{Gage}.
Note that there is a connection to another problem in geometric analysis, namely the isoperimetric problem in $\Rzwei$, i.e.~finding the shortest curve under all closed curves that enclose a certain fixed area. The area preserving curve shortening flow deforms a given convex, simple closed curve with enclosed area $A_0$ to the solution of the isoperimetric problem in $\Rzwei$: the circle with enclosed area $A_0$.
The same result was proven by Huisken in \cite{Huisken87} for $n\geq 2$. In this case, the evolution equation is as follows: Let $M^n$ be a compact $n$-dimensional manifold without boundary and $F_0:M^n\to\R^{n+1}$ a smooth immersion. 
A family $F:M^n\times[0,T)\to\Rnplus$ with $T>0$ is a solution to the \emph{volume preserving mean curvature flow} if it is  a smooth family of smooth immersions satisfying
\begin{alignat*}{2}
 \dt F (p,t) &=-( H(p,t)-h(t))\nu(p,t) \hspace*{1cm}&& \text{for } (p,t)\in M^n\times [0,T), \\
 F(p,t) & = F_0 (p)  && \text{for }  p\in M^n,
\end{alignat*}
where $H$ denotes the mean curvature, $\nu$ the outer unit normal and $ h(t)\defi\frac{\int_{M^n} H(p,t)\de \mu_t}{\int_{M^n}\de\mu_t}$
 the average of the mean curvature. Note that $h$ is a global term in the equation. Thus, the analysis of this problem is not only a matter of local estimates but also of global considerations.

We are interested in considering the flow equation (\ref{E1}) for curves with boundary.
In \cite{StahlDiss} (see also \cite{StahlPaper1, StahlPaper2}), Stahl considered the mean curvature flow with Neumann free boundary conditions on an arbitrary support surface $\Sigma$: Let $\Sigma$ be a smooth, embedded hypersurface in $\Rnplus$ without boundary, called support surface. 
Consider $M^n$, a smooth $n$-dimensional compact, orientable manifold with smooth boundary, and $F_0:M^n\to \Rnplus$, a smooth immersion, such that with the notation $M_0\defi F_0(M^n)$, 
\begin{align*}
 \partial M_0  \defi F_0(\partial M^n) & = M_0\cap \Sigma\\
\langle \nu_0(p),\, ^\Sigma\vec\nu(F_0(p))\rangle &=0  \ \ \ \ \forall p\in\partial M^n, 
\end{align*}
where $\Sig\vec\nu:\Sigma\subset\R^{n+1}\to\mathbb S^n$ is a unit normal to $\Sigma$. Then $M_0$ is called initial surface. Now let $F:M^n\times [0,T)\to\Rnplus$ be a smooth family of smooth immersions which are evolving under the \emph{mean curvature flow with Neumann free boundary conditions}, i.e.
\begin{alignat}{2}
 \frac{\partial F}{\partial t} (p,t)  &=-H(p,t)\nu(p,t)\hspace*{1cm}  & \forall (p,t)&\in M^n\times(0,T),\label{k1}\\
 F(p,0)&=F_0(p) & \forall p&\in M^n,\label{k2}\\
F(p,t)&\in\Sigma & \forall (p,t)&\in\partial M^n\times [0,T),\label{k3}\\
\langle \nu(p,t),& ^\Sigma\vec\nu(F(p,t))\rangle=0 & \forall (p,t)&\in \partial M^n\times [0,T).\label{k4}
\end{alignat}
Under the assumption that $\Sigma$ is mean convex, i.e.~$\Sig H\geq 0$, with respect to the outward\footnote{In the work of Stahl, ``outward'' denotes ``outward relative to a solution $M_t$''. This means: if $\Sigma$ is the boundary of a smooth bounded domain $G$ and $M_t$ touches $\Sigma$ from inside, then the outward unit normal $\mu$ of $\Sigma$ is the ``usual'' outer unit normal pointing in the outer domain with respect to the domain~$G$.} unit normal, Stahl showed that an embedded initial surfaces stays embedded under the flow. If $\Sigma$ is umbilic and the surfaces touch $\Sigma$ from inside (if $\Sigma$ is not a hyperplane but a sphere), then also convexity and a pinching condition is preserved. In the end, an embedded strictly convex hypersurface shrinks to a single point on $\Sigma$ ($t\to T<\infty$), where the singularity at time $T$ is of type I. After rescaling, a sequence of hypersurfaces converges in the $C^\infty$-topology to a hemisphere with boundary in a hyperplane, see \cite[Theorem 1.3]{StahlPaper1}. Regarding curves, he proved: If $\Sigma=\partial G$, where $G\subset\Rzwei$ is a convex domain, consider a convex embedded initial curve $M_0\subset \bar G$ evolving according to (\ref{k1}) to (\ref{k4}), then the curves converge to a single point on $\Sigma$ as $t\to T<\infty$ (see \cite[Proposition~1.4]{StahlPaper1}).

%
%
As in the context without boundary, it is interesting to consider the situation of the \emph{volume preserving mean curvature flow}, but now \emph{with Neumann free boundary conditions}. Athanassenas studied in \cite{Athanassenas1997} and \cite{Athanassenas2003} the situation of rotationally symmetric surfaces between two parallel hyperplanes and with boundary in this hyperplane intersecting these hyperplanes orthogonally at the boundary. No convexity condition is assumed.
In the first paper, Athanassenas considered initial surfaces satisfying $|M_0|\leq \frac{V}{d}$, where $V$ is the enclosed volume of $M_0$ and $d$ is the distance between the hypersurfaces. She showed longtime existence up to infinity and convergence to a cylinder which is orthogonal to the two hyperplanes and which encloses the volume $V$. 
In \cite{Athanassenas2003}, Athanassenas showed that in general the flow develops singularities. She proved this in the same setting as above but without any condition concerning the area of $M_0$. But the singular set is finite and discrete, and it is located along the axis of rotation. If $M_0$ is mean convex then all singularities are of type I.
In \cite{Athanassenas2012}, Athanassenas and Kandanaarachchi considered rotationally symmetric surfaces with boundary in one hyperplane and intersecting the hyperplane orthogonally at the boundary. Under a certain lower height bound they proved that the flow does not develop singularities and it converges to a half-sphere. 
The results of Athanassenas were generalized by Cabezas-Rivas and Miquel to surfaces of revolution in a rotationally symmetric ambient space, see \cite{CabezasMiquel2009, CabezasMiquel2012}.

In this article, we consider the volume preserving equation (\ref{E1})
and the boundary condition as in the papers of Stahl, namely (\ref{k2}) to (\ref{k4}), but for curves. 
We focus on the case of a convex initial curve meeting the convex support curve from the outside. In the classical mean curvature flow (volume preserving or not) for closed surfaces, many proofs for surfaces cannot be done for the case of curves. 
Some results of this article can be transfered to the case of surfaces, but some not. 

In Section~\ref{S1}, we introduce some notation and explain the setting in detail. Afterwards we show basic properties like preserving of convexity and preserving of the oriented enclosed area. Most of this section can be found in the Diplom thesis of Achim Roth\footnote{Formerly known as Achim Windel.}, see \cite{AchimDipl}.
%
%
We present in Section~\ref{S2} some more results of \cite{AchimDipl}, 
in particular an $L^\infty$-bound on the average of the curvature $\bar\kappa$ under the following conditions: The initial curve is strictly convex, has no self-intersection, it is completely contained in the outer domain created by the support curve and it satisfies $L(c_0)<\min\{|x-y|:x,y\in\Sigma,\Sig\vec\tau(x)= - \Sig\vec\tau(y)\}$. 
In Section~\ref{S3}, we prove that there are no type I singularities in finite time under the conditions that the initial curve is convex and the flow satisfies  $|\bar\kappa|\leq c_2<\infty$ and $L(\ct)\geq c_1>0$ uniformly in time.
Section~\ref{geometricproperties} contains the proof that an initial curve with the properties of Section~\ref{S2} stays embedded if the flow satisfies $\int\kappa\de s<\pi + \frac{\pi}{2}$. Furthermore, if the initial curve is short enough then the curve $c(\cdot,t)$ and the line segment from $c(a,t)$ to $c(b,t)$ trace out a convex domain.
Singularities of type II are studied in Section~\ref{S5}. We describe conditions on the initial curve which have the effect that the corresponding flow does not develop any singularity. The initial curve is assumed to satisfy the conditions of Section~\ref{geometricproperties}, but additionally, it has to be shorter than a constant which depends on the ''isoperimetric quotient`` $\frac{A_0}{L(c_0)^2}$ of the initial curve ($A_0$ denotes the enclosed area at time $t=0$).
Finally, in Section~\ref{S6} we first show integral estimates and then our main result:

\newpage

\begin{theorem} \label{mainthm}
  Let $c:[a,b]\times[0,T)\to \Rzwei $ be a solution of the area preserving curve shortening problem with Neumann free boundary conditions outside of  a convex support curve $\Sigma$. Let the initial curve $c_0$ satisfy the following four conditions: 
 \begin{enumerate}
  \item The curvature $\kappa_0$ is positive on $[a,b]$,
  \item the curve $c_0$ has no self-intersection,
  \item it is contained in the outer domain created by $\Sigma$
  \item and it satisfies $L_0< C$, where $C\defi\frac{4}{5 \sigmax} \arcsin\left(\frac{A_0}{L_0^2}\right)$.
 \end{enumerate}
    Then the flow does not develop a singularity, consequently $T=\infty$. Furthermore, for every sequence $t_j\to\infty$ the curves $c(\cdot,t_j)$ subconverge (after reparametrization) smoothly to an embedded arc of a circle $\gamma_\infty$ with $\int\kappa_\infty\de s_{\gamma_\infty} =\varphi$, $\varphi\in [\pi,2\pi)$. The two boundary points of this arc lie on $\Sigma$, the arc is perpendicular to $\Sigma$ at these points, and the curve is outside of the convex domain created by $\Sigma$. The curve $\gamma_\infty$ encloses (with a part of $\Sigma$) the area $A_0$.
\end{theorem}
We collect some useful calculations in the appendix. All the results presented here are part of the authors thesis \cite{MeineDiss} which was written under the supervision of Prof. Kuwert in Freiburg.

%
%
%

\begin{center} 
 {\large Acknowledgement}
\end{center}
I thank my advisor Prof. Kuwert for his excellent supervision during my time as a doctoral student in Freiburg. I was partially supported by the DFG Collaborative Research Center SFB/Trans\-re\-gio 71, and I would like to express my thanks to the DFG for that.

\section{Notations and basic properties} \label{S1}

\begin{notations}
In this article we deal with regular, planar curves $f:[a,b]\to\Rzwei$. The two properties ``regular'' and ``planar'' will not be mentioned anymore. The curves are at least piecewise $C^{2,\alpha}$. Therefore, the curvature  is $ \kappa (p) = \langle \partial^2_s f(p),\nu(p)\rangle =  \frac{\langle \partial^2_p f(p),\nu(p)\rangle}{|\partial_p f(p)|^2}$ where the curve is smooth. Here, $\partial_s = \frac{1}{|\partial_p f|}\partial_p$ denotes the derivative with respect to  arclength, and $\nu=J\tau$ is the normal of the curve, where $\tau=\ps f$ denotes the tangent and $J$ the rotation by $+\frac{\pi}{2}$.
 \end{notations}

\begin{definition}
  We call a smooth, convex, simple and smoothly closed curve $\nix^\Sigma f:[\nix^\Sigma a,\nix^\Sigma b]\to\Rzwei$ that is parametrized by arclength a \emph{support curve}. 
  We orientate the curve positively such that $\nix^\Sigma\kappa\geq 0$ follows. We use the notation $\Sigma\defi \nix^\Sigma f\left([\nix^\Sigma a,\nix^\Sigma b]\right). $ The bounded, ``inner'' domain that is created by $\Sigma$ is denoted by $G_\Sigma$.
\end{definition}

\begin{definition}
 A smooth curve $c_0:[a,b]\to\Rzwei$ is called \emph{initial curve} if it satisfies the conditions
 \begin{align*}
 c_0(a),c_0(b)\in\Sigma, \ \ \
\angle_{\Rzwei} \left(\tau_0(a),\nix^\Sigma \vec\tau(c_0(a))\right) = \frac{\pi}{2}, \ \ \
\angle_{\Rzwei} \left(\tau_0(b),\nix^\Sigma \vec\tau(c_0(b))\right) = -\frac{\pi}{2},
\end{align*}
  where $\tau_0:[a,b]\to\mathbb S^1\subset\Rzwei$ is the tangent of $c_0$ and $\nix^\Sigma\vec \tau$ is defined through $$\Sig\vec\tau:\Sigma\to\mathbb S^1\subset\Rzwei,\ \ \Sig\tau=\Sig\vec\tau\circ\fsig,$$  and $\Sig\tau$ is the tangent of $\Sigma$. 
  The symbol $\angle_{\Rzwei} \left(v,w\right)$ denotes the oriented angle between two vectors $v$ and $w$ in $\Rzwei$.
\end{definition}

\begin{definition}
 Let $c_0:[a,b]\to\Rzwei$ be an initial curve. A smooth family of smooth, regular curves $c:[a,b]\times[0,T)\to \Rzwei $ that satisfies  
 \begin{align}\begin{split} \label{flow}
   \frac{\partial c}{\partial t} (p,t) = ( \kappa(p,t)& - \bar \kappa(t))\nu(p,t) \qquad\forall (p,t)\in [a,b]\times [0,T),\\
   c(p,0)& =c_0(p) \qquad \hspace*{1.28cm}\forall p\in [a,b],\\[0.2cm]
      c(a,t),c(b,t)&\in\Sigma \qquad\hspace*{2.02cm}\forall t \in [0,T),\\[0.1cm]
      \angle_{\Rzwei} \left(\tau(a,t),\nix^\Sigma\vec \tau(c(a,t))\right) &= \frac{\pi}{2}\qquad \hspace*{1.92cm}\forall t \in[0,T),\\
      \angle_{\Rzwei} \left(\tau(b,t),\nix^\Sigma \vec\tau(c(b,t))\right)& = -\frac{\pi}{2} \qquad \hspace*{1.68cm}\forall t \in[0,T) \end{split}
 \end{align}
%
is called a \emph{solution of the area preserving curve shortening problem with Neumann free boundary conditions}, shortly called \emph{a solution of} (\ref{flow}). Here, $\bar\kappa$ denotes the average of the curvature, $\bar\kappa(t)\defi \frac{\int_a^b \kappa(p,t)\de s}{\int_a^b\de s}$,
where $\de s \defi |\partial_p c(\cdot,t)|\de p$. 
\end{definition}
\begin{remark}
 Since the angles at the end points are oriented we defined the ``outer'' case of the area preserving curve shortening problem with Neumann free boundary conditions, that is, the curves $c(\cdot,t)$ start into $\Rzwei\setminus G_\Sigma$ and meet $\Sigma$ again from the outside. Away from the end points the curves are allowed to intersect $\Sigma$ and can therefore be inside of $G_\Sigma$. 
\end{remark}

\begin{proposition} \label{exis}
The area preserving curve shortening problem with Neumann free boundary conditions has a unique solution on a maximal time interval $[0,T)$. It has regularity 
\begin{align*}
 c \in C^{2+\alpha, 1 +\frac{\alpha}{2}}\left([a,b]\times [0,T),\mathbb R^2\right)\cap C^\infty\left([a,b]\times (0,T),\mathbb R^2\right)
\end{align*}
for any $\alpha \in(0,1)$\footnote{$C^{2+\alpha, 1 +\frac{\alpha}{2}}$ denotes the usual parabolic H\"older spaces.}. 
 If $T<\infty$, we have $\  \max_{[a,b]}|\kappa|(\cdot,t) \to\infty \  \ (t\to T).$
\end{proposition}

\begin{proof}
 The proof works as in \cite[Theorem~7.24]{StahlDiss},
 where the case of curves moving by (\ref{flow}) but without the $\bar\kappa-$term is treated. Stahl first shows in \cite[Chapter~2]{StahlDiss} 
 short time existence by writing the curves as a graph over the initial curve for a short time, where the scalar function in the graph representation satisfies a certain parabolic Neumann equation. We only have to add the corresponding term of $\bar\kappa\nu$ into this equation. 
 Stahl then proved in \cite[Chapter~7]{StahlDiss} gradient estimates on the local graph representation under the condition that the curvature is uniformly bounded. Under this uniform bound, the term $\bar \kappa$ is bounded and it is a uniform Lipschitz term and therefore harmless in these considerations. By standard parabolic theory, we hence get for every $k\in\N$ (away from $t_0=0$) $C^{k+\alpha,\frac{k+\alpha}{2}}$-estimates for the local graph representation which is need in the proof of  $\max_{[a,b]}|\kappa|(\cdot,t) \to\infty, $ $t\to T<\infty$. In \cite[Chapter~5]{MeineDiss}
 we study in detail, how the estimates of Stahl are reproven for our situation. The rest of the proof is done as in \cite[Theorem~7.24]{StahlDiss}.
 \end{proof}

\begin{remark}
 All the results from the following lemma untill Theorem~\ref{Linftyabschaetzung} are based on the Diplom thesis \cite{AchimDipl} of Achim Roth. We sometimes changed his proofs slightly and we rearranged his results for our purpose.
\end{remark}

    \begin{lemma}[Evolution equations] \label{evolution}
     Let $c:[a,b]\times[0,T)\to \Rzwei $ be a solution of (\ref{flow}).
     Then it follows that
      \begin{alignat}{2}
       \partial_t (\de s) &= - \kappa (\kappa-\bar\kappa)\de s & \qquad\qquad\mbox{for }  &t\geq 0, \label{evolution1}\\
        \partial_t \partial_s - \partial_s\partial_t &=  \kappa (\kappa-\bar\kappa)\, \partial_s & \mbox{for } &t\geq 0,\label{evolution2}\\
	\partial_t \tau &= \partial_s \kappa\, \nu & \qquad\qquad\mbox{for }  &t> 0,\label{evolution3}\\
	 \partial_t \nu &= -\partial_s\kappa\,\tau & \qquad\qquad\mbox{for }  &t> 0,\label{evolution4}\\
	  \partial_t \kappa &= \partial^2_s\kappa + \kappa^2(\kappa -\bar\kappa) & \qquad\qquad\mbox{for }  &t> 0.\label{evolution5}
      \end{alignat}
    \end{lemma} 

    \begin{proof}
    This is an easy calculation, see for example \cite[Lemma~2.1]{DziukKuwertSch}.
\end{proof}

    \begin{lemma} \label{shortening}
    The area preserving curve shortening flow shortens the length of the curves, we have $ \frac{\de}{\de t} L(c_t)\leq 0 \quad \forall t\in[0,T).$
    \end{lemma}
      \begin{proof}
	  We use (\ref{evolution1}) and $\int_a^b (\kappa-\bar\kappa) \de s  \equiv 0$ to get 
	  \begin{align*}
	   \frac{\de}{\de t} L(c_t) =  - \int_a^b  \kappa (\kappa-\bar\kappa) \de s = - \int_a^b  (\kappa-\bar\kappa)^2 \de s \leq 0 \quad \forall t\in[0,T).
	  \end{align*}
      \end{proof}


    \begin{definition}
     Let $\Sigma$ be a support curve and let $f:[a,b]\to\Rzwei$ be a curve with $f(a),f(b)\in\Sigma$. Then we call a curve $\gamma:[\tilde a,\tilde b]\to\Sigma\subset\Rzwei$ with $\gamma(\tilde a)=f(a)$ and $\gamma(\tilde b)=f(b)$ a \emph{boundary curve on $\Sigma$ with respect to $f$}.
    \end{definition}

 \begin{lemma} \label{lift}
 Let $\ftilde:\R \to\Sigma$ be the periodic extension of $\fsig$.  
 Define
\begin{align*}
 a_0&\defi \fsig^{-1}(c(a,0)), \\
b_0&\defi \left\{
\begin{array}{ll} 
\fsig^{-1}(c(b,0)) &\mbox{ if } \fsig^{-1}(c(b,0))> a_0 \\ \fsig^{-1}(c(b,0)) + (\Sig b - \Sig a) &\mbox{ if } \fsig^{-1}(c(b,0))\leq a_0.
\end{array}
\right. 
\end{align*}
Then there exist unique maps
\begin{align*}
 a:[0,T)\to \R \mbox{ with } a(0)=a_0 \mbox{ and } \ftilde(a(t)) = c(a,t) \ \forall t\in[0,T),\\
b:[0,T)\to \R \mbox{ with } b(0)=b_0 \mbox{ and } \ftilde(b(t)) = c(b,t)\ \forall t\in[0,T).
\end{align*}
The maps $a$ and $b$ are continuously differentiable on $[0,T)$ and smooth on $(0,T)$.
\end{lemma}

\begin{proof}
 Choose any $\epsilon \in (0,T)$. We consider the paths $c(a,\cdot)|_{[0,T-\epsilon]}$ and $c(b,\cdot)|_{[0,T-\epsilon]}$ in $\Sigma$. We have $\ftilde(a_0)= c(a,0) \in\Sigma$ and $\ftilde(b_0)= c(b,0) \in\Sigma$. By \cite[Chapter~5.6, Proposition~2]{DoCarmo}, we get unique lifts
\begin{align*}
 a_\epsilon: [0,T-\epsilon] \to\R \mbox{ with } a_\epsilon(0)= a_0 \mbox{ and } \ftilde\circ a_\epsilon = c(a,\cdot)|_{[0,T-\epsilon]},\\
b_\epsilon: [0,T-\epsilon] \to\R \mbox{ with } b_\epsilon(0)= b_0 \mbox{ and } \ftilde\circ b_\epsilon = c(b,\cdot)|_{[0,T-\epsilon]}.
\end{align*} 
The map $ a(t) \defi a_\epsilon(t) \ \ \forall t\in [0,T-\epsilon] \ \ \forall \epsilon \in(0,T)$
is well-defined because of the uniqueness of the lift.
The same can be done with $b$. The differentiability follows because $c(a,\cdot)$ and $c(b,\cdot)$ are continuously differentiable in zero and smooth on $(0,T-\epsilon]$.
\end{proof}

\begin{lemma}[Existence of boundary curves] \label{defiboundarycurve}
  The family of curves $\tilde\gamma:[0,1]\times[0,T)\to\Sigma$, defined as
      \begin{align*}
       \tilde\gamma(p,t)\defi\ftilde\left((b(t)-a(t))p + a(t)\right) \ \ \forall (p,t) \in [0,1]\times[0,T),
      \end{align*}
	is a family of smooth boundary curves $\tilde\gamma_t\defi\tilde\gamma(\cdot,t)$ on $\Sigma$ with respect to $c_t$. With respect to $t$, they are continuously differentiable in $t=0$ and smooth for $t>0$. 
	They start at the points $c(a,t)$, follow the support curve $\Sigma$  and reach the points $c(b,t)$. If $a(t)<b(t)$ then $\tildega$ has the same positive orientation as $\ftilde$. If $a(t)=b(t)$, then $\tildega (p,t)= c(a,t)=c(b,t)$ $\forall p\in[0,1]$.
\end{lemma}

\begin{proof}
	The smoothness comes from the smoothness of $\ftilde$ and $c$.
	In the case $a(t)=b(t)$, the curve $\tildegat\equiv  \ftilde \circ a(t)= \ftilde\circ b(t)=c(b,t)= c(a,t)$ is clearly not regular. But in this case we have $c(a,t)=c(b,t)$, what means that $\ct$ is closed. 
\end{proof}

    \begin{definition} \label{orientedVolume} 
      Let $c:\ab\times[0,T)\to\Rzwei$ be a solution of (\ref{flow}).
      Consider  a $C^1$-family of smooth curves $\gamma:[\tilde a,\tilde b]\times [0,T)\to\Sigma$ with $\gamma(\tilde a,t)=c(a,t)$ and $\gamma(\tilde b,t)=c(b,t)$ for all $t\in[0,T)$, i.e.~$\gamma_t\defi\gamma(\cdot,t)$ is a boundary curve on $\Sigma$ with respect to $c_t\defi c(\cdot,t)$. 
      Then for each $ t\in[0,T)$, we call the following expression the \emph{oriented area enclosed by $c_t$ and $\Sigma$:}
      \begin{align} \label{enclosedVolume}
       A(c_t,\gamma_t)\defi \frac{1}{2} \int_{c_t} p^1 \de p^2 - p^2 \de p^1 - \frac{1}{2} \int_{\gamma_t} p^1 \de p^2 - p^2 \de p^1.
      \end{align}
    \end{definition}

    \begin{remark} 
    Let $c_t-\gat$ be the assembled curve of $c_t$ and $-\gat$ (the orientation of $\gat$ is inverted). Then $c_t-\gat$ is immersed closed and the formula indicates the oriented area traced out by $c_t-\gamma_t$. It takes account of the orientation and how often the  pieces of areas are circulated.
      \end{remark}

 \begin{lemma} \label{blob}
      The area preserving curve shortening flow preserves the oriented area enclosed by $c_t$ and $\Sigma$. More precisely,
	\begin{align*}
	 \frac{\de}{\de t} A(c_t,\gamma_t) = - \int_a^b \langle \partial_t c_t,\nu\rangle \de s \equiv 0 \quad \forall t\in[0,T)
	\end{align*}
	and therefore $A(c_t,\gamma_t) \equiv A(c_0,\gamma_0) \eqqcolon A_0$ for all $t\in[0,T)$.
    \end{lemma}

    \begin{proof}
	 An easy calculation shows 
	 \begin{align} \label{bla1}
       A(c_t,\gamma_t)= - \frac{1}{2} \int_a^b \langle c_t,\nu\rangle \de s + \frac{1}{2}\int_{\tilde a}^{\tilde b}\langle \gamma_t,\nu_{\gamma_t}\rangle\de s_{\gamma_t},
      \end{align}
      where $\nu_{\gamma_t}$ denotes the normal of $\gamma_t$ and $\de s_{\gamma_t}\defi |\partial_p \gamma(p,t)|\de p$.
	 We use this to get
	 \begin{align*}
	 \frac{\de}{\de t} A(c_t,\gamma_t) = -\frac{1}{2} \int_a^b \langle \partial_t c_t,J\partial_p c_t\rangle + \langle c_t, J \partial_p\partial_t c_t\rangle \de p + \frac{1}{2} \int_{\tilde a}^{\tilde b} \langle \partial_t\gamma_t,J\partial_p\gamma_t\rangle + \langle \gamma_t,J\partial_p\partial_t\gamma_t\rangle \de p.
	\end{align*}
	Some calculations show
	\begin{align*}
	  \frac{\de}{\de t} A(c_t,\gamma_t) = &- \int_a^b \langle \partial_t c_t,J\partial_p c_t\rangle - \frac{1}{2} \int_a^b \partial_p\langle c_t,J\partial_t c_t\rangle \de p \\
	  &+ \int_{\tilde a}^{\tilde b} \langle \partial_t \gamma_t,J\partial_p \gamma_t\rangle + \frac{1}{2} \int_{\tilde a}^{\tilde b} \partial_p\langle \gamma_t,J\partial_t \gamma_t\rangle \de p. 
	\end{align*}
	  The boundary terms 
	  vanish because the curve $c_t-\gat$ is closed.
	  Moreover, we have $\gamma(p,t)\in\Sigma$ for all $(p,t)$ in $[\tilde a,\tilde b]\times [0,T)$. 
	  Hence $\partial_t\gamma(p,t), \partial_p\gamma (p,t) \in T_{\gamma(p,t)}\Sigma$ for all $(p,t)$ in $[\tilde a,\tilde b]\times [0,T)$. 
	  It follows that $\langle \partial_p \gamma_t,J\partial_t \gamma_t\rangle \equiv 0$ on $[\tilde a,\tilde b]\times [0,T)$, which implies the result.
    \end{proof}

\begin{lemma}\label{lemmastahl5}
 Let $c:[a,b]\times[0,T)\to \Rzwei $ be a solution of (\ref{flow}).
 Then for all $t\in (0,T)$, we have
\begin{align*}
 \partial_s \kappa (a,t) &= \left(\kappa(a,t) - \bar\kappa(t)\right) \nix^\Sigma\kappa\left(\nix^\Sigma f^{-1}(c(a,t))\right)\quad \mbox{ and }\\
 \partial_s \kappa (b,t) &= -\left(\kappa(b,t) - \bar\kappa(t)\right) \nix^\Sigma\kappa\left(\nix^\Sigma f^{-1}(c(b,t))\right).
\end{align*}
\end{lemma}

\begin{proof}
    We consider the vector fields
    \begin{align*}
     \nix^\Sigma \vec\tau \coloneqq \nix^\Sigma\tau\circ \nix^\Sigma f^{-1}:\Sigma\to\Rzwei \mbox{ and }   \nix^\Sigma\vec\nu \coloneqq \nix^\Sigma\nu\circ \nix^\Sigma f^{-1}:\Sigma\to\Rzwei.
    \end{align*}
      They are smooth on the smooth curve $\Sigma$. Therefore, we extend them locally smoothly on an open subset of $\Rzwei$ so that we can differentiate them. 
      We use the boundary condition 
      \begin{align*}
       \langle \nu (p,t),\nix^\Sigma \vec \nu(c(p,t))\rangle= 0 \text{ for } t\in [0,T), p=a \text{ or } p=b,
      \end{align*}
 the evolution equation (\ref{evolution4}) and the equation $\partial_t c = (\kappa-\bar\kappa)\nu$ to get for $p=a$ or $p=b$
      \begin{align} \label{randgl1}
       0=\partial_t \langle \nu, \nix^\Sigma \vec \nu \circ c\rangle= -\partial_s \kappa \langle\tau,\nix^\Sigma \vec\nu\circ c\rangle + (\kappa - \bar\kappa) \langle \nu,(D\nix^\Sigma\vec\nu)\circ c \cdot \nu\rangle, \quad t\in (0,T).
      \end{align}
      The boundary conditions of our problem can be described as  $\tau(a,t)= -\nix^\Sigma \vec\nu(c(a,t))$ and $\tau(b,t)= \nix^\Sigma \vec\nu(c(b,t))$. We use this in (\ref{randgl1}) and get for example in $p=a$
      \begin{align*}
       0= \partial_s \kappa + (\kappa-\bar\kappa) \langle  \nix^\Sigma \vec\tau \circ c,(D\nix^\Sigma\vec\nu)\circ c \cdot \nix^\Sigma \vec\tau \circ c\rangle. 
      \end{align*}
      The expression $\langle  \nix^\Sigma \vec\tau \circ c,(D\nix^\Sigma\vec\nu)\circ c \cdot \nix^\Sigma \vec\tau \circ c\rangle$ at $(a,t)$ is actually $-\nix^\Sigma \kappa$ at $\nix^\Sigma f^{-1} (c(a,t))$ because we have
      \begin{align*}
       \nix^\Sigma \kappa = - \langle \nix^\Sigma \tau,\partial_s \nix^\Sigma\nu\rangle = - \langle \nix^\Sigma \tau, (D \nix^\Sigma \vec\nu)\circ \nix^\Sigma f \cdot \nix^\Sigma \tau\rangle.
      \end{align*}
      The result for $p=b$ follows analogously.
\end{proof}

\begin{proposition}\label{kappageq0}
  Let $c:[a,b]\times[0,T)\to \Rzwei $ be a solution of (\ref{flow}),
  where the initial curve satisfies $\kappa_0\geq 0$ on $[a,b]$. Then the curves $c$ satisfy $\kappa \geq 0$ on $[a,b]\times [0,T)$.
\end{proposition}

\begin{proof}
      We choose an arbitrary $T' \in (0,T)$ and define $D\defi (a,b)\times (0,T')$. For $\delta > 0$ consider the function
      \begin{align*}
       f(p,t)\defi \kappa(p,t) + \delta e^{\alpha t}, \ \ \ (p,t)\in [a,b]\times [0,T),
      \end{align*}
      where $\alpha \geq 0$ is chosen such that $\alpha \geq \max_{[a,b]\times[0,T']}|\kappa(\kappa-\bar\kappa)|$. Clearly $f(\cdot,0)>0$. We show that $f>0$ is true in all of $[a,b]\times[0,T')$. 
      Else, there must be a first time $t_0\in(0,T')$ where $f$ reaches zero. That means, there must be a point $p_0 \in [a,b]$ such that $f(p_0,t_0)=0$ and $f(p,t)>0 $ for $(p,t)\in [a,b]\times [0,t_0)$. 
      Now, we consider the differential operator defined as
      \begin{align*}
      Lu &\defi \partial_s^2 u + \kappa(\kappa-\bar\kappa)u - \partial_t u\\
      & \ = \frac{1}{|\partial_p c|^2} \partial_p^2 u + \left(\frac{1}{|\partial_p c|} \partial_p \frac{1}{|\partial_p c|}\right)\partial_p u + \kappa(\kappa-\bar\kappa) u -\partial_t u
      \end{align*}
      for $u\in C^{2,1}(D)$. This operator is parabolic and it has bounded coefficients. \\
      There are two cases. The first case is $p_0\in (a,b)$. By the evolution equation of the curvature (\ref{evolution5}), we get in $D$
      \begin{align*}
       Lf & =L\kappa + \kappa(\kappa - \bar\kappa)\delta e^{\alpha t} - \alpha \delta e^{\alpha t}\\
	& = \delta e^{\alpha t} \left( \kappa(\kappa - \bar\kappa) - \alpha\right)\leq 0,
      \end{align*}
	where we used the definition of $\alpha$ in the last step. For this situation,
      \begin{align*}
       f(p_0,t_0)&=0, (p_0,t_0) \in D,\\
      Lf &\leq 0 \mbox{ in } (a,b)\times (0,t_0],\\
      f &\geq 0 \mbox{ in } (a,b)\times (0,t_0],
      \end{align*}
      we can use the strong maximum principle to get $f \equiv 0$ in $(a,b)\times (0,t_0]$, see \cite[Chapter~2, Thm.~5]{Friedman}. It follows that $\kappa \equiv -\delta e^{\alpha t} < 0$ on $[a,b]\times[0,t_0]$, a contradiction to $\kappa_0 \geq 0$. 
      Therefore, we have to be in the second case: $p_0 \in \{a,b\}$. We set a ball $B$ with small radius tangentially at $(a,t_0)$ such that $\bar B \subset \bar D$.
      For the situation
      \begin{align*}
       f(a,t_0)& = 0,\\
      Lf &\leq 0  \mbox{ in } D,\\
      f &>0 \mbox{ in } (a,b)\times (0,t_0],    
      \end{align*}
      we use the parabolic Hopf lemma, see \cite[Chapter~3, Theorem~7]{ProtterWeinberger} and get      
      \begin{align*}
       \frac{\partial f}{\partial v_B}(a,t_0) < 0,
      \end{align*}
      where $v_B=(-1,0)$ is the outer unit normal to $\partial B$ in $(a,t_0)$. It follows that $\partial_s f(a,t_0) > 0$. But by Lemma~\ref{lemmastahl5} we have
      \begin{align*}
       \partial_s f (a,t_0)= \partial_s\kappa(a,t_0) = \left(\kappa(a,t_0) - \bar\kappa(t_0)\right) \nix^\Sigma\kappa\left(\nix^\Sigma f^{-1}(c(a,t_0))\right) \leq 0,
      \end{align*}
      where we used $\kappa (a,t_0) = -\delta e^{\alpha t_0}$ and $\kappa (p,t_0) \geq -\delta e^{\alpha t_0}$ on $(a,b)$. This is a contradiction. 
      The argumention for $b$ follows analogously. Hence we have shown $f = \kappa + \delta e^{\alpha t}>0$ in $[a,b]\times[0,T')$. Since $\delta>0$ and $T' \in (0,T)$ were arbitrary, the result follows.
\end{proof}

\begin{corollary} \label{kappagroesser0}
 Let $c:[a,b]\times[0,T)\to \Rzwei $ be a solution of (\ref{flow}),
 where the initial curve satisfies $\kappa_0\geq 0$ on $[a,b]$. Then the curves $c$ satisfy $\kappa > 0$ on $[a,b]\times (0,T)$.
\end{corollary}

\begin{proof}
      We slightly vary the proof of Proposition~\ref{kappageq0}. The notation is the same as in that proof. We assume that there is a time $t_0 \in (0,T')$ where $\kappa$ vanishes, i.e.~$\kappa (p_0,t_0)=0$ for a $p_0 \in [a,b]$ (notice that we have used Proposition~\ref{kappageq0}). \\
      We first consider the case $p_0\in (a,b)$. The evolution equation (\ref{evolution5}) yields $ L \kappa \equiv 0$ in $D$. 
      By the strong parabolic maximum principle \cite[Chapter~2, Thm.~5]{Friedman}, we get $\kappa \equiv 0$ in $(a,b)\times (0,t_0]$. Continuity implies $\kappa_0 \equiv 0$ on $[a,b]$. 
      But this yields a contradiction because of the boundary conditions of the initial curve and the convexity of the domain traced out by $\Sigma$.
      Hence there are no ``inner points`` in $D$ with $\kappa=0$.\\
      The case $p_0 \in\{a,b\}$ is again treated with the parabolic Hopf lemma, see \cite[Chapter~3, Theorem~7]{ProtterWeinberger}. Set a small ball in $\bar D$ tangentially at $(a,t_0)$ and get
      \begin{align*}
       \frac{\partial \kappa}{\partial v}(a,t_0) < 0,
      \end{align*}
      where $v=(-1,0)$ is the outer unit normal to $\partial B$ in $(a,t_0)$. This yields $\partial_s \kappa (a,t_0) > 0$, which is a contradiction to   
        \begin{align*} 
       \partial_s\kappa(a,t_0)= \left(\kappa(a,t_0) - \bar\kappa(t_0)\right) \nix^\Sigma\kappa\left(\nix^\Sigma f^{-1}(c(a,t_0))\right) \leq 0
      \end{align*}
      cf. Lemma~\ref{lemmastahl5}. 
	The result in the case $p_0=b$ follows analogously.
\end{proof}

\section{A bound on the average of the curvature}  \label{S2}

\begin{proposition} \label{anfangstheorem} 
Let $c:[a,b]\times[0,T)\to \Rzwei $ be a solution of (\ref{flow}),
where the initial curve satisfies $\kappa_0\geq 0$ on $[a,b]$. Then we have  $\int_a^b \kappa \de s \geq \pi$ for all $t\in[0,T)$. If there is a time $t_0\in\nut$ with $c(a,t_0)=c(b,t_0)$, then $\int_a^b \kappa \de s|_{t=t_0} > \pi$.
\end{proposition}

\begin{proof}
      We prove the statements under the condition $\kappa_0>0$. Choose $t\in[0,T)$. 
      By Corollary~\ref{kappagroesser0} the curvature is positive: $\kappa>0$ on $[a,b]\times[0,T)$. W.l.o.g., we have
      \begin{align*}
       c(a,t) = 0 \in \Sigma\ \  \mbox{ and } \ \ \frac{c(b,t)-c(a,t)}{|c(b,t)-c(a,t)|} = - e_1.
      \end{align*}
      We also assume $c(a,t)\neq c(b,t)$ and treat the other case afterwards. Since we only consider the curve at a fixed time, we leave out the time dependence in the notation in the following. Because of the convexity and the positive orientation of $\Sigma$, we have 
      \begin{align*}
       \langle \nix^\Sigma f(p)-\nix^\Sigma f(p_0),\nix^\Sigma \nu(p_0)\rangle \geq 0 \mbox{ for } p,p_0 \in [\nix^\Sigma a,\nix^\Sigma b].
      \end{align*}
      We use this to get $\langle c(b)-c(a),\nix^\Sigma \vec \nu (c(a))\rangle \geq0$ and thus 
      \begin{align*}
       \langle \nix^\Sigma \vec \tau(c(a)),e_2\rangle \geq 0 \quad \mbox{ and } \quad\langle \nix^\Sigma \vec \tau(c(b)),e_2\rangle \leq 0 \mbox{ respectively}. 
      \end{align*}
      The boundary conditions imply
      \begin{align*}
       \langle \tau(a),e_1\rangle \geq 0 \ \ \mbox{ and }\ \ \langle \tau(b),e_1\rangle \geq0.
      \end{align*}
      In addition with $\kappa(a)>0$ and $\kappa(b)>0$, this yields the following statements:
      \begin{align}\label{enum3}
      \begin{split}
      & \mbox{If } \tau(a) \neq e_2 \mbox{ then there is a } \delta>0  \mbox{ such that } c( (a,a+\delta))\subset \{x\in\Rzwei: x^1 > 0\},\\
      & \mbox{If } \tau(b) \neq - e_2 \mbox{ then there is a } \tilde\delta>0  \mbox{ such that } c( (b-\tilde\delta,b))\subset \{x\in\Rzwei: x^1 < c^1(b)\}.
      \end{split}
      \end{align}
%
      As the function $c^1$ is continuous and $[a,b]$ is compact we get two points $p_1,p_2 \in [a,b]$ such that
      \begin{align*}
       c^1(p_1)=\inf\{c^1(p): p\in[a,b]\} \mbox{ and } c^1(p_2)=\sup\{c^1(p): p\in[a,b]\}.
      \end{align*}
      Therefore, we have $c^1(p_1) \leq c^1(b)<c^1(a)\leq c^1(p_2)$. It follows that $p_1 \neq a$, $p_2\neq b$ and $p_1\neq p_2$. From (\ref{enum3}) we obtain 
      \begin{align*}
	\mbox{if } p_1=b \mbox{ then } \tau(b)= -e_2 \ \ \mbox{ and if } p_2=a \mbox{ then } \tau(a)= e_2.
      \end{align*}
      In general, the situation is similar. We get
      \begin{align*}
       \partial_p c^1(p_1) = \lim\limits_{h\searrow 0} \frac{c^1(p_1 + h)-c^1(p_1)}{h} \geq 0\ \text{ and }\
      \partial_p c^1(p_1) = \lim\limits_{h\searrow 0} \frac{c^1(p_1)-c^1(p_1-h)}{h} \leq 0
      \end{align*}
      and therefore $\tau^1(p_1)=0$. Thus, we have $\tau(p_1)\parallel e_2$ (and $\tau(p_2)\parallel e_2$ respectively). 
      We specify this to 
      \begin{align} \label{tangenten}
       \tau(p_1)=-e_2 \ \ \ \text{and} \ \ \ \tau(p_2)=e_2.
      \end{align}
     Assume $\nu(p_1) = -e_1$. We have $ 0<\kappa(p_1) |\pp c(p_1)|^2 = \langle \partial_p^2c(p_1),\nu(p_1)\rangle = - \partial_p^2 c^1(p_1).$
      It follows that $\partial_p^2 c^1(p_1) < 0$. We have seen $\partial_p c^1(p_1) = \tau^1 (p_1) |\pp c(p_1)|=0$, thus the function $c^1$ has a local maximum at $p_1$. 
      Near $p_1$, we then have $c^1< c^1(p_1)$, but this is a contradiction to the definition of $p_1$. We have shown $\nu(p_1)= e_1$ and hence $\tau(p_1) = -e_2$. For $p_2$, the statement follows analogously. \\
      
      The geometric situation helps us to compute $\int\kappa \de s$. The map $\tau:[a,b]\to\mathbb S^1, \ \tau(p)=\partial_s c(p)$ is continuously differentiable, and by the Frenet equation $\partial_s \tau = \kappa \nu$ we get
      \begin{align*}
       \partial_p\tau=\kappa|\partial_p\tau|\nu \Rightarrow |\partial_p\tau|=\kappa|\partial_p c|\Rightarrow \int|\partial_p\tau| \de p = \int\kappa|\partial_p c|\de p = \int\kappa\de s.
      \end{align*}
      Since $\tau$ is continuous and due to (\ref{tangenten}) the function $\tau$ must run through at least a half circle which has length $\pi$. We have two cases:\\
      Case $p_1<p_2$:
      \begin{align*}
       \pi\leq L\left(\tau|_{[p_1,p_2]}\right) = \int_{p_1}^{p_2 } |\partial_p \tau| \de p = \int_{p_1}^{p_2}\kappa \de s < \int_a^b \kappa \de s \quad(p_1\neq a,p_2\neq b, \kappa > 0).
      \end{align*}
      Case $p_2<p_1$:
      \begin{align*}
       \pi \leq L\left(\tau|_{[p_2,p_1]}\right) = \int_{p_2}^{p_1}\kappa \de s \leq \int_a^b \kappa\de s \quad (p_2\geq a, p_1\leq b, \kappa > 0).
      \end{align*}
      The case  $c(a,t_0) =c(b,t_0)$ is proven similarly. After rotation and translation we have at the time~$t_0$
      \begin{align*}
       c(a)=c(b)=0, \quad \nix^\Sigma\vec\tau(0) = -e_1.
      \end{align*}
       By the boundary conditions we immediately get $\tau(a)= e_2$, $\tau(b)= -e_2$. We repeat the rest of the proof above and get  $\intab\kappa\de s|_{t= t_0}>\pi$ by carefully considering the situation.
\end{proof}

\begin{definition}
 Let $f:[a,b]\to\Rzwei$ be a smooth, regular and closed curve. The number
\begin{align*}
\ind(f)\defi n(\partial_pf,0) \in \mathbb Z
\end{align*}
is called the \emph{index} (or \emph{turning number}) of $f$. Here, $n(\partial_p f,0)$ denotes the winding number of the curve $\partial_p f:[a,b]\to\Rzwei$ with respect to $0\in\Rzwei$.
\end{definition}

\begin{theorem} \label{indexformel}
 Let $f$ be a piecewise smooth, regular and closed curve, defined on intervals $[a_j,b_j]$, $j=1,\dots,k$, and with exterior angles $\alpha_j$, $j=1,\dots,k$. Then
\begin{align*}
 \ind(f)=\frac{1}{2\pi}\sum_{j=0}^k\int_{a_j}^{b_j}\kappa_f\de s_f + \frac{1}{2\pi} \sum_{j=0}^k\alpha_j  \ \ \in\mathbb Z.
\end{align*}
\end{theorem}

\begin{proof}
 See \cite[Theorem 2.1.6]{Klingenberg}. 
\end{proof}

\begin{theorem} \label{umlaufsatz}
 Let $f$ be a piecewise smooth, regular and simple closed curve with exterior angles in $(-\pi,\pi)$. Then the index of $f$ is $\ind(f)=\pm 1$, where the sign depends on the orientation of $f$.
\end{theorem}

\begin{proof}
 See \cite[Theorem 2.2.1]{Klingenberg} and \cite[Chapter 5.7, Theorem 2]{DoCarmo}. 
\end{proof}

\begin{corollary} \label{indexc0} Let $\tildega_t\defi\tildega(\cdot,t)$ be the boundary curves from Lemma~\ref{defiboundarycurve} and $\ftilde$ the periodic extension of $\fsig$.
 For all $t\in[0,T)$ with $a(t)<b(t)$, we have the formula
\begin{align*}
 \ind(c_t-\tildega_t)=\frac{1}{2\pi}\int_a^b\kappa\de s - \frac{1}{2\pi} \int_{a(t)}^{b(t)}\Sig\tilde\kappa \de s_{\ftilde} + \frac{1}{2} = \frac{1}{2\pi}\int_a^b\kappa\de s - \frac{1}{2\pi} \int_{0}^{1}\kappa_{\tildega_t} \de s_{\tildega_t} + \frac{1}{2}.
\end{align*}
\end{corollary}

\begin{proof}
      For $t\in[0,T)$ with $a(t)<b(t)$, the boundary curves $\tildega$ are regular. It is easy to show that the assembled curve $c_t-\tildega_t$ is a piecewise smooth, regular, closed curve with exterior angles $\alpha_1,\alpha_2=\frac{\pi}{2}$. We use Theorem~\ref{indexformel} for these curves and get the result.
\end{proof}

\begin{theorem} \label{satz12}
  Let $c:[a,b]\times[0,T)\to \Rzwei $ be a solution of (\ref{flow}),
  where the initial curve satisfies the following three conditions: 
  The curvature $\kappa_0$ is positive on $[a,b]$, the curve $c_0$ has no self-intersection and it is contained in the outer domain created by the convex support curve $\Sigma$. Then it follows that
\begin{enumerate}
 \item $a(t)<b(t)$ $\forall t\in[0,T)$,\label{aleqb}
 \item $\ind(c_t-\tildega_t)=1$ $\forall t\in[0,T)$.\label{gl2} 
%
\end{enumerate}
\end{theorem}

\begin{proof}
	The conditions on the initial curve ensure that $c_0-\tildega_0$ is simple closed, piecewise smooth, regular and has exterior angles in $(-\pi,\pi)$. By Theorem~\ref{umlaufsatz}, the index of $\tilde c_0$ must be $1$ or $-1$. If it was $-1$, then by Corollary \ref{indexc0}
	\begin{align*}
	 \int_a^b\kappa\de s\big|_{t=0}=-3\pi +  \int_{a(0)}^{b(0)}\Sig\tilde\kappa \de s_{\ftilde}.
	\end{align*}
	 We estimate the last term in the following way:
	\begin{align*}
	  \int_{a(0)}^{b(0)}\Sig\tilde\kappa& \de s_{\ftilde} \leq  \int_{a(0)}^{a(0) + (\Sig b - \Sig a)}\Sig\tilde\kappa \de s_{\ftilde} 
		= 2\pi,
	\end{align*}
	where we used the definition of $a(0),b(0)$, Theorem~\ref{umlaufsatz} and the positive orientation of $\fsig$.
	We therefore have $\int_a^b\kappa\de s\big|_{t=0}\leq-\pi$, which contradicts $\kappa_0>0$. It follows that we are in the case $\ind(c_0-\tildega_0) = 1$.  We know $a(0)<b(0)$ and assume that there is a time $t_0\in(0,T)$ such that $a(t_0)=b(t_0)$ and $a(t)<b(t)$ for $t\in[0,t_0)$. By Corollary~\ref{indexc0}, we have
	\begin{align*}
	 \ind( c_t-\tildega_t)=\frac{1}{2\pi}\int_a^b\kappa\de s - \frac{1}{2\pi} \int_{a(t)}^{b(t)}\Sig\tilde\kappa \de s_{\ftilde} + \frac{1}{2}\qquad \forall t\in[0,t_0).
	\end{align*}
	The left hand side is a number in $\mathbb Z$ and the right hand side is continuous in $t$. It follows that $\ind(c_t-\tildega_t)\equiv 1$ for all $t$ in $[0,t_0)$. This is equivalent to
	\begin{align*}
	 \int_a^b\kappa\de s = \pi + \int_{a(t)}^{b(t)}\Sig\tilde\kappa \de s_{\ftilde}  \ \ \forall t\in[0,t_0).
	\end{align*}
	We use  continuity to get $ \int_a^b\kappa\de s\big|_{t= t_0} = \pi.$
	This is a contradiction to Proposition~\ref{anfangstheorem}, because $a(t_0)=b(t_0)$ implies $c(a,t_0)=c(b,t_0)$. The statement of (\ref{gl2}) follows from (\ref{aleqb}), Corollary~\ref{indexc0} and the continuity like above.
\end{proof}

\begin{theorem} \label{mitdSig}
    Let $c:[a,b]\times[0,T)\to \Rzwei $ be a solution of (\ref{flow}), where
    the initial curve satisfies the following four conditions: 
    The curvature $\kappa_0$ is positive on $[a,b]$, the curve $c_0$ has no self-intersection, it is contained in the outer domain created by the convex support curve $\Sigma$ and it satisfies $L(c_0)<\Sig d$, where $\Sig d$ is defined as $$\Sig d\defi\min\{|x-y|: x,y \in\Sigma,\Sig\vec\tau(x)=-\Sig\vec\tau(y)\}.$$ 
    Then it follows that $\int_a^b\kappa\de s< 2 \pi$ $\ \forall t\in[0,T)$.
\end{theorem}

\begin{proof}
    We first show
    \begin{align} \label{Lleq}
    L(\tilde \gamma_0)\leq L(c_0),
    \end{align}
    where $\tilde \gamma_0\defi \tildega(\cdot,0)$ is the boundary curve from Lemma~\ref{defiboundarycurve}. Let $G\subset\Rzwei$ be the outer domain created by $\Sigma$. 
    Since $\Sigma$ is convex and $\fsig$ is injective we have: For every $x\in G\cup \Sigma$ there is a unique $p_x \in [\Sig a,\Sig b)$ such that $|x-\fsig(p_x)|= \dist (x,\Sigma)$. We consider the continuously differentiable vector field 
    \begin{align*}
     X: G\cup \Sigma \to \Rzwei, \ \ \ \ X(x)\defi -\Sig\nu(p_x).
    \end{align*}
    For each $x\in G$, we choose $\{\Sig\tau(p_x),\Sig\nu(p_x)\}$ as a basis for $\Rzwei$ and set
    \begin{align*}
     X= X^1 \Sig\tau(p_x) + X^2 \Sig\nu(p_x)
    \end{align*}
    on a neighbourhood of $x$. Thus, we obtain
    \begin{align*}
     \Div X(x)=\frac{\partial X^1}{\partial \Sig\tau(p_x)}(x) + \frac{\partial X^2}{\partial \Sig\nu(p_x)}(x).
    \end{align*}
    These terms can be computed in the following way: We set $\lambda_x\defi \frac{1}{1+ \text{dist}(p,\Sigma)\Sig\kappa(p_x)}$ and consider the smooth curves $\alpha_1,\alpha_2:[0,\epsilon)\to\Rzwei$ for a small $\epsilon>0$,
    \begin{align*}
     \alpha_1(t)&\defi \fsig\left(p_x + \tfrac{1}{\lambda_x}t\right) - \dist(x,\Sigma) \Sig \nu \left(p_x + \tfrac{1}{\lambda_x}t\right)\\
      \alpha_2(t) &\defi x + t\Sig\nu(p_x).
    \end{align*}
    We have
    \begin{align*}
     \alpha_1(0) &=x = \alpha_2(0),\\
      \frac{\de}{\de t} \alpha_1(t)\big|_{t=0} &= \tfrac{1}{\lambda_x} \Sig\tau(p_x) +  \tfrac{1}{\lambda_x}\dist(x,\Sigma) \Sig\kappa(p_x)\Sig\tau(p_x)  = \Sig\tau(p_x),\\
      \frac{\de}{\de t} \alpha_2(t)\big|_{t=0} &= \Sig\nu(p_x),\\
      X(\alpha_1(t))&=-\Sig\nu\left(p_x + \tfrac{1}{\lambda_x}t\right),\\
      X(\alpha_2(t))&=-\Sig\nu\left(p_x\right)
    \end{align*}
     and calculate
    \begin{align*}
     \frac{\partial X^1}{\partial \Sig\tau(p_x)}(x) &=  \frac{\de}{\de t} X^1(\alpha_1(t))\big|_{t=0} 
       = \langle \frac{\de}{\de t} X(\alpha_1(t))\big|_{t=0},\Sig\tau(p_x)\rangle \\ 
      &= \langle \tfrac{1}{\lambda_x} \Sig\kappa(p_x) \Sig\tau(p_x),\Sig\tau(p_x)\rangle = \tfrac{1}{\lambda_x}\Sig\kappa(p_x)\geq 0,\\
     \frac{\partial X^2}{\partial \Sig\nu(p_x)}(x) &= \frac{\de}{\de t} X^2(\alpha_2(t))\big|_{t=0}  = \langle \frac{\de}{\de t} X(\alpha_2(t))\big|_{t=0},\Sig\nu(p_x)\rangle =0.
    \end{align*}
    This shows $\Div X \geq 0$ in $G$, what we use for the divergence theorem. By assumptions, the curve $c_0-\tildega_0$ is simple closed and traces out a domain $U\subset\Rzwei$. By $a_0<b_0$, we know that $c_0-\tildega_0$ bounds $U$ positively. Since  $-\tildega_0$ has the opposite orientation compared to $\fsig$, we have $U\subset G$. Let $\vec \nu_U:\partial U\setminus\{c_0(a),c_0(b)\}\to\Rzwei$ be the outer normal to $U$. We then have
    \begin{align*}
     0 &\leq \int_U \Div X \de \mathcal L = \int_{\partial U} \langle X, \vec \nu_U \rangle \de \mathcal H^1                      \\
    &= \int_a^{b}\langle X\circ c_0, \vec \nu_U \circ c_0\rangle \de s_{c_0} + \int_0^{1} \langle X\circ \tildega_0,\vec \nu_U \circ  \tildega_0\rangle \de s_{ \tildega_0}.
    \end{align*}
    By construction, we have the identity $X\circ  \tildega_0= - \vec \nu_U\circ  \tildega_0$ on $[0,1]$. It follows that
    \begin{align*}
     L(\tilde \gamma_0) = \int_0^{1}\de s_{ \tildega_0} \leq \int_a^b \langle X\circ c_0, \vec \nu_U \circ c_0\rangle \de s_{c_0} \leq \int_a^b \de s_{c_0} =L(c_0).
    \end{align*}
    We use (\ref{Lleq}) to get 
    \begin{align*}
     |c(a,0)-\tilde\gamma_0(p)|\leq L\left(\tilde\gamma_0\big|_{[0,p]}\right)\leq L(\tilde\gamma_0)<\Sig d \ \ \ \forall p\in[0,1].
    \end{align*}
    By definition of $\Sig d$, it follows that $\Sig\vec\tau(c(a,0))\not = -\Sig\vec\tau(\tilde\gamma_0(p))\ \forall p\in[0,1]$. By continuity we get
    \begin{align*} 
     \winkel\left(\Sig\vec\tau(c(a,0)),\Sig\vec\tau(\tilde\gamma_0(p)) \right) \in[0,\pi)\ \ \forall p\in[0,1], 
    \end{align*}
      in particular 
          \begin{align}\label{winkel0}
     \winkel\left(\Sig\vec\tau(c(a,0)),\Sig\vec\tau(c(b,0)) \right) \in[0,\pi). 
    \end{align}
    For $t\in (0,T)$, we have
\begin{align*}
 |c(b,t)-c(a,t)|\leq L(\ct)\leq L_0 <\Sig d.
\end{align*}
It follows that
\begin{align*}
 \Sig\vec\tau(c(a,t))\not = -\Sig\vec\tau(c(b,t)),
\end{align*}
  which implies by continuity and (\ref{winkel0})
\begin{align} \label{wink}
 \winkel\left(\Sig\vec\tau(c(a,t)),\Sig\vec\tau(c(b,t)) \right) \in [0,\pi).
\end{align}
Theorem~\ref{satz12} (\ref{aleqb}) implies that the curves $\tildega_t$ are regular and positively oriented. This yields together with the Frenet equation
\begin{align*}
  \winkel\left(\Sig\vec\tau(c(a,t)),\Sig\vec\tau(\tilde\gamma_t(p)) \right) = L\left(\Sig\vec\tau(\tilde\gamma_t)\big|_{[0,p]}\right) = \int_0^p \kappa_{\tildega_t}\de s_{\tildega_t}
\end{align*}
for $p\in[0,1]$, $t\in(0,T)$.
By Theorem~\ref{satz12} (\ref{gl2}), Corollary~\ref{indexc0} and (\ref{wink}) we have
\begin{align*}
 \int_a^b\kappa\de s = \pi + \int_0^{1} \kappa_{\tildegat}\de s_{\tildegat} < 2\pi.
\end{align*}
\end{proof}

\begin{lemma} \label{cinD}
 Consider a solution $c:[a,b]\times[0,T)\to \Rzwei $ of (\ref{flow}).
 Then we have
\begin{align*}
 c(p,t)\in D\defi\{x\in\Rzwei: \dist(x,\Sigma)\leq \tfrac{L_0}{2}\}\quad \forall (p,t)\in[a,b]\times[0,T),
\end{align*}
where $L_0$ denotes the length of the initial curve.
\end{lemma}

\begin{proof}
    By the boundary conditions $c(a,t)\in D$ and $c(b,t)\in D$ is trivial. If there is a point $(p,t)\in(a,b)\times\nut$ such that $\dist(c(p,t),\Sigma)> \frac{L_0}{2}$ then
    \begin{align*}
     |c(p,t)-c(a,t)|>\frac{L_0}{2} \text{ and } |c(p,t)-c(b,t)|>\frac{L_0}{2}.
    \end{align*}
    This gives a contradiction via
      \begin{align*}
       L_0 < |c(p,t)-c(a,t)| + |c(p,t)-c(b,t)|\leq L\left(\ct\big|_{[a,p]}\right) + L\left(\ct\big|_{[p,b]}\right) \leq L(\ct)\leq L_0.
      \end{align*}
\end{proof}

\begin{proposition}\label{Laengenachuntenabsch}
    Let $c:[a,b]\times[0,T)\to \Rzwei $ be a solution of (\ref{flow}),
    where the initial curve $c_0$ satisfies the following three conditions: 
    The curvature $\kappa_0$ is positive on $[a,b]$, the curve $c_0$ has no self-intersection and it is contained in the outer domain created by the convex support curve~$\Sigma$. Then it follows that
    \begin{align*}
     \frac{4 A_0}{L_0 + 2\diam\Sigma}\leq L(\ct) \ \ \ \forall t\in[0,T),
    \end{align*}
    where $L_0= L(c_0)$ and $A_0= A(c_0,\tilde\gamma_0)$.
\end{proposition}

\begin{proof}
 Since the area is preserved and by the formula for the oriented area (\ref{bla1}), we get
    \begin{align*}
      A_0= A(c_t,\tilde\gamma_t)= - \frac{1}{2} \int_a^b \langle c_t,J\pp\ct\rangle \de p + \frac{1}{2}\int_{0}^{1}\langle \tilde\gamma_t,J\pp\tildegat\rangle\de p.
    \end{align*}
    By translation invariance of the oriented area we can assume that the origin $0\in\Rzwei$ is in $G_\Sigma$. As $\tildega_t$ is positively oriented we have that $\nu_{\tildega_t}$ is the inner normal of $G_\Sigma$. From $0\in G_\Sigma$ we get $\langle\tildega_t,\nu_{\tildega_t}\rangle\leq 0$ on $[0,1]$ for all $t\in\nut$. This implies
    \begin{align*}
     A_0\leq - \frac{1}{2} \int_a^b \langle c_t,J\pp\ct\rangle \de p \leq\frac{1}{2} \int_a^b | c_t||J\pp\ct| \de p =\frac{1}{2} \int_a^b | c_t||\pp\ct| \de p.
    \end{align*}
   For any $(p,t) \in[a,b]\times\nut$ there is an $x\in\Sigma$ such that $\dist(c(p,t),\Sigma)= |c(p,t) - x|$. We estimate with the previous lemma
    \begin{align*}
     |c(p,t)|\leq|c(p,t)-x| + |x|\leq \frac{L_0}{2} + \diam\Sigma,
    \end{align*}		
    which yields $A_0 \leq \frac{1}{2}\left(\frac{L_0}{2} + \diam\Sigma\right) \intab|\pp\ct|\de p = \frac{L_0 + 2 \diam\Sigma}{4} L(c_t).$
\end{proof}

\begin{theorem} \label{Linftyabschaetzung}
   Let $c:[a,b]\times[0,T)\to \Rzwei $ be a solution of the area preserving curve shortening problem with Neumann free boundary conditions, where
   the initial curve $c_0$ satisfies the following four conditions: 
    The curvature $\kappa_0$ is positive on $[a,b]$, the curve $c_0$ has no self-intersection, it is contained in the outer domain created by the convex support curve $\Sigma$ and $L_0<\Sig d$ with $\Sig d\defi \min\{|x-y|: x,y\in\Sigma,\Sig\vec\tau(x)=-\Sig\vec\tau(y)\}$. We then have
    \begin{align*}
       \frac{\pi}{L_0}\leq \bar\kappa(t) \leq\frac{(L_0 + 2\diam\Sigma)\pi}{2 A_0} \ \ \ \forall t\in[0,T),
      \end{align*}
     where again the notation $L_0= L(c_0)$, $A_0=A(c_0,\tilde\gamma_0)$ is used.
    
\end{theorem}

\begin{proof}
The estimate from below is the curve shortening property and Proposition~\ref{anfangstheorem}. 
  Combine Theorem~\ref{mitdSig} and Proposition~\ref{Laengenachuntenabsch} for the other inequality.
\end{proof}

\section{Finite type I singularities} \label{S3}

\begin{proposition} \label{Amax}
 Let $c_0$ be an initial curve with $\kappa_0\geq 0$ on $[a,b]$. Consider the solution $c:[a,b]\times[0,T)\to \Rzwei $ of (\ref{flow}),
 where $T<\infty$ and where the curvature is not bounded as $t \nearrow T$. Then we have
\begin{align*}
 \max_{[a,b]} \kappa^2(\cdot,t)\geq \frac{1}{2(T-t)} \ \ \ \forall t\in (0,T).
\end{align*}
\end{proposition}

\begin{proof}
      The proof is similar to the proof of $\max |A|^2(\cdot,t)\geq \frac{1}{2(T-t)}$ for the mean curvature flow of compact hypersurfaces in $\R^{n+1}$, see \cite[Lemma 1.2]{Huisken90}, but it has to be adapted to the situation with boundary.
      The evolution equation (\ref{evolution5}) implies for $t>0$
      \begin{align*}
       \pt \kappa^2 &= 2\kappa\pt\kappa = 2\kappa(\ps^2\kappa + \kappa^2(\kappa-\bar\kappa))\\
	&=\ps^2\kappa^2-2(\ps\kappa)^2 + 2\kappa^4 - 2\kappa^3\bar\kappa\\
	&\leq \ps^2\kappa^2 + 2\kappa^4,
      \end{align*}
      using that $\kappa \geq 0$ in $[a,b]\times [0,T)$, see Proposition~\ref{kappageq0}. The function $\kappa^2$ is in $C^1\left([a,b]\times(0,T)\right)$, therefore $\kappa^2_{\text{max}}$ is locally Lipschitz in $(0,T)$. At a differentiable time $t$ we have $\frac{\de}{\de t}\kappa^2_{\text{max}} (t) = \frac{\partial \kappa^2 (p,t)}{\partial t},$
      where $p\in[a,b]$ is a point at which the maximum is attained, see \cite{HamTrick}.
      It follows that      
      \begin{align}\begin{split}\label{re24}
       \frac{\de}{\de t}\kappa_{\text{max}}^2 (t) &= \frac{\partial \kappa^2(p,t)}{\partial t} \\
	&\leq \ps^2\kappa^2(p,t) + 2(\kappa_{\text{max}}^2(t))^2,\end{split}
      \end{align}
      where $p\in[a,b]$ is a point at which the maximum is attained. We prove $\ps^2\kappa^2 (p,t)\leq 0$. For that, we have to consider two cases.
      If $p\in (a,b)$, then $\ps^2\kappa^2 (p,t)\leq 0$ is immediate since we have a maximum in the inner part of $[a,b]$. But if $p\in\{a,b\}$, we only can conclude the following inequality for the first derivative: if $p=a$ we have
      \begin{align}\label{kroko3}
       \ps \kappa^2(a,t)\leq0,
      \end{align}
      and if $p=b$ we have $\ps \kappa^2(b,t)\geq0$. We treat the situation $p=a$. As $\kappa\geq 0$, since the maximum is attained in $a$ and by the geometric situation we have $\kappa(a,t)>0$ (otherwise we would have $\kappa\equiv 0$ in all of $[a,b]$ - and this is not possible). From (\ref{kroko3}) we then have $\ps\kappa(a,t)\leq 0$, and by Lemma~\ref{lemmastahl5} we get 
      \begin{align*}
       0\geq\ps \kappa (a,t)=(\kappa(a,t)-\bar\kappa(t))\nix^\Sigma\kappa\left(\nix^\Sigma f^{-1}(c(a,t))\right)\geq 0,
      \end{align*}
      because $\nix^\Sigma\kappa\geq0$ and $\bar\kappa(t)\leq\max_{[a,b]}\kappa(\cdot,t)=\kappa(a,t)$. It follows that $\ps\kappa(a,t)=0$ and, of course, $\ps\kappa^2(a,t)=0$. 
      Now $\ps^2\kappa^2(a,t)>0$  would imply the existence of a strict local minimum in $(a,t)$. But since there is a maximum in $(a,t)$, we obtain $\ps^2\kappa^2(a,t)\leq0$.
      In the case $p=b$ the same argument can be done because in the formula for $\ps\kappa$ at the boundary, there is a minus for $p=b$, see Lemma~\ref{lemmastahl5}. We have proven $ \ps^2\kappa^2(p,t)\leq0.$ 
      Together with (\ref{re24}), this implies
      \begin{align*}
       \frac{\de}{\de t}\kappa_{\text{max}}^2 (t) \leq 2(\kappa_{\text{max}}^2(t))^2.
      \end{align*}
      We compute $ -\frac{\de}{\de t}\left(\frac{1}{\kappa_{\text{max}}^2(t)}\right) \leq 2$
      and get for $0<t_1<t_2<T$
      \begin{align*}
       \frac{1}{\kappa_{\text{max}}^2(t_1)} \leq 2(t_2-t_1) + \frac{1}{\kappa_{\text{max}}^2(t_2)}.
      \end{align*}
      Since $\kappa$ is not bounded as $t\nearrow T$, one can find a sequence $t_i\nearrow T$ such that $\kappa_{\text{max}}^2(t_i)\to\infty$ as $i\to\infty$. For $t\in (0,T)$ we now have $ \frac{1}{\kappa_{\text{max}}^2(t)} \leq 2(T-t),$
      which yields the result.
  \end{proof}

\begin{definition}
 Let $c:[a,b]\times[0,T)\to \Rzwei $ be a solution of (\ref{flow}).
 We say that $c$ develops a \emph{singularity at $T\leq \infty$} if $\max_{p\in [a,b]}|\kappa|(p,t) \to \infty$ for $t\nearrow T$.
\end{definition}

\begin{definition}
 Let $c:[a,b]\times[0,T)\to \Rzwei $ be a solution of (\ref{flow})
 with singularity at $T< \infty$. Then we differentiate between two types of singularities. The singularity is of \emph{type I} if there is a uniform constant $\bar c_0>0$ such that
\begin{align*}
 \sup_{p\in [a,b]}\kappa^2(p,t) \leq \frac{\bar c_0}{2(T-t)} \quad \forall t\in [0,T).
\end{align*}
If there is no such constant then we call the singularity a \emph{type II singularity}.
\end{definition}

\begin{definition}\label{Parabolicblowuppoint}
Let $c:[a,b]\times[0,T)\to \Rzwei $ be a solution of (\ref{flow})
with singularity at $T\leq \infty$.
 A point $x_0\in\Rzwei$ is called \emph{blowup point} of $c$, if there is a sequence of points $(p_j,t_j) \in[a,b]\times[0,T)$ such that 
\begin{align*}
 &t_j\nearrow T \ (j\to\infty),\\
&p_j\to p_0\in[a,b]\ (j\to\infty),\\
&Q_j\defi|\kappa(p_j,t_j)|=\max_{p\in[a,b]}|\kappa(p,t_j)|\to\infty\ (j\to\infty),\\
&c(p_j,t_j)\to x_0\ (j\to\infty).
\end{align*}
\end{definition}

\begin{lemma}\label{existenceblowup}
 Let $c:[a,b]\times[0,T)\to \Rzwei $ be a solution of (\ref{flow})
 with singularity at $T\leq \infty$. Then there exists a blowup point $x_0\in D\defi\{x\in\Rzwei: \dist(x,\Sigma)\leq \frac{L_0}{2}\}$.
\end{lemma}

\begin{proof}
The existence is clear by the definition of a singularity and since $\ab$ is compact. The property $x_0\in D$ follows from Lemma~\ref{cinD}.
\end{proof}

\begin{definition}\label{Parabolicrescaling}
Let $x_0 \in \Rzwei$ be a blowup point of $c$ with type~I singularity at $T<\infty$. The following procedure of rescaling of the curves $c$ is called \emph{parabolic rescaling}:
\begin{align*}
\tilde c_j (p,\tau)  \defi Q_j \left( c(p,\tfrac{\tau}{Q_j^2} + T) - x_0\right), \ \ \text{ for } (p,\tau)\in [a,b]\times  [-Q_j^2 T,0).
\end{align*}
The rescaled curves are a family of solutions of the area preserving curve shortening flow with Neumann free boundary conditions. 
\end{definition}

\begin{proposition}\label{ParabolicGrenzuebergangII} 
  Let $c:[a,b]\times[0,T)\to \Rzwei $ be a solution of (\ref{flow})
  with type~I singularity at $T<\infty$. Let $x_0\in\Rzwei$ be a blowup point.
 Assume further $L(\ct)\geq c_1>0 $ and $|\bar\kappa(t)|\leq c_2<\infty$,
  where $c_1$ and $c_2$ are constants independent of~$t$.
  Consider the parabolic rescaling as in Definition~\ref{Parabolicrescaling}. Then there exist reparametrizations $\psi_j:I_j\to[a,b]$ with $|I_j|\to\infty$ $(j\to\infty)$ such that a subsequence of the curves $$\tilde c_j(\psi_j,\cdot): I_j\times [-Q_j^2T,0)\to\Rzwei $$ 
  converges smoothly on compact subsets of $ I\times (-\infty,0)$ (where $I$ is an unbounded interval containing $0$) to a solution of the curve shortening flow $\tilde\gamma_\infty: I\times(-\infty,0)\to\Rzwei$ with the following properties:
\begin{enumerate}
 \item The length of $\tilde\gamma_\infty(\cdot,\tau)$ is not bounded for each $\tau\in (-\infty,0)$. \label{laenge}
 \item If $\tilde M^\infty_\tau\defi\tildega_\infty( I,\tau)$ has  boundary $\partial\tilde M^\infty_\tau$ then $\partial\tilde M^\infty_\tau\subset\, \nix^\infty\Sigma$, where $\nix^\infty\Sigma$ is a line through $0\in\Rzwei$, and $\langle \tilde\nu_\infty, ^{ ^\infty\Sigma}\nu\rangle = 0$ on $\partial\tilde M^\infty_\tau$. \label{rand}
\end{enumerate}
\end{proposition}

\begin{proof}
    For the convergence, we first follow Ecker, see \cite[Remark~4.22~(2)]{EckerBuch}. Fix $\delta\in(0,\frac{1}{2})$. By rescaling and by the type I assumption we get $|\tilde \kappa_j|^2(p,\tau) \leq \frac{\bar c_0}{2\delta}$ for $j\geq j_0(\delta)$, $\tau \in [-\frac{T}{\delta},-\delta]$. With the type I property, it is easy to show that there is a radius $R_0=R_0(\bar c_0,\delta)$ such that $\tilde c_j(\cdot,\tau)$ intersects $\overline{B_{R_0}(0)}$. Choose a fixed $\tau_2\in(-\infty,0)$. We get reparametrizations $\psi_j: I_j\to \ab$ such that $\tildec(\psi_j,\tau_2)$ is parametrized by arclength. We use the notation $\psi_j^{-1}\eqqcolon \varphi_j$. Because of $L\left(c_t\right)\geq c_1>0$ we know that $|I_j|\to\infty\ (j\to\infty)$. We have two possibilities (after choosing a subsequence): Either $I_j \to(-\infty,\infty)$ or we are in the boundary case $I_j \to (-\infty,\tilde b]$ for a $\tilde b \geq 0$ (or $I_j\to [a,\infty)$ for an $\tilde a\leq 0$). We define $I\defi \lim_{j\to\infty} I_j$. 
    Consider now
    \begin{align*}
      \gamma_j:I_j\times [-Q_j^2 T,0)\to\Rzwei,\ \ \gamma_j(p,\tau)\defi \tildec(\psij(p), \tau).
    \end{align*}
    Then $\gaj(\cdot, \tau_2)$ is parametrized by arclength, $\gaj$ satisfies 
    \begin{align*} 
     |\kappa_j|^2\leq \frac{\bar c_0}{2\delta}\ \text{ on }\ \Ij\times[-\Tdel,-\delta]\ \  \forall j\geq j_0(\delta)
    \end{align*}
    and $\gaj$ are solutions of the area preserving curve shortening problem with Neumann free boundary conditions. Using $L(\ct)\geq c_1>0$ we have that $L(\gaj(\cdot,\tau))\geq Q_j c_1>0$ $\forall \tau\in[-Q_j^2 T,0)$, $\forall j\in\N$.
    Under these conditions, we reprove the gradient estimates from \cite[Chapter~7]{StahlDiss} for the local graph representation (see also the proof of Proposition~\ref{exis}). This yields at first $C^{k + \alpha,\frac{k+\alpha}{2}}-$estimates on the graph representation (away from the initial time) for every $k\in\N$. (We need the lower bound on $L(\gaj(\cdot,\tau))$ at this point to get estimates on the time derivatives of $\bar\kappa_j$.) Since the terms $\ps^m\kappa_j$ ($m\in\N_0$) are independent of the parametrization and by the flow equation, we then get
    \begin{align*}
     |\pt^k\ps^m\kappa_j|\leq c(k,m,\bar c_0,\delta, T, \Sigma,c_1)\ \text{ on } \ \Ij\times[-\Tdel,-\delta],\ \ \forall j\geq j_0(\delta).
    \end{align*}
      As in  \cite[Proof of Theorem 3.1]{DziukKuwertSch} we decompose a derivative $\pp^m\gaj$ into its normal and tangential part. We get for $m\geq 2$ with $h_j(p,\tau)\defi|\pp \gaj(p,\tau)|$
    \begin{align*}
     \langle \pp^m \gaj,\tau_j\rangle &= P^{N(m)}(h_j,\dots,\ps^{m-1} h_j,\kappa_j,\dots,\ps^{m-3} \kappa_j),\\
     \langle \pp^m \gaj,\nu_j\rangle &= P^{N(m)}(h_j,\dots,\ps^{m-2} h_j,\kappa_j,\dots,\ps^{m-2} \kappa_j),
    \end{align*}
    where $N(m) \in\N$ and  $\ps^{l}\kappa\defi \kappa$ for $l\leq 0$. The symbol $P^{N}\left(\{f_i\}_{i\leq n_0}\right)$ denotes a linear combination of products with at most $N$ factors which are elements of the set $\{f_i\}_{i\leq n_0}$. Since $h_j$ satisfies the evolution equation $\pt h_j = -\kappa_j(\kappa_j-\bar\kappa_j)h_j$ we get estimates on $|\ps^m h_j|$ independent of $j$ by an induction argument. Using again the evolution equation, we have
      \begin{align*}
       |\pt^k\pp^m\gaj|\leq c(k,m,\bar c_0,\delta, T, \Sigma,c_1)\ \text{ on } \ \Ij\times[-\Tdel,-\delta], \ \ \forall j\geq j_0(\delta).
      \end{align*}
      We use this estimate and the inequality $|\gaj(0,\tau)|\leq R_0$ $\forall j\geq j_0(\delta)$ to get for every compact subset $K$ of $I$ a subsequence  of the curves $\gaj$ converging by the Arzela-Ascoli theorem smoothly on $K\times[-\Tdel,-\delta]$ to a (non-empty) limit flow. By the usual diagonal sequence argument and by rearranging the sequences in every step, we get a subsequence of $\gaj$ such that this subsequence converges on every compact subset of $I\times (-\infty,0)$ to a well-defined limit flow $\gainfty: I \times (-\infty,0)\to \Rzwei$. 
      The limit flow is a solution of the curve shortening flow because we have $ |\bar{\tilde\kappa}_j(\tau)|  \to 0 \ \ (j\to\infty)$ due to $|\bar\kappa|\leq c_2$,
    and this yields $\pt \gainfty \leftarrow \pt\gamma_{j_l} = (\kappa_{j_l}-\bar\kappa_{j_l})\nu_{j_l} \to \tilde\kappa_\infty \tilde\nu_\infty$ with $l\to\infty$ and therefore $\pt\gainfty = \tilde\kappa_\infty\tilde\nu_\infty$.\\
  
   \noindent For (\ref{laenge}), we fix any $\tau_0\in (-\infty,0)$. If $\tau_0=\tau_2$ then $\gainfty(\cdot,\tau_0):I\to\Rzwei$ is parametrized by arclength (because $|\gamma_{j_l}'(\cdot,\tau_2)|=1$ for every $j_l$) and $|I|=\infty$, therefore $L(\gainfty(\cdot,\tau_2))=\infty$ is immediate. 
   For $\tau_0\not=\tau_2$ we use the fact that $\gainfty$ is a solution of the curve shortening flow. We know by the type I assumption that
    \begin{align*} 
     |\kappa_{\gamma_{j_l}}(\cdot,\tau)|^2\leq \frac{\bar c_0}{2 (-\tau)}\leq
	\begin{cases}
              \frac{\bar c_0}{2 (-\tau_0)} & \text{on  } I_{j_l}\times [\tau_2,\tau_0] \text{ if } \tau_2<\tau_0,\\
	 \frac{\bar c_0}{2 (-\tau_2)} & \text{on  } I_{j_l}\times [\tau_0,\tau_2] \text{ if } \tau_0<\tau_2,\\
            \end{cases}
    \end{align*}
    where $j_l$ is big enough such that $\gamma_{j_l}$ is defined on the time interval. This implies for all $p\in I$ 
    \begin{align} \label{kappainfty}
     |\tilde \kappa_{\infty}(p,\tau)|^2\leq 
	\begin{cases}
              \frac{\bar c_0}{2 (-\tau_0)} & \text{on  } [\tau_2,\tau_0] \text{ if } \tau_2<\tau_0,\\
	 \frac{\bar c_0}{2 (-\tau_2)} & \text{on  }  [\tau_0,\tau_2] \text{ if } \tau_0<\tau_2.\\
            \end{cases}
    \end{align}
    We define $g(p,\tau)\defi |\gainfty'(p,\tau)|$ and compute 
    \begin{align*}
     \partial_\tau g = \frac{\langle\left(\partial_\tau \gainfty\right)',\gainfty'\rangle}{|\gainfty'|} = -\frac{1}{g}\langle \partial_\tau \gainfty,\gainfty''\rangle = -\frac{1}{g} \tilde\kappa_\infty \langle\tilde\nu_\infty,\gainfty''\rangle = -\tilde\kappa_\infty^2 g.
    \end{align*}
    With $g(\cdot,\tau_2) = 1$, this implies
    \begin{align*}
     \partial_\tau\left(\ln g\right) = -\tilde\kappa_\infty^2 \ \ \ \text{ and } \ \ \ g(p,\tau_0) = \exp\left(-\int_{\tau_2}^{\tau_0}\tilde\kappa_\infty^2(p,\sigma)\de \sigma\right).
    \end{align*}
    We use (\ref{kappainfty}) and get
    \begin{align*}
     g(p,\tau_0) \geq \exp\left(-(\tau_0-\tau_2)\frac{\bar c_0}{2(-\tau_0)}\right)>0 \ \ \text{ if } \tau_2<\tau_0
    \end{align*}
      and
     \begin{align*}
     g(p,\tau_0) = \exp\left(\int_{\tau_0}^{\tau_2}\tilde\kappa_\infty^2\right)\geq 1 \ \ \text{ if } \tau_0<\tau_2.
    \end{align*}
    These estimates are independent of the point $p$. Together with $|I|=\infty$, this yields
    \begin{align*}
     L\left(\gainfty(\cdot,\tau_0)\right) = \int_{I}|\gainfty'(p,\tau_0)|\de p =  \int_{I}g(p,\tau_0)\de p=\infty.
   \end{align*}
 \ \\

      \noindent For (\ref{rand}), note that if $x_0\not\in\Sigma$ then for every $y\in \tilde\Sigma^{x_0}_j\defi Q_j\left(\Sigma-x_0\right)$ there is an $x\in\Sigma$ such that 
      \begin{align*}
       |y|=Q_j|x-x_0| \geq Q_j \ \dist(x_0,\Sigma) \to\infty  \ \ \ \ (j\to\infty),
      \end{align*}
      which implies that $\Sigma$ is drifting off to infinity. Therefore, if $\tilde M_\tau^\infty$ has a boundary then $x_0\in\Sigma$. In this situation, either $\varphi_j(a)$ converge to an $\tilde a\leq0$ or $\varphi_j(b)$ converge to a $\tilde b\geq0$. 
      We treat the case $I=[\tilde a,\infty)$. It is immediate that $\tilde\Sigma^{x_0}_j$ converges to a line $^\infty\Sigma$ through $0\in\Rzwei$ (note $0\in \tilde\Sigma^{x_0}_j$ $\forall j$ and $|\kappa^{\tilde\Sigma^{x_0}_j}| = \frac{1}{Q_j}|\kappa^\Sigma|\to 0$). We have
      \begin{align}\label{blub}
       \gaj(\varphi_j(a),\tau)= Q_j(c(a,\tfrac{\tau}{Q_j^2} + T) - x_0)\in \tilde\Sigma^{x_0}_j
      \end{align}
      This implies $\gainfty(\tilde a,\tau)\in \,^\infty\Sigma$. We know that $\langle\nu^{\tilde\Sigma^{x_0}_j}(x),\vec\nu_{\gaj}(x,\tau)\rangle =0$ for $x=\gaj(\varphi_j(a),\tau)$ $\forall \tau$ because of (\ref{blub}) and the boundary conditions of $\gaj$. This implies $\langle\nu^{^\infty\Sigma}(\gainfty(\tilde a,\tau)),\nu_{\gainfty}(\tilde a,\tau)\rangle =0$.
\end{proof}

\begin{corollary} \label{nichttriv}
 Consider the situation from Proposition~\ref{ParabolicGrenzuebergangII}, but with the additional assumption $\kappa_0\geq 0$ for the initial curve $c_0$. Then we have for the limit flow $\gainfty$: There is a time $\tau\in[-\frac{\bar c_0}{2},-\frac{1}{2}]$ such that $\tilde\kappa_\infty(0,\tau)=1$.
\end{corollary}

\begin{proof}
We define $\tau_j\defi-Q_j^2(T-t_j)$ and compute (where $t_j\nearrow T$ and $p_j\to p_0$ comes from the blowup sequence) 
    \begin{align}\label{gleich1a}
     \kappa_{\gamma_j}(0,\tau_j)= \kappa_{\tildec}(p_j,\tau_j)=\frac{1}{Q_j}\kappa\left(p_j,\tfrac{-Q_j^2(T-t_j)}{Q_j^2} + T\right)=\frac{1}{Q_j}\kappa(p_j,t_j) = 1 \ \ \ \forall j\in\mathbb N.
    \end{align}
      By the type I assumption, we get
      \begin{align*}
       \tau_j = -\kappa^2(p_j,t_j)(T-t_j)\geq- \frac{\bar c_0}{2(T-t_j)}(T-t_j) = - \frac{\bar c_0}{2}.
      \end{align*}
      The other inequality follows from Proposition~\ref{Amax} (where we need $\kappa\geq 0$):
      \begin{align*}
       \tau_j = -\kappa^2(p_j,t_j)(T-t_j)\leq- \frac{1}{2(T-t_j)}(T-t_j) = - \frac{1}{2}.
      \end{align*}
      Therefore, there is a $\tau\in [- \frac{\bar c_0}{2},- \frac{1}{2}]$ such that after passing to a subsequence $\tau_{j}\to\tau$ ($j\to\infty$).
      Together with (\ref{gleich1a}) and the local smooth convergence of $\gamma_{j}$ to $\gainfty$, this implies $ \tilde\kappa_\infty(0,\tau)=1.$
\end{proof}

\subsection{The blowup point on the support curve}

\begin{proposition}[The monotonicity formula for curves] \label{monocurve}
 Let $c:[a,b]\times[0,T)\to \Rzwei $ be a solution of the area preserving curve shortening problem with Neumann free boundary conditions meeting the support curve $\Sigma$ from the outside. Consider
\begin{align*}
 \rho_{x_0,T}(x,t) &\coloneqq \frac{1}{(4\pi (T-t))^{\frac{1}{2}}} \exp{\left(-\frac{|x-x_0|^2}{4(T-t)}\right)}\ \ \mbox{ for } x,x_0 \in \Rzwei, t\in\nut,\\
f(t) &\defi \exp{\left( -\frac{1}{2}\int_0^t \bar\kappa^2(\sigma)\mbox{\,d}\sigma\right)}\ \ \mbox{ for } t\in\nut.
\end{align*} 
 If $x_0\in\Sigma$ and $\Sigma$ is convex then we have for all $t\in[0,T)$
\begin{align*}
   \frac{\de}{\de t}\left( f\int_{c_t}  \rho_{x_0,T} \de s_t\right) \leq -\frac{1}{2} f\int_{c_t}\left(\big| \kappa + \frac{\langle x-x_0,\nu\rangle}{2(T-t)}\big|^2 + \big|(\kappa-\bar\kappa) + \frac{\langle x-x_0,\nu\rangle}{2(T-t)}\big|^2 \right)\rho_{x_0,T} \de s_t.
\end{align*}
\end{proposition}

\begin{proof}
 We use the notation $\rho\defi\heatkernel$ and compute with the evolution equation (\ref{evolution1}):
    \begin{align}\begin{split}\label{helps}
     \frac{\de}{\de t}&\left( f\int_{c_t}  \rho\de s_t\right) = \frac{\de}{\de t}\left( f\intab  \rho\circ c_t \de s\right)\\
      & = \left(\frac{\de}{\de t} f\right) \intab \rho\circ c_t \de s + f\intab \frac{\de}{\de t} \left( \rho\circ c_t \right) \de s + f\intab  \rho\circ c_t \pt\left(\de s\right)\\
      &= -\frac{1}{2}  \bar\kappa^2f \intab \rho \circ \ct\de s +  f\intab D\rho\circ \ct \;\pt\ct + \pt \rho\circ \ct \de s - f \intab \kappa(\kappa-\bar\kappa)\rho\circ \ct\de s\\	
      &= -\frac{1}{2}  f \intab \rho \circ \ct\left(\bar\kappa^2 + 2\kappa(\kappa-\bar\kappa)\right)\de s +  f\intab D\rho\circ \ct \;\pt\ct + \pt \rho\circ \ct \de s. \end{split}
    \end{align}
    The derivatives of the heatkernel are
    \begin{align*}
     D\rho (x,t) = -\frac{1}{2(T-t)}\rho  (x,t) (x-x_0),\\
      \pt \rho (x,t) = \left(\frac{1}{2(T-t)}-\frac{|x-x_0|^2}{4(T-t)^2}\right)\rho(x,t).
    \end{align*}
    We note that
    \begin{align}\label{Q(g)}
     \pt \rho \circ\ct + \langle \ps\left(D\rho\circ\ct\right),\tau\rangle + \frac{|D^\perp\rho\circ\ct|^2}{\rho\circ\ct}=0.
    \end{align}
    This follows because we have
    \begin{align*}
     \langle \ps\left(D\rho\circ\ct\right),\tau\rangle &= -\frac{1}{2(T-t)}\frac{1}{|\pp \ct|}\langle \pp\left(\rho\circ\ct (\ct-x_0)\right),\tau\rangle\\
      &= -\frac{1}{2(T-t)}\rho\circ\ct \langle \tau,\tau\rangle - \frac{1}{2(T-t)}\frac{1}{|\pp \ct|}\pp\left(\rho\circ\ct\right)\langle \ct-x_0,\tau\rangle\\
      &=-\frac{1}{2(T-t)}\rho\circ\ct + \frac{1}{4(T-t)^2}\rho\circ\ct\langle\ct-x_0,\tau\rangle^2\\
      &= -\pt \rho \circ\ct - \frac{|D^\perp\rho\circ\ct|^2}{\rho\circ\ct}.
    \end{align*}
    By (\ref{Q(g)}) and the Frenet equation $\ps \tau=\kappa\nu$ we compute  
    \begin{align}\begin{split}\label{Q(g)tilde}
     \pt \rho \circ\ct & = - \langle \ps\left(D\rho\circ\ct\right),\tau\rangle -\frac{|(\ct-x_0)^\perp|^2}{4(T-t)^2} \rho\circ\ct\\
    &= - \langle \ps\left(D^\top\rho\circ\ct\right),\tau\rangle - \langle \ps\left(D^\perp\rho\circ\ct\right),\tau\rangle -\frac{|(\ct-x_0)^\perp|^2}{4(T-t)^2} \rho\circ\ct\\
    &= - \ps \langle D\rho\circ\ct,\tau\rangle + \langle D\rho\circ\ct,\kappa\nu\rangle -\frac{|(\ct-x_0)^\perp|^2}{4(T-t)^2} \rho\circ\ct\\
     &= \frac{1}{2(T-t)} \ps\left(\rho\circ\ct\langle\ct-x_0,\tau\rangle\right) - \kappa\frac{\langle \ct-x_0,\nu\rangle}{2(T-t)}\rho\circ\ct -\frac{|(\ct-x_0)^\perp|^2}{4(T-t)^2} \rho\circ\ct.\end{split}
    \end{align}
    Thus, we get with the flow equation $\pt \ct = (\kappa-\bar\kappa)\nu$, (\ref{Q(g)tilde}) and (\ref{helps})
    \begin{align*}
      \frac{\de}{\de t}&\left( f\int_{c_t}  \rho\de s_t\right) =  -\frac{1}{2}  f \intab \rho \circ \ct\left(\bar\kappa^2 + 2\kappa(\kappa-\bar\kappa)\right)\de s -   f\intab \frac{\langle\ct-x_0,\pt\ct\rangle}{2(T-t)}\rho\circ \ct \de s \\
      &\ \ \ \ +f\intab \pt \rho\circ \ct \de s\\ 
      &=  -\frac{1}{2}  f \intab \rho \circ \ct\left(\bar\kappa^2 + 2\kappa(\kappa-\bar\kappa) + 2 (\kappa-\bar\kappa) \frac{\langle\ct-x_0,\nu\rangle}{2(T-t)}\right)\de s + f\intab \pt \rho\circ \ct \de s\\  \displaybreak[0]
      &= -\frac{1}{2}  f \intab \rho \circ \ct\left(\bar\kappa^2 + 2\kappa(\kappa-\bar\kappa) + 2 (\kappa-\bar\kappa) \frac{\langle\ct-x_0,\nu\rangle}{2(T-t)}\right)\de s\\  
      &\ \ \ \ + \frac{1}{2(T-t)} f \intab \ps\left(\rho\circ\ct\langle\ct-x_0,\tau\rangle\right)\de s-f\intab\left( \kappa \frac{\langle\ct-x_0,\nu\rangle}{2(T-t)} + \frac{|(\ct-x_0)^\perp|^2}{4(T-t)^2} \right) \rho\circ \ct \de s\\
      &= -\frac{1}{2}  f \intab \rho \circ \ct\left(\bar\kappa^2 + 2\kappa(\kappa-\bar\kappa) + 2 (2\kappa-\bar\kappa) \frac{\langle\ct-x_0,\nu\rangle}{2(T-t)} +2\frac{|(\ct-x_0)^\perp|^2}{4(T-t)^2} \right)\de s\\
      &\ \ \  \ + \frac{1}{2(T-t)} f\; \Big[\rho(c(b,t),t)\langle c(b,t)-x_0,\tau(b,t)\rangle-\rho(c(a,t),t)\langle c(a,t)-x_0,\tau(a,t)\rangle\Big]\\
      &=  -\frac{1}{2}  f \intab \rho \circ \ct\left[\left( \kappa + \frac{\langle \ct-x_0,\nu\rangle}{2(T-t)}\right)^2 + \left((\kappa-\bar\kappa) + \frac{\langle x-x_0,\nu\rangle}{2(T-t)}\right)^2 \right]\de s\\
      &\ \ \ \ + \frac{1}{2(T-t)} f\; \Big[\rho(c(b,t),t)\langle c(b,t)-x_0,\tau(b,t)\rangle-\rho(c(a,t),t)\langle c(a,t)-x_0,\tau(a,t)\rangle\Big]\\
      &\leq -\frac{1}{2}  f \intab \rho \circ \ct\left[\left( \kappa + \frac{\langle \ct-x_0,\nu\rangle}{2(T-t)}\right)^2 + \left((\kappa-\bar\kappa) + \frac{\langle \ct-x_0,\nu\rangle}{2(T-t)}\right)^2 \right]\de s,
    \end{align*}
    where we used in the last step $\langle c(b,t)-x_0,\tau(b,t)\rangle\leq0$ and $\langle c(a,t)-x_0,\tau(a,t)\rangle\geq0$. This is true because $\Sigma$ is convex, $x_0\in\Sigma$ and the curves $c_t$ meet $\Sigma$ from the outside.
\end{proof}

\begin{remark}
 By Buckland's expansion formula \cite[Proposition~2.3]{Buckland} modified for general flows as in the proof of Lemma~\ref{Parabolicstone}, we have the analogous monotonicity formula as above for arbitrary dimension (volume preserving mean curvature flow with Neumann free boundary conditions), see \cite[Chapter~1]{MeineDiss}.
\end{remark}

\begin{proposition} \label{proper}
Consider the situation of Proposition~\ref{ParabolicGrenzuebergangII} with the additional condition $x_0\in\Sigma$. Then for each $\tau\in (-\infty,0)$, we have
\begin{align*}
 L\left(\gainfty^\tau \cap B_R(0)\right) \leq c(\tau,T, L(c_0),c_2,R),
\end{align*}
where $L\left(\gainfty^\tau\cap B_R(0)\right)$ denotes the length of the curve $\gainfty^\tau$ inside of the ball $B_R(0)$. 
It follows that each of the curves $\tildega_\infty(\cdot,\tau)= \gainfty^\tau$ is proper. 
\end{proposition}

\begin{proof}
  We use the monotonicity formula for $\tildec^\tau \defi\tilde c_j\left(\cdot,\tau\right) $, see Proposition~\ref{monocurve}. 
  By rescaling we get  
    \begin{align*}
     \frac{\de}{\de \tau }\left(\tilde f_j \int_{\tildec^\tau}\rho_{0,0}\de\tilde s_\tau\right)\leq -\frac{1}{2}\tilde f_j\int_{\tildec^\tau}\left(\big|\tilde \kappa_j + \frac{\langle \tilde x,\tilde \nu\rangle}{-2\tau}\big|^2 + \big|(\tilde\kappa_j-\bar{\tilde\kappa}_j)+\frac{\langle\tilde x,\tilde\nu\rangle}{-2\tau}\big|^2\right)\rho_{0,0}\de\tilde s_\tau
    \end{align*}
    for $\tau\in [-Q_j^2T,0)$, where 
    \begin{align*}
     \rho_{0,0}(x,\tau)&\defi \frac{1}{\sqrt{-4\pi \tau}} e^{-\frac{|x|^2}{-4 \tau}},\\
      \tilde f_j(\tau)&\defi\exp\left(-\frac{1}{2}\int_{-Q_j^2 T}^\tau\bar{\tilde \kappa}_j^2(\sigma)\de\sigma\right).
    \end{align*}
    This implies 
    \begin{align*}
     \tilde f_j(\tau) \int_{\tildec^\tau}  \rho_{0,0} \de\tilde s_\tau\leq \tilde f_j(-Q_j^2 T) \int_{\tildec^{-Q_j^2 T}}\rho_{0,0}(\cdot,-Q_j^2 T)\de \tilde s_{-Q_j^2 T} \leq \frac{L(c_0)}{(4\pi T)^{\frac{1}{2}}}
    \end{align*} 
    for $\tau \in[-Q_j^2 T, 0)$.
    Since $T<\infty$ and $|\bar\kappa|<c_2$ there is a constant $c_3=c_3(T,c_2)>0$ such that $c_3\leq f(t)\leq 1 \ \forall t\in [0,T)$. With $\tilde f_j (\tau) = f(\frac{\tau}{Q_j^2} + T)$ we conclude
    \begin{align*}
     \int_{\tildec^\tau}  \rho_{0,0} \de\tilde\mu_\tau \leq c = c(T, L(c_0), c_2).
    \end{align*}
      We use Fatou's lemma to get
      \begin{align*}
       \int_{I} \rho_{0,0}\circ \gainfty(\cdot,\tau)|\gainfty'(\cdot,\tau)|\de\mathcal L^1\leq c = c(T, L(c_0), c_2)
      \end{align*}
    and thus
    \begin{align*}
     c(T,L(c_0),c_2) &\geq  \int_{\gainfty^\tau\cap B_R(0)}  \rho_{0,0} \de\tilde\mu^\tau_\infty 
      = \frac{1}{\sqrt{- 4\pi \tau}} \int_{\gainfty^\tau\cap B_R(0)} \exp \left(-\frac{|\tilde x|^2}{-4 \tau}\right) \de\tilde\mu^\tau_\infty\\
      &\geq \frac{1}{\sqrt{- 4\pi \tau}} \exp\left(-\frac{R^2}{-4\tau }\right) L(\gainfty^\tau \cap B_R(0)) \quad \text{ for each }  \tau \in (-\infty,0).
    \end{align*}
    This yields $L(\gainfty^\tau\cap B_R(0)) \leq c(\tau,T, L(c_0),c_2,R).$ 
    This property implies that $\gainfty(\cdot,\tau)$ is proper. 
\end{proof}

\begin{proposition} \label{Parabolicselfsimilar}
  Let $c:[a,b]\times[0,T)\to \Rzwei $ be a solution of (\ref{flow})
  with type~I singularity at $T<\infty$. Assume that $c$ satisfies $ L(\ct)\geq c_1>0 $ and $|\bar\kappa(t)|\leq c_2<\infty$.
 If the blowup point is on the support curve $\Sigma$ then every limit flow $\gainfty$ from Proposition~\ref{ParabolicGrenzuebergangII} satisfies $\tilde \kappa^\infty_\tau =  \frac{\langle \tilde x^\infty,\tilde \nu^\infty\rangle}{2\tau}$.
%
\end{proposition}

\begin{proof}
  Because of the monotonicity formula, see Proposition~\ref{monocurve}, we have
    \begin{align*}
      \frac{\de}{\de t}\left( f(t)\int_{c_t}  \rho_{x_0,T}\de s_t\right) \leq 0
    \end{align*}
    and therefore the limit $ \lim_{t\nearrow T} \left( f(t)\int_{c_t}  \rho_{x_0,T}\de s_t\right)$
    exists. By $0<c_3(T, c_2)\leq f(t)\leq 1 \ \forall t\in [0,T)$, the limit $ \lim_{t\nearrow T} \int_{c_t}  \rho_{x_0,T}\de s_t$
    exists. By rescaling we have 
    \begin{align}\label{Parabolicconv}
      \lim_{t \nearrow T}\int_{c_t}  \rho_{x_0,T}\de s_t = \lim_{j_l\to\infty}\int_{\tilde c_{j_l}^{\tau}} \rho_{0,0} \de\tilde s_\tau =  \int_{\gainfty^\tau}  \rho_{0,0} \de \tilde s_\infty^\tau.
    \end{align}
    We note that for the last step local convergence is not enough. To get (\ref{Parabolicconv}) we use a lemma of Andrew Stone modified for our case, see \cite[Lemma 2.9]{Stone} and Lemma~\ref{Parabolicstone}. Equality (\ref{Parabolicconv}) means that the integral $\int_{\gainfty^\tau}  \rho_{0,0} \de \tilde s_\infty$ does not depend on $\tau$. 
    If we are in the case that $\tilde M^\infty_\tau$ has a boundary, we reflect the curves at the line $\Sigma^\infty$ to get a family of curves $\left(\gainfty^\tau\right)_{\tau<0}$ without boundary which evolve under the curve shortening flow. Now we can use Huisken's monotonicity formula (see \cite{Huisken90}), 
    \begin{align*} 
     \frac{\de}{\de \tau} \int_{ \gainfty^\tau}  \rho_{0,0} \de \tilde s_\infty^\tau = - \int_{ \gainfty^\tau}  \big|\kappa^\infty - \frac{\langle x^\infty, \nu^\infty\rangle}{2\tau} \big| \de\tilde s_\infty^\tau.
    \end{align*}	
    Since the integral on the left hand side is independent of $\tau$, the term on the right hand side vanishes and we get the result $ \kappa^\infty_\tau =  \frac{\langle x^\infty,\nu^\infty\rangle}{2\tau}$. If $\tilde M^\infty_\tau$ does not have a boundary then we can apply the monotonicity formula of Huisken directly to $\gainfty^\tau$ and get the same result.
\end{proof} 

\begin{proposition} \label{ausschlussinSigma}
Let $c_0:[a,b]\to\Rzwei$ be a convex initial curve, $\kappa_0\geq 0$. Consider a solution $c:[a,b]\times[0,T)\to \Rzwei $ of the area preserving curve shortening problem with Neumann free boundary conditions with singularity at $T<\infty$,
  where the  blowup point lies on the support curve $x_0\in\Sigma$. Assume further 
  $L(\ct)\geq c_1>0$ and  $|\bar\kappa(t)|\leq c_2<\infty$.
  Then the singularity cannot be of type I.
\end{proposition}

\begin{proof}
      By Proposition~\ref{Parabolicselfsimilar}, we obtain for the limit curve $\tilde \kappa^\infty_\tau = \frac{ \langle \tilde x^\infty,\tilde \nu^\infty\rangle}{2\tau}$. If $\tilde M^\infty_\tau$ has boundary $\partial\tilde M^\infty_\tau$ then $\partial\tilde M^\infty_\tau\subset \nix^\infty\Sigma =  \{0\} \times \R$ (after rotation), and it is perpendicular to that line. 
      Like in the proof on Proposition~\ref{Parabolicselfsimilar}, we reflect $\tilde M^\infty_\tau$ at  $ \nix^\infty\Sigma$ to get a smooth curve $ M^\infty_\tau$ without boundary which satisfies $\kappa^\infty_\tau = \frac{ \langle x^\infty, \nu^\infty\rangle}{2\tau}$. 
      This equation implies that $M^\infty_\tau = \sqrt{-\tau}M_{-1}^\infty$, i.e.~the curves are self-similarly shrinking. These curves are classified, see \cite[Section~5]{Halldorsson}, and also \cite{AbreschLanger}. 
      We use this classification to obtain that $ M^\infty_\tau$ is one of the following:
    \begin{enumerate}
      \item the line $\R\times \{0\}$,\label{Parabolictrivial}
    \item the shrinking sphere $\mathbb S^1_{\sqrt{-2\tau}}$,\label{Parabolicspaehre}
    \item one of the closed, homothetically shrinking Abresch-Langer curves $\Gamma$,\label{Parabolicabresch-langer}
    \item a curve whose image is dense in an annulus of $\Rzwei$. \label{dense}
    \end{enumerate}
    The first case, (\ref{Parabolictrivial}), clearly contradicts $\tilde\kappa_\infty(0,\tau)=1$ for a $\tau\in[-\frac{\bar c_0}{2}, -\frac{1}{2}]$, see Corollary~\ref{nichttriv}. 
    The cases (\ref{Parabolicspaehre}) and (\ref{Parabolicabresch-langer}) are excluded because this would be a contradiction to the properness (Proposition~\ref{proper}) and the unbounded length of $ M^\infty_\tau$. Also (\ref{dense}) is not a possible shape of the limiting curves because then you could find a ball where the length of the curves is not bounded. This contradicts Proposition~\ref{proper}.
\end{proof}

\subsection{The blowup point not on the support curve}

\noindent We modify the approach of Ecker in \cite[Remark 4.8, Proposition 4.17]{EckerBuch} for our flow.

\begin{proposition} \label{monocurveII}
 Let $c:[a,b]\times[0,T)\to \Rzwei $ be a solution of (\ref{flow})
 and $\varphi\in C^{2,1}(\Rzwei\times (0,T),\R)$. As in Proposition~\ref{monocurve}, consider the ``backward heat kernel'' $\heatkernel$ and $$f(t) = \exp{\left( -\frac{1}{2}\int_0^t \bar\kappa^2(\sigma)\mbox{\,d}\sigma\right)}.$$
  Then we have for $t\in(0,T)$
\begin{align*}
   \frac{\de}{\de t}\left( f\intab \varphi\, \rho_{x_0,T}\de s_t\right) =&  -\frac{1}{2} f \hspace*{-0.1cm}\intab \hspace*{-0.1cm} \varphi\left(\big| \kappa + \frac{\langle \ct-x_0,\nu\rangle}{2(T-t)}\big|^2 \hspace*{-0.1cm} + \big|(\kappa-\bar\kappa) + \frac{\langle \ct-x_0,\nu\rangle}{2(T-t)}\big|^2 \right)\rho_{x_0,T} \de s_t\\
       & + f \intab\heatkernel\left(\frac{\de}{\de t}- \ps^2\right)\varphi \de s_t\\
    &+ f \left[\heatkernel\,\varphi\frac{\langle x-x_0,\tau\rangle}{2(T-t)} + \heatkernel \langle D\varphi,\tau\rangle\right]_{c(a,t)}^{c(b,t)}.
\end{align*}
Here, $\dt$ denotes the total derivative with respect to $t$, i.e.~$\dt\psi$ is $$\dt (\psi\circ c)=\pt\psi \circ c + \langle D\psi\circ c, \pt c\rangle= \pt\psi\circ c + (\kappa-\bar\kappa)\langle D\psi\circ c, \nu\rangle.$$Similarly, $\ps^2\varphi$ has to be understood as $\ps^2 (\varphi\circ c)$.
\end{proposition}


\begin{proof}
The beginning of the proof is done like in Proposition~\ref{monocurve}. We only have to add the term where $\varphi$ is differentiated. It follows that
\begin{align}\begin{split}\label{re23}
  \frac{\de}{\de t}\left(  f\int_{c_t} \varphi\, \rho\de s_t\right) =  &-\frac{1}{2}  f \intab \varphi\circ\ct\,\rho \circ \ct\left(\bar\kappa^2 + 2\kappa(\kappa-\bar\kappa) + 2 (\kappa-\bar\kappa) \frac{\langle\ct-x_0,\nu\rangle}{2(T-t)}\right)\de s  \\
      &+f\intab \varphi\circ\ct\,\pt \rho\circ \ct \de s
      + f\intab \rho\circ \ct\frac{\de}{\de t}(\varphi\circ \ct)\de s. \end{split}
\end{align}
By (\ref{Q(g)}), we have
\begin{align*}
\begin{split}
     \varphi\circ\ct&\,\pt \rho \circ\ct  = - \langle \ps\left(D\rho\circ\ct\right),\tau\rangle\,\varphi\circ\ct -\frac{|(\ct-x_0)^\perp|^2}{4(T-t)^2} \rho\circ\ct\,\varphi\circ\ct\\
    &= - \langle \ps\left(D^\top\rho\circ\ct\right),\tau\rangle\,\varphi\circ\ct - \langle \ps\left(D^\perp\rho\circ\ct\right),\tau\rangle\,\varphi\circ\ct -\frac{|(\ct-x_0)^\perp|^2}{4(T-t)^2} \rho\circ\ct\,\varphi\circ\ct\\
    &= - \ps (\langle D\rho\circ\ct,\tau\rangle)\,\varphi\circ\ct + \langle D\rho\circ\ct,\kappa\nu\rangle\,\varphi\circ\ct -\frac{|(\ct-x_0)^\perp|^2}{4(T-t)^2} \rho\circ\ct\,\varphi\circ\ct\\
    &= - \ps(\langle D\rho\circ\ct,\tau\rangle\,\varphi\circ\ct) + \langle D\rho\circ\ct,\tau\rangle \ps(\varphi\circ\ct) + \kappa\langle D\rho\circ\ct,\nu\rangle\,\varphi\circ\ct\\
    &\qquad\qquad\qquad\qquad\qquad\qquad\qquad\qquad\qquad\qquad-\frac{|(\ct-x_0)^\perp|^2}{4(T-t)^2} \rho\circ\ct\,\varphi\circ\ct\\
    &= \ps\left(\rho\circ\ct\,\frac{\langle \ct-x_0,\tau\rangle}{2(T-t)}\,\varphi\circ\ct\right) + \ps(\rho\circ\ct) \ps(\varphi\circ\ct) - \kappa \frac{\langle\ct-x_0,\nu\rangle}{2(T-t)}\, \rho\circ\ct\,\varphi\circ\ct\\
    &\qquad\qquad\qquad\qquad\qquad\qquad\qquad\qquad\qquad\qquad-\frac{|(\ct-x_0)^\perp|^2}{4(T-t)^2} \rho\circ\ct\,\varphi\circ\ct\\
    &= \ps\left(\rho\circ\ct\,\frac{\langle \ct-x_0,\tau\rangle}{2(T-t)}\,\varphi\circ\ct + \rho\circ\ct \ps(\varphi\circ\ct)\right)  - \rho\circ\ct\ps^2(\varphi\circ\ct)\\
    &\qquad\qquad\qquad\qquad- \kappa \frac{\langle\ct-x_0,\nu\rangle}{2(T-t)}\, \rho\circ\ct\,\varphi\circ\ct-\frac{|(\ct-x_0)^\perp|^2}{4(T-t)^2} \rho\circ\ct\,\varphi\circ\ct\\
    &= \ps\left(\rho\circ\ct\,\frac{\langle \ct-x_0,\tau\rangle}{2(T-t)}\,\varphi\circ\ct + \rho\circ\ct \langle D\varphi\circ\ct,\tau\rangle\right)  - \rho\circ\ct\ps^2(\varphi\circ\ct)\\
    &\qquad\qquad\qquad\qquad\qquad + \rho\circ\ct\,\varphi\circ\ct\left(- \kappa \frac{\langle\ct-x_0,\nu\rangle}{2(T-t)}-\frac{|(\ct-x_0)^\perp|^2}{4(T-t)^2} \right).
\end{split}
\end{align*}
We use this equality in (\ref{re23}), collect the terms as in the proof of Proposition~\ref{monocurve} an get the desired result.
\end{proof}

\begin{definition} \label{localize}
 For $x_0\in\Rzwei\setminus\Sigma$, $0<T<\infty$ and $\lambda>0$ we define the $C^2$-function
\begin{align*}
 \varphi_{(x_0,T),\lambda}(x,t)\defi \left(1-\frac{|x-x_0|^2 + 2(1 + c_5)(t-T)}{\lambda^2}\right)^3_+, \ \ \ \ x\in\Rzwei, t\in[0,T),
\end{align*}
where $$c_5\defi c_2 (\diam\Sigma + L_0),$$ $c_2$ is a constant (later we use $c_2$ such that $|\bar\kappa|\leq c_2$).
Furthermore, we choose $\lambda_0\in(0,\sqrt{T})$ such that
\begin{align*}
 \overline{B_{\lambda_0\sqrt{1 + 2(1+c_5)}}(x_0)} \times [T-\lambda_0^2,T)\subset \Rzwei\setminus\Sigma\times (0,T)
\end{align*}
and remark for all $\lambda\in(0,\lambda_0]$ and $t\in [T-\lambda^2,T)$ that
\begin{align*}
\spt \varphi_{(x_0,T),\lambda}(\cdot,t)\subset \Rzwei\setminus\Sigma.
\end{align*}
\end{definition}
\begin{remark}
 The definition of the function $ \varphi_{(x_0,T),\lambda}$ comes from \cite[Chapter 3]{Brakke}.
\end{remark}

\begin{lemma}\label{dsleq0}
  Let $c:[a,b]\times[0,T)\to \Rzwei $ be a solution of (\ref{flow})
  with singularity at $T<\infty$, where the blowup point $x_0$ satisfies $x_0\in\Rzwei\setminus\Sigma$. Assume that we have a uniform bound 
 $|\bar\kappa|\leq c_2<\infty.$ 
  Then, for $\lambda>0$, we have
\begin{align*}
 \left(\frac{\de}{\de t} - \ps^2\right)(\phir\circ c)\leq 0
\end{align*}
\end{lemma}

\begin{proof}
 As in \cite[Remark 4.8]{EckerBuch} we define 
\begin{align*}
 \eta(r)\defi\left(1-\frac{r}{\lambda^2}\right)^3_+, \ \ \ \ \ r\in\R
\end{align*}
and note that $\eta'(r)\leq 0$, $\eta''(r)\geq 0$. With $r(x,t) = |x-x_0|^2 + 2(1 + c_5)(t-T)$
we have that $\phir  = \eta\circ r$ and
\begin{align*}
 Dr(x,t) &= 2(x-x_0)\\
D^2r(x,t) &= 2 \text{Id}\\
\partial_t r(x,t) &= 2(1 + c_5).
\end{align*}
Now we compute
\begin{align*}
  \left(\frac{\de}{\de t} - \ps^2\right)(\phir\circ c) &= \left(\frac{\de}{\de t} - \ps^2\right)(\eta\circ r\circ c)\\
      &=  \eta'\circ r\circ c_t \left(\frac{\de}{\de t} - \ps^2\right) (r\circ c) - \eta''\circ r\circ\ct \left(\ps(r\circ \ct)\right)^2\\
      &\leq \eta'\circ r\circ c_t \left(\frac{\de}{\de t} - \ps^2\right) (r\circ c).
\end{align*}
 Therefore, it is sufficient to show $\left(\frac{\de}{\de t} - \ps^2\right) (r\circ \ct)\geq 0$. This is
\begin{align*}
 \left(\frac{\de}{\de t} - \ps^2\right) (r\circ c) &=\langle Dr\circ c,\pt c\rangle + \pt r\circ c - \ps\langle Dr\circ c, \tau\rangle\\
    &= \langle Dr\circ c,\pt c\rangle + \pt r\circ c - \langle D^2 r\circ c \,\tau,\tau\rangle - \langle Dr\circ c,\ps \tau\rangle\\
    &= 2(\kappa-\bar\kappa)\langle c-x_0,\nu\rangle + 2(1+c_5) - 2\langle\tau,\tau\rangle - 2\kappa\langle c-x_0,\nu\rangle\\
    &= 2(c_5 - \bar\kappa \langle c-x_0,\nu\rangle).
\end{align*}
By Lemma~\ref{cinD}, we get $|c-x_0|\leq \diam\Sigma + L_0$ on $\ab\times\nut$ and thus
\begin{align*}
 \bar\kappa\langle c-x_0,\nu\rangle\leq c_2|c-x_0|\leq c_5.
\end{align*}
\end{proof}

\begin{theorem} \label{nottypeI}
Let $c_0:[a,b]\to\Rzwei$ be a convex initial curve, $\kappa_0\geq 0$. Consider a solution $c:[a,b]\times[0,T)\to \Rzwei $  of the area preserving curve shortening problem with Neumann free boundary conditions with singularity at $T<\infty$. Assume further that $ L(\ct)\geq c_1>0$ and $|\bar\kappa(t)|\leq c_2<\infty$,
  where $c_1$ and $c_2$ are constants independent of $t$. Then the singularity cannot be of type I.
\end{theorem}

\begin{proof}                                                                                                                                                   
 Due to Proposition~\ref{ausschlussinSigma} it is sufficient to consider the case $x_0\in\Rzwei\setminus\Sigma$. The proof is almost the same as in the case $x_0\in\Sigma$. We choose $\lambda_0$ as in Definition~\ref{localize}. Then we use the localized monotonicity formula from Proposition~\ref{monocurveII} together with Lemma~\ref{dsleq0} to get 
\begin{align*}
 \frac{\de}{\de t}\left(f\hspace*{-0.1cm}\int_{\ct}\phir\heatkernel\de s_t\right)\leq -\frac{1}{2} f\hspace*{-0.1cm}\int_{c_t}\phir\left(\big| \kappa + \tfrac{\langle x-x_0,\nu\rangle}{2(T-t)}\big|^2 + \big|(\kappa-\bar\kappa) + \tfrac{\langle x-x_0,\nu\rangle}{2(T-t)}\big|^2 \right)\rho_{x_0,T} \de s_t
\end{align*}
for $\lambda\in(0,\lambda_0]$ and $t\in [T-\lambda^2,T)$. Since $f(T)= \exp(-\frac{1}{2}\int_0^T\bar\kappa^2(\sigma)\de \sigma)$ and $\bar\kappa$ is bounded, the limit $\lim_{t\nearrow T}\int_{\ct}\phir\heatkernel\de s_t$
exists. By rescaling, we get
\begin{align} \label{Konvergenz}
  \lim_{t\nearrow T}\int_{\ct}\phir\heatkernel\de s_t= \lim_{j\to\infty }\int_{\tildec^\tau} \tildephi \rho_{0,0}\de \tilde s_\tau = \int_{\gainfty^\tau}\rho_{0,0}\de\tilde s^\tau_\infty,
\end{align}
where we again modify the Lemma of Stone as in Lemma~\ref{Parabolicstone} for the global convergence in the last step. Equation (\ref{Konvergenz}) implies that $\int_{\gainfty^\tau}\rho_{0,0}\de\tilde s^\tau_\infty$ is independent of $\tau$. \\
In the proof of Proposition~\ref{ParabolicGrenzuebergangII} (\ref{rand}) we saw that in the case $x_0\not\in\Sigma$ the limit curve $\Minfty$ does not have a boundary. For the smooth family of curves $(\gainfty)_{\tau<0}$ without boundary which is evolving under the curve shortening flow we use Huisken's monotonicity formula \cite{Huisken90} and get
\begin{align*}
     0=\frac{\de}{\de \tau} \int_{\gainfty^\tau}  \rho_{0,0} \de \tilde s_\infty^\tau = - \int_{ \gainfty^\tau}  \big|\kappa^\infty - \frac{\langle x^\infty, \nu^\infty\rangle}{2\tau} \big| \de\tilde s_\infty^\tau.
\end{align*}
Therefore, $\Minfty$ is a homothetically shrinking solution of the curve shortening flow. We rescale the localized monotonicity formula as in Proposition~\ref{proper} and use $\phir\leq \left(1 + \tfrac{2(1+c_5)T}{\lambda^2}\right)^3$ to see that each limit curve $\tildega_\infty(\cdot,\tau)$ is proper.
With the results of Abresch-Langer and Halldorsson, \cite{AbreschLanger} and \cite[Section 5]{Halldorsson}, we conclude that the only properly immersed, homothetically shrinking curve of unbounded length is a line through the origin. But this is not possible because there is a time $\tau\in[-\frac{\bar c_0}{2}, -\frac{1}{2}]$ such that $\tilde\kappa_\infty(0,\tau)=1$ due to Corollary~\ref{nichttriv}.
\end{proof}

\section{Preserving of a geometric property} \label{geometricproperties}

\renewcommand*\theenumi{\textit{\Alph{enumi}}}
\renewcommand*\labelenumi{(\theenumi)}

In this section, we assume that the curves satisfy the following conditions (which are the conditions of Theorem~\ref{mitdSig}):
\begin{assumptions}
 Let $c:[a,b]\times[0,T)\to \Rzwei $ be a solution of (\ref{flow}),
 where the initial curve  satisfies the following four conditions: The curve $c_0$
\begin{enumerate}
 \item has positive curvature, $\kappa_0>0$ on $[a,b]$, \label{A}
\item  has no self-intersection, \label{B}
\item  is contained in the outer domain created by the convex support curve $\Sigma$,\label{C}
\item satisfies $L(c_0)<\Sig d = \min\{|x-y|: x,y \in\Sigma,\Sig\vec\tau(x)=-\Sig\vec\tau(y)\}$.\label{D}
\end{enumerate}
\end{assumptions}
\renewcommand*\theenumi{\textit{\roman{enumi}}}
\renewcommand*\labelenumi{\theenumi)}

\begin{lemma} \label{transv}
  Let $c:[a,b]\times[0,T)\to \Rzwei $ be a solution of (\ref{flow}),
  where the initial curve satisfies the conditions (\ref{A}), (\ref{B}), (\ref{C}) and (\ref{D}). 
  We assume furthermore that there is a time $\tilde t_0\in (0,T)$ such that the curve $c_{\tilde t_0}$ is not embedded. Then the self-intersection must be transversal, i.e.~the tangents at the intersection points are not parallel.
\end{lemma}

\begin{proof}
    The reason for the transversality is convexity combined with the result from Theorem~\ref{mitdSig}: $\intab\kappa\de s<2\pi$. Let $p_1<p_2$ be points such that $\ctnt(p_1)=\ctnt(p_2)$. We rotate the situation such that $\tau(p_1,\tilde\tn) = e_2$. If the self-intersection is tangential we have $\tau(p_2,\tilde\tn)= \pm e_2$. But $\tau(p_2,\tilde\tn)= + e_2$ cannot hold since then the angle of the tangent from $p_1$ to $p_2$ would already be $2\pi$. So we have $$\tau(p_2,\tilde\tn) = - e_2.$$ 
    As in the proof of Proposition~\ref{anfangstheorem} we have $\int_{p_1}^{p_2}\kappa\de s> \pi$ (replace $[a,b]$ by $[p_1,p_2]$ and  the proof works just the same). The curve $\ctnt|_{[p_1,p_2]}$ is smooth, regular and closed with exterior angle $\pi$. The formula for the index, Theorem~\ref{indexformel},  yields
    \begin{align*}
      k\defi\ind(\ctnt)= \frac{1}{2\pi}\int_{p_1}^{p_2}\kappa\de s|_{t=\tilde\tn}    + \frac{1}{2}
      \end{align*}
      and $k\in\mathbb Z$. As $\pi<\int_{p_1}^{p_2}  \kappa \de s|_{t=\tilde \tn} = 2k\pi - \pi$
      it follows that $k>1$. Therefore $k\geq 2$ and we get
	\begin{align*}
	2\pi>\intab\kappa \de s|_{t=\tilde \tn} \geq  \int_{p_1}^{p_2}\kappa \de s|_{t=\tilde \tn}= 2k\pi - \pi \geq 3\pi.
	\end{align*}
\end{proof}

\begin{lemma}\label{randschnitt}
  Let $c:[a,b]\times[0,T)\to \Rzwei $ be a solution of (\ref{flow}),
  where the initial curve satisfies  the conditions (\ref{A}), (\ref{B}), (\ref{C}) and (\ref{D}). 
  Under the assumption that the curves do not stay embedded, there is a first time $t_0\in (0,T)$ such that $c_{t_0}$ is not embedded, i.e.~there are $p_1,p_2\in\ab$ with $p_1<p_2$ such that $c_{t_0}(p_1)=c_{t_0}(p_2)$, and $c_t$ is embedded for $t\in[0,t_0)$. Furthermore, we have
\begin{align*}
 p_1 = a \ \text{ or } \ p_2 = b.
\end{align*}
This means that the first time a self-intersection happens is at the boundary.
\end{lemma}

\begin{proof} The existence of $t_0$ follows from the fact that the curves are immersed.
We now assume that the self-intersection happens at first in the interior: $p_1,p_2\in (a,b)$. For $\epsilon,\rho,\delta>0$ small enough, we define the function
\begin{align*}
 f:&(t_0-\epsilon,\tn + \epsilon)\times (p_1-\rho,p_1 + \rho)\times (p_2-\delta,p_2 + \delta) \to\Rzwei,\\
&f(t,p,\tilde p) \defi c(p,t)-c(\tilde p,t).
\end{align*}
This function is $C^\infty$ and satisfies $f(\tn,p_1,p_2)=0$ and
\begin{align*}
 \left(\partial_p f,\partial_{\tilde p} f\right)|_{(\tn,p_1,p_2)}=\left(\pp c(p_1,\tn), - \pp c(p_2,\tn)\right).
\end{align*}
  Lemma~\ref{transv} yields the transversality of the self-intersection. In other words, the vectors $\pm \tau(p_1,\tn)$ and $\pm \tau(p_2,\tn)$ are linearly independent. Thus, the vectors $\pp c(p_1,\tn)$ and $-\pp c(p_2,\tn)$ are linearly independent as well. 
  But that means that  $\left(\partial_p f,\partial_{\tilde p} f\right)|_{(\tn,p_1,p_2)}$ is invertible. By the implicit function theorem one gets open neighbourhoods $U,V_1,V_2$ of $\tn,p_1,p_2$, respectively, and a function $g\in C^1(U,V_1\times V_2)$ such that
\begin{align*}
 \left\{(t,p,\tilde p)\in U\times V_1\times V_2: f(t,p,\tilde p)=0\right\}= \left\{ (t,g(t)):t\in U\right\}.
\end{align*}
Hence there was a time before $\tn$ where the curve had a self-intersection. This is a contradiction.
\end{proof}

\begin{lemma} \label{noself} 
  Let $c:[a,b]\times[0,T)\to \Rzwei $ be a solution of (\ref{flow}),
  where the initial curve  satisfies the conditions (\ref{A}), (\ref{B}), (\ref{C}) and (\ref{D}). 
  If for all $t\in\nut$, the vector $\tau(a,t)$\footnote{Here, the vector $\tau(a,t)$ is seen as a vector with origin the point $c(a,t)$.} points into the quadrant 
  \begin{align*}
   Q_1\defi \left\{x\in\Rzwei: \langle x - c(a,t), c(b,t)-c(a,t)\rangle \leq 0 , \langle x - c(a,t), J( c(b,t)-c(a,t))\rangle \leq 0\right\}
  \end{align*}
  and $\tau(b,t)$ 
  points into the quadrant 
  \begin{align*}
   Q_2\defi \left\{x\in\Rzwei: \langle x-c(b,t), c(b,t)-c(a,t)\rangle \leq 0 , \langle x-c(b,t), J( c(b,t)-c(a,t))\rangle \geq 0\right\}
  \end{align*}
  then the curves cannot have self-intersections. The symbol $J$ denotes rotation by $+\frac{\pi}{2}$. See Figure~\ref{Bild8} for an illustration of $Q_1$ and $Q_2$.
\end{lemma}

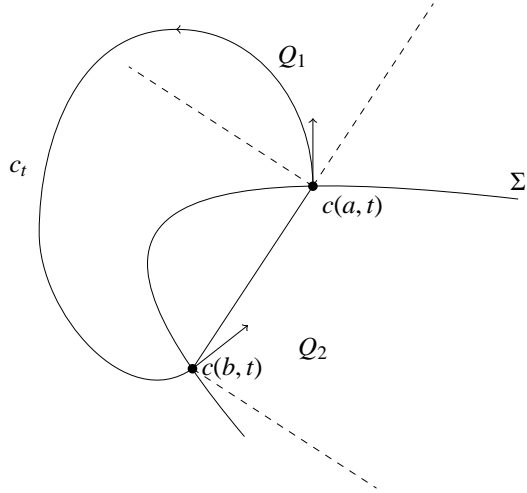
\begin{figure}
	  \begin{center}
	\scalebox{0.9}{ \begin{tikzpicture}
\draw (4,-0.5) .. controls (-2,0.2) and (-2.5,-1) .. (0,-4);
\coordinate[label=-55: ${c(a,t)}$] (a) at (1,-0.31);
\fill (a) circle (2pt);
\coordinate[label=right: ${c(b,t)}$] (b) at (-0.76,-3);
\fill (b) circle (2pt);
\draw (1,-0.31) to (-0.76,-3);
\draw[->] (1,-0.31) to (1,0.69);
\draw[->] (-0.76,-3) to (0.05, -2.36);
\draw[->,out=+90, in= 0] (1,-0.31) to (-1,2);
\draw[out=180, in= 90] (-1,2) to (-3,-1);
\draw[out=-90, in= 215] (-3,-1) to (-0.76,-3);
\draw[dashed] (1,-0.31) to (2.76,2.38);
\draw[dashed] (1,-0.31) to (-1.69,1.45);
\draw[dashed] (-0.76,-3) to (1.93,-4.76);
\coordinate[label = above: $\Sigma$] (C) at (4,-0.5);
\coordinate[label = above: $c_t$] (D) at (-3.3,-0.3);
\coordinate[label = above: $Q_1$] (E) at (0.7,1.3);
\coordinate[label = above: $Q_2$] (F) at (1,-3);
      \end{tikzpicture}}
	  \end{center}
	   \caption{Definition of $Q_1$ and $Q_2$ in Lemma~\ref{noself}.} \label{Bild8} 
\end{figure}

\begin{proof}
	We argue by contradiction. Due to Lemma~\ref{transv} and Lemma~\ref{randschnitt}, we can assume that there is a first time $\tn>0$ such that $\ctn$ has a transversal self-intersection at the boundary, $c(p_1,\tn)=c(p_2,\tn)$. 
	We only treat the case $p_1=a$, as the other case is proven analogously. If $p_2=b$ then $\intab\kappa\de s|_{t=t_0}\geq 3\pi$ since the index is an integer, a contradiction. Therefore, we have $p_2<b$. Since the self-intersection is transversal and $c_{\tn}$ is convex we have 
	  \begin{align}\label{groepi}
	  \winkel(\tau(a,\tn),\tau(p_2,\tn))>\pi.
	\end{align}
	We show that the tangent $\tau(\cdot,\tn)$ runs through at least a half-circle  from $p_2$ to $b$, which contradicts $\intab\kappa\de s< 2\pi$. We rotate and translate the situation such that 
      \begin{align*}
       c(a,\tn) = 0 \in \Sigma\ \  \mbox{ and } \frac{c(b,\tn)-c(a,\tn)}{|c(b,\tn)-c(a,\tn)|} = - e_1.
      \end{align*}
      By assumption, we then have 
      \begin{align*}
       & \langle \tau(a,\tn),e_1\rangle \geq 0,\ \ \langle \tau(a,\tn),e_2\rangle \geq 0, \\
	& \langle \tau(b,\tn),e_1\rangle \geq0, \ \ \langle \tau(b,\tn),e_2\rangle \leq0.
      \end{align*}
      By (\ref{groepi}) and since altogether we also have to stay smaller than $2\pi$, it follows that $\tau(p_2,\tn)$ points into
      \begin{align*}
       \left\{x\in\Rzwei: x^1> 0\right\} \cup  \left\{x\in\Rzwei: x^2 < 0\right\}.
      \end{align*}
      Now we distinguish two cases:\\
      \underline{Case 1}: $\langle \tau(p_2,\tn),e_1\rangle \geq 0$. In this case, the proof of Proposition~\ref{anfangstheorem} works, we only have to replace $a$ by $p_2$. It follows that there are points $u,v\in[p_2,b]$, $u\neq v$ such that $\tau(u,\tn)=-e_2$, $\tau(v,\tn)= e_2$, which implies $\pi\leq \int_{p_2}^b\kappa\de s|_{t=\tn}$.\\
      \underline{Case 2}:  $\langle \tau(p_2,\tn),e_1\rangle \leq 0$ and $\langle \tau(p_2,\tn),e_2\rangle < 0$. In this case, we modify the proof of Proposition~\ref{anfangstheorem}: we consider $c^2(\cdot,\tn)$ instead of $c^1(\cdot,\tn)$. We get $u,v\in[p_2,b]$ such that 
	\begin{align*}
       c^2(u,\tn)=\inf\{c^2(p,t): p\in[p_2,b]\} \mbox{ and } c^2(v,\tn)=\sup\{c^2(p,t): p\in[p_2,b]\}.
	\end{align*}
	 As in case 1 we prove that the angle between the tangents in $u$ and $v$ has to be at least $\pi$. It follows that $   \pi + \pi < \int_a^{p_2}\kappa\de s + \int_{p_2}^b\kappa \de s = \intab\kappa \de s<2 \pi,$
      which proves the lemma.

\end{proof}

\begin{proposition} \label{noself2}
 Let $c:[a,b]\times[0,T)\to \Rzwei $ be a solution of the area preserving curve shortening problem with Neumann free boundary conditions, where the initial curve  satisfies the conditions (\ref{A}), (\ref{B}), (\ref{C}) and (\ref{D}). Under the additional assumption 
\begin{align*}
 \intab\kappa\de s < \pi + \frac{\pi}{2} \  \ \ \ \ \ \forall t\in\nut
\end{align*}
the curves stay embedded.
\end{proposition}

\begin{proof}
  We study the geometric situation for any $t\geq 0$. By Theorem~\ref{mitdSig}, we know that $c(a,t)\neq c(b,t)$. We again rotate and translate the situation such that $  c(a,t) = 0 \in \Sigma$ and $\frac{c(b,t)-c(a,t)}{|c(b,t)-c(a,t)|} = - e_1$. 
The proof of Proposition~\ref{anfangstheorem} shows that we have
      \begin{align}\label{croco6}
       \langle \tau(a,t),e_1\rangle \geq 0 \ \ \mbox{ and }\langle \tau(b,t),e_1\rangle \geq0
      \end{align}
      by the boundary conditions of the flow. We now show that $\intab\kappa\de s< \pi + \frac{\pi}{2}$ yields
\begin{align}\label{croco5}
 \langle \tau(a,t),e_2\rangle \geq 0 \ \ \mbox{ and }\ \ \langle \tau(b,t),e_2\rangle \leq0.
\end{align}
Once more, the proof of (\ref{croco5}) works as that of Proposition~\ref{anfangstheorem}: Else, we may assume $\langle \tau(a,t),e_2\rangle < 0$ (the case $\langle \tau(b,t),e_2\rangle > 0$ can be treated analogously). By convexity and $\tau(a,t)\neq  e_2$ it follows that there is a $\delta>0$ such that $c( (a,a+\delta),t)\subset \{x\in\Rzwei: x^1 > 0\}.$
Since $\ab$ is compact, there are $u,v\in\ab$ such that 
\begin{align*}
       c^1(u,t)=\inf\{c^1(p,t): p\in[a,b]\} \ \ \mbox{ and } \ \  c^1(v,t)=\sup\{c^1(p,t): p\in[a,b]\}.
 \end{align*}
      Therefore, we have $c^1(u,t) \leq c^1(b,t)<c^1(a,t)< c^1(v,t)$. 
%
%
      It follows that $u \in (a,b]$, $v\in(a,b)$ and $u\neq v$. 
      By the strict convexity and by the definition of $u$ and $v$ we have $ \tau(u,t)=-e_2$ and $\tau(v,t)=e_2$.
      Thus, we have shown that $\tau(\cdot,t)$ runs through at least $\frac{3}{4}$ of a circle (it starts at $a$ in the quadrant $\{x\in\Rzwei: x^1\geq0, x^2<0\}$, has to go up to $e_2$ and then down to $-e_2$) which means $\intab\kappa\de s\geq \tfrac{3}{2}\pi,$ a contradiction. The properties  (\ref{croco6}) and (\ref{croco5}) imply that $\tau(a,t)$ and $\tau(b,t)$ point into the ''good`` quadrants of $\Rzwei$. 
Hence, the assumptions of Lemma~\ref{noself} are satisfied and the proposition is proven.
\end{proof}

%

\begin{theorem} \label{noself3} 
 Let $c:[a,b]\times[0,T)\to \Rzwei $ be a solution of (\ref{flow}),
 where the initial curve satisfies the conditions (\ref{A}), (\ref{B}), (\ref{C}) and (\ref{D}). Then there is a constant $C=C(\Sigma)>0$ such that: If $c_0$ satisfies $L_0< C(\Sigma)$ then the curves stay embedded under the flow. It suffices to choose $C(\Sigma)=\frac{1}{2\sigmax}$ but this choice need not be optimal.
\end{theorem}

\begin{proof}
 The strategy of the proof is the following: We use the formula 
 \begin{align*}
 \intab \kappa \de s = \pi + \int_{a(t)}^{b(t)}\Sig\tilde \kappa \de s_{\ftilde}\leq \pi + (b(t)-a(t)) \Sig\kappa_{\text{max}},
\end{align*}
 which can be found in Theorem~\ref{satz12}. The difference $b(t)-a(t)$ is (due to arc length parametrization) the length of the part of $\Sigma$ which connects $c(a,t)$ and $c(b,t)$. First, we force the curve $L_0$ to be so short that the standard graph representation at a certain point on $\Sigma$ is well defined. 
  We then estimate $b(t)-a(t)$ by a constant $c$ times $L_0$, using the graph representation. Now, if $c\, L_0 \Sig\kappa_{\text{max}}<\frac{\pi}{2}$, then there cannot be a self-intersection owing to Proposition~\ref{noself2}.\\
  Since $\winkel\left(\Sig\vec\tau(c(a,t)),\Sig\vec\tau(c(b,t)) \right) \in [0,\pi)$ (cf. Theorem~\ref{mitdSig}, (\ref{wink})), there is a ''side of $\Sigma$`` where we can use the graph representation. Consider the line $L_t$ that is parallel to the line segment from $c(a,t)$ to $c(b,t)$ and that meets $\Sigma$ tangentially at the ''side`` of $\Sigma$, where the tangent between $c(a,t)$ and $c(b,t)$ is less than $\pi$.
  By $\Sig q$ we denote the (non-unique) point in $[\Sig a,\Sig b]$ such that $\fsig (\Sig q) \in L_t$. The intersection must be tangential. 
  Now, we are in the situation that the part of $\Sigma$ which goes from $\ftilde(a(t))$ to $\ftilde (b(t))$ is in the graph representation of $\Sigma$ at $\fsig(\Sig q)$, but only under the condition that $|c(b,t)-c(a,t)|\leq L_0$ is small enough. As in \cite[Lemma 7.4]{StahlDiss} we show that 
  $L_0<\frac{1}{2 \Sig\kappa_{\text{max}}}$ is sufficient to be in the graph representation. We furthermore have $\sqrt{1 + w'(q)^2}\leq 1 + \frac{1}{4}$ in this case, where $w$ denotes the graph function of the graph representation. This yields
%
\begin{align*}
 b(t)-a(t) &= \int_{P(c(a,t))}^{P(c(b,t))}\sqrt{1 + w'(q)^2}\de q
     \leq \left|P(c(a,t))-P(c(b,t))\right|\frac{5}{4} 
     \leq L_0 \frac{5}{4} 
    < \frac{\pi}{2\Sig\kappa_{\text{max}}}, 
\end{align*}
 where $P(c(a,t))$ and $P(c(b,t))$ are the projections of $c(a,t)$ and $c(b,t)$ onto $L_t$. 
See also Figure~\ref{Bild9}.
\begin{figure}
	  \begin{center}
\scalebox{0.8}{\begin{tikzpicture}
\draw[->,out=-100, in= 100] (-2,2.5) to (-2,1.5);
\draw[out=280, in= 180] (-2,1.5) to (0,0);
\draw[out=0, in= 210] (0,0) to (5,1.5);
\draw[out=30, in= -140] (5,1.5) to (6.5,2.5);
\draw (5,1.5) to (-2,1.5);
\draw (-3,0) to (6,0);
\draw[dashed] (-2,1.5) to (-2,0);
\draw[dashed] (5,1.5) to (5,0);
 \coordinate[label = left: $\Sigma$] (A) at (-2,2.5);
\coordinate[label = left: ${c(a,t)}$] (B) at (-2,1.5);
  \coordinate[label = right: $\ \ {c(b,t)}$] (C) at (5,1.5);
\coordinate[label = below: ${^\Sigma f(^\Sigma q)}$] (D) at (0,0);
\coordinate[label = below: ${P(c(a,t))}$] (E) at (-2,0);
\coordinate[label = below: ${P(c(b,t))}$] (F) at (5,0);
 \fill (B) circle (2pt);
 \fill (C) circle (2pt);
 \fill (D) circle (2pt);
 \fill (E) circle (2pt);
 \fill (F) circle (2pt);
 \end{tikzpicture}}
	  \end{center}
	   \caption{The graph representation around $\Sig f(\Sig q)$ in the proof of Theorem~\ref{noself3}.} \label{Bild9} 
\end{figure}
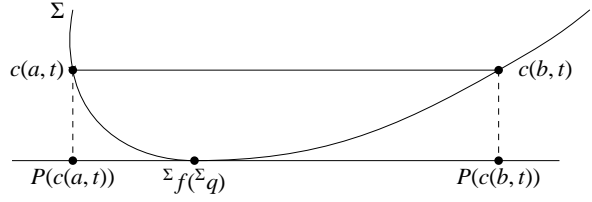
Combined with the remarks at the beginning of the proof we get
\begin{align*}
 \intab \kappa \de s \leq \pi + (b(t)-a(t)) \Sig\kappa_{\text{max}} < \pi + \frac{\pi}{2}
\end{align*}
under the condition $L_0<\frac{1}{2 \Sig\kappa_{\text{max}}} \eqqcolon C(\Sigma)$ and this shows the theorem due to Proposition~\ref{noself2}.
\end{proof}

\begin{theorem} \label{convexDom} 
  Let $c:[a,b]\times[0,T)\to \Rzwei $ be a solution of the area preserving curve shortening problem with Neumann free boundary conditions, where the initial curve satisfies the conditions (\ref{A}), (\ref{B}), (\ref{C}) and (\ref{D}). 
  Then there is a constant $C=C(\Sigma)>0$ such that: If $c_0$ satisfies $L_0< C(\Sigma)$, then, for each $t\in\nut$, the curve $\ct$ and the line segment from $c(b,t)$ to $c(a,t)$ trace out a convex domain in $\Rzwei$. 
\end{theorem}

\begin{proof}
 We use the constant $C(\Sigma)=\frac{1}{2 \Sig\kappa_{\text{max}}}$ from Theorem~\ref{noself3}. There, it is shown that the condition $L_0< C(\Sigma)$ implies embeddedness of $c_t$ and $\intab\kappa\de s<\pi + \frac{\pi}{2}$ for all $t\in\nut$. In the proof of Proposition~\ref{noself2}, we show that this implies that $\tau(a,t)$ points into $Q_1$ and $\tau(b,t)$ points into $Q_2$, where $Q_1$ and $Q_2$ were defined in Lemma~\ref{noself}, see also Figure~\ref{Bild8}. But this means that the exterior angle of the closed curve, which one gets by assembling the curve $c_t$ and the line segment from $c(b,t)$ to $c(a,t)$ are in $[0,\frac{\pi}{2}]$.\\
Any piecewise smooth, simple closed, positively oriented planar curve with $\kappa\geq 0$ and with exterior angles in $[0,\pi)$ traces out a convex domain.
We only have to prove that $c_t$ has no intersection with the line segment from $c(b,t)$ to $c(a,t)$. This works again by contradiction and ''counting angles`` as in the proof of Proposition~\ref{anfangstheorem}, \ref{noself} or \ref{noself2}. 
\end{proof}

\section{Finite type II singularities} \label{S5}
We consider the situation of a finite type II singularity, that means $T<\infty$ and
\begin{align*}
 &\sup_{p\in[a,b]}|\kappa|(p,t)\to\infty \quad (t\nearrow T)\quad\mbox{ and }\\
  &\sup_{p\in[a,b]}\left(|\kappa|^2(p,t)(T-t)\right) \mbox{ is unbounded.}
\end{align*}
We use the rescaling of Hamilton in his work on Ricci flow, see \cite{Hamilton2}: 
 For $j\in\mathbb N$ choose $t_j\in [0,T-\frac{1}{j}]$ and $p_j\in[a,b]$ such that
\begin{align*}
 |\kappa|^2(p_j,t_j) \left(T-\tfrac{1}{j} - t_j\right) = \max\left\{\left(|\kappa|^2(p,t) \left(T-\tfrac{1}{j} - t\right)\right) : t\in[0,T-\tfrac{1}{j}], p\in[a,b]\right\}.
\end{align*}
Then define $Q_j\defi |\kappa|(p_j,t_j)$ and 
\begin{align*}
\tilde\gamma_j(\cdot,\tau)\defi Q_j\left(c(\cdot,\tfrac{\tau}{Q_j^2} + t_j) - c(p_j,t_j)\right) \mbox{ for } \tau\in[-Q_j^2t_j,Q_j^2(T-t_j-\tfrac{1}{j})] \mbox{ on }[a,b].
\end{align*}
The following properties are known, see \cite[Lemma~4.3, Lemma~4.4]{HuiskenSinestrari}:

\begin{lemma} \label{lemmaconvZUS}
 Let $c:[a,b]\times[0,T)\to \Rzwei$ be a solution of (\ref{flow})
 developing a singularity of type II at $T<\infty$. Then we have with the notation $\tilde M^j_\tau \defi \tilde\gamma_j([a,b],\tau)$
\begin{enumerate}
 \item $Q_j^2(T-t_j-\tfrac{1}{j}) \to + \infty\ \ (j\to\infty),$ \label{conv1}
\item $ Q_j\to + \infty\ \ (j\to\infty)$,  \label{conv2}
\item $-Q_j^2t_j\to-\infty \ \ (j\to\infty),$\label{conv3}
\item $t_j\nearrow T\ \ (j\to\infty).$ \label{conv4}
\item $ 0\in\tilde M^j_0 \quad\forall j\in\mathbb N,$ \label{re17} 
\item $ |\tilde\kappa_j|(p_j,0)=1 \quad \forall j\in\mathbb N,$ \label{re18}
\item $ |\tilde\kappa_j|(\cdot,\tau)\leq 1 \quad \forall \tau\leq 0,$ \label{re19}
\item $ |\kappa|(p_j,t_j) = \max\{|\kappa|(p,t): p\in[a,b], t\in [0,t_j]\} \quad \forall j\in\mathbb N$, \label{re20}
\item $ \forall \epsilon>0 \ \forall \bar \tau>0 \ \exists j_0(\epsilon,\bar \tau)\in\mathbb N,\  \forall j\geq j_0: 
|\tilde\kappa_j|^2(p,\tau) \leq 1 + \epsilon \ \ \ \forall \tau \in [-Q_{j_0}^2t_{j_0},\bar \tau], \forall p\in[a,b]$. \label{re21}
\end{enumerate}
\end{lemma}

\begin{proposition}\label{limitflow} Consider a convex initial curve $c_0$.
  Let $c:[a,b]\times[0,T)\to \Rzwei $ be the solution of (\ref{flow})
  with a singularity at time  $T<\infty$, where the singularity is of type II and with $L(\ct)\geq c_1>0$ and $|\bar\kappa(t)|\leq c_2$,
  where $c_1$ and $c_2$ are constants independent of $t$. Consider the Hamilton-rescaled solution of the area preserving curve shortening flow with Neumann free boundary conditions 
\begin{align*}
 \tilde\gamma_j: [a,b]\times[-Q_j^2t_j,Q_j^2(T-t_j-\tfrac{1}{j})] \to\Rzwei.
\end{align*}
Then there exist reparametrizations $\psi_j: I_j \to [a,b]$ with $|I_j|\to\infty$ ($j\to\infty$) such that a subsequence of the rescaled curves
\begin{align*}
 \gamma_j\defi\tilde\gamma_j(\psi_j,\cdot): I_j\times [-Q_j^2t_j,Q_j^2(T-t_j-\tfrac{1}{j})]\to\Rzwei
\end{align*}
converges locally smoothly to a limit flow $\tilde\gamma_\infty: \tilde I \times (-\infty,\infty)\to \Rzwei$ (where $\tilde I$ is an unbounded interval containing $0$). The limit flow $\tilde\gamma_\infty$ is a smooth solution of the curve shortening flow and satisfies $0<\tilde\kappa_\infty\leq 1$ everywhere and $\tilde\kappa_\infty = 1$ at least at one point. If $\tilde M_\tau^\infty\defi\tilde\gamma_\infty (\tilde I,\tau)$ has a boundary, then $\partial\tilde M^\infty_\tau\subset \, \nix^\infty\Sigma$, where $ ^\infty\Sigma $ is a line through $ 0\in\Rzwei$, and $\langle \tilde\nu_\infty, ^{ ^\infty\Sigma}\nu\rangle = 0$ on $\partial\tilde M_\infty$. 
\end{proposition}

\begin{proof}
The convergence is proven like in Proposition~\ref{ParabolicGrenzuebergangII}. We only use the bounds from Lemma~\ref{lemmaconvZUS} (\ref{re21}) instead of the bounds on the curvature that come from the type I hypothesis. Furthermore, 
Lemma~\ref{lemmaconvZUS} (\ref{re17})  implies $\gaj(0,0) = \gaj(\varphi_j(p_j),0) = 0 \ \ \forall j\in\N,$
  which yields $$|\gaj(0,\tau)|\leq 0 +\big|\int_0^\tau \kappa_{\gaj}(\kappa_{\gaj}-\bar\kappa_{\gaj})\big|\leq 2\bar\tau(1+\epsilon)$$ on $[-\bar\tau,\bar\tau]$. We now choose $B_{R_k}\nearrow \tilde I\defi \lim_{j\to\infty}\Ij$ and $\bar\tau_k\nearrow \infty$ ($k\to\infty$). 
As in Proposition~\ref{ParabolicGrenzuebergangII} we get a subsequence $\gamma_{j_l}$ that converges locally smoothly to a limit flow $\gainfty:\tilde I \times \R\to\Rzwei$. This flow satisfies
\begin{align*}
 &0\leq\tilde\kappa_\infty\leq 1,\\
&\tilde \kappa_\infty (0,0) = 1  \qquad \ (\text{because }\kappa_{\gaj}(0,0) =1  \ \ \forall j\in\N)\\
&|\gainfty'(\cdot,0)|= 1\\
&\partial_\tau\gainfty = \tilde\kappa_\infty \tilde\nu_\infty  \qquad (\text{because }|\bar\kappa_{\gaj}| \leq \frac{1}{Q_j}c_2 \to 0).
\end{align*}
The statements about the boundary of  $\partial\tilde M^\infty_\tau$ are proven like the corresponding claims in Proposition~\ref{ParabolicGrenzuebergangII}. 
  The curvature $\tilde\kappa_\infty$ satisfies $\partial_t \tilde\kappa_\infty = \partial^2_s \tilde\kappa_\infty + \tilde\kappa_\infty^3$ (proven as in Lemma~\ref{evolution}). A maximum principle argument together with the Hopf lemma as in Corollary~\ref{kappagroesser0} yields $\tilde\kappa_\infty>0$.
\end{proof}

\begin{corollary}
  Let $c_0$ be a convex initial curve. Consider the solution $c:[a,b]\times[0,T)\to \Rzwei $ of (\ref{flow}).
  Assume that the flow satisfies $ L(\ct)\geq c_1>0$ and $ |\bar\kappa(t)|\leq c_2$
  where $c_1$ and $c_2$ are constants independent of $t$. If the flow develops a singularity of type~II at $T<\infty$, then the limit flow of Proposition~\ref{limitflow} is (after rotation and translation) the ``grim reaper'', which is the flow of curves given by $x=-\log\cos y + \tau$ for $y\in (-\frac{\pi}{2},\frac{\pi}{2})$.
\end{corollary}

\begin{proof}
  By Proposition~\ref{limitflow}, the limit flow of the rescaled flow is an eternal solution of the curve shortening flow. The limit curves are complete, the flow has bounded and positive curvature, and the maximum value of the curvature is attained at least at one point. By \cite[Theorem~1.3]{Hamilton}, 
  the limit flow must be a translating solution and the only translating solution in the case of curves is the ``grim reaper'', see for example \cite[Example 3.2]{RitoreSinestrari}.
\end{proof}

\begin{remark}
In Theorem~\ref{Linftyabschaetzung}, we proved existence of such constants $c_1$ and $c_2$ as required in the previous corollary. Therefore, we can apply the previous corollary to the situation of these theorems. But the conditions there probably do not prevent the flow to develop a singularity. Therefore, in order to exclude also type II singularities, we have to assume stronger conditions. In Section~\ref{geometricproperties}, we get more information about the flow, namely, when the curves stay embedded. We use this in our considerations of the type~II case.
\end{remark}

\begin{theorem}  \label{typeIIinnen} 
  Let $c:[a,b]\times[0,T)\to \Rzwei $ be a solution of the area preserving curve shortening problem with Neumann free boundary conditions, where the initial curve satisfies the following four conditions: 
  The curve $c_0$ has positive curvature, it has no self-intersection, it is contained in the outer domain created by the convex support curve $\Sigma$ and it satisfies $L(c_0)<\frac{1}{2\sigmax}$. Then the flow cannot develop a singularity of type II in the inside of the flow in finite time. This means that if there is a singularity of type II at $T<\infty$, then the limiting flow in Proposition~\ref{limitflow} has a boundary, i.e.~the rescaled and reparametrized flow converges to ``half a grim reaper''.
\end{theorem}

\begin{proof}
 As $\Sigma$ is a convex Jordan curve it encloses a convex domain $G_\Sigma$. The smallest osculating circle of $\Sigma$ has radius $\Sig r\defi\frac{1}{\sigmax}$, it touches $\Sigma$ in the point, where $\sigmax$ is attained and it is inside of $G_\Sigma$. 
 Since $\Sig d\defi \min\{|x-y|: x,y \in\Sigma,\Sig\vec\tau(x)=-\Sig\vec\tau(y)\}$ is the smallest distance of two parallel lines that touch $\Sigma$ we have $\Sig r < \Sig d$, and, of course, $\frac{1}{2}\Sig r < \Sig d$. It follows that the conditions of Theorem~\ref{convexDom} are satisfied and thus the curves $\ct$ together with the line segment connecting $c(b,t)$ and $c(a,t)$ trace out a convex domain $D_t$ for every $t\in\nut$. 
 Since $\Sigma$ is convex and $\ct$ sits on the outside of $\Sigma$, we cleary have $A_t\defi A(c_t,\tilde\gamma_t)=A_0\leq A ( D_t),$ where $A(D_t)$ denotes the area of $D_t$.\\ 
 We now assume that we have a singularity of type II in the inside, i.e.~the blowup procedure yields a limit flow $\gainfty$ without boundary, which is the ``grim reaper''. We consider the situation $\tau=0$ and omit the notation of $\tau$ from now on. After rotation we are in the situation that the two asymptotic lines of $\gainfty$ are the lines in direction of $e_1$ with height $-\frac{\pi}{2}$ and $\frac{\pi}{2}$. We choose points $p^1,p^2\in K\subset\subset\tilde I$ such that the tangents of $\gainfty$ at that points are almost $-e_1$ and $e_1$, i.e.~for $\epsilon>0$ arbitrary small we get 
 \begin{align*}
  \tilde\tau_\infty(p^1)&=(\cos \varphi^1,\sin\varphi^1) \text{ with } \varphi^1=\pi+\epsilon \text{ and analogously }\\
  \tilde\tau_\infty(p^2)&=(\cos \varphi^2,\sin\varphi^2)  \text{ with } \varphi^2=-\epsilon.
 \end{align*}
 Locally around $0\in K\subset\subset\tilde I$ the curves $\gaj$ look like $\gainfty$ for big $j$. In particular, there is $j_0(\epsilon)\in\N$ such that 
 \begin{align*}
  \tau_j(0) = (\cos \varphi_j,\sin \varphi_j) \text{ with } \varphi_j\in (\tfrac{3}{2}\pi-\epsilon,\tfrac{3}{2}\pi + \epsilon) \text{ for all } j\geq j_0,
 \end{align*}
i.e.~the tangent in $0$ is near $-e_2$. Furthermore, we choose $j_1(\epsilon)\geq j_0$ such that for all $j\geq j_1$
\begin{align*}
 \tau_j(p^1)&=(\cos \varphi^1_j,\sin \varphi^1_j)  \text{ with } \varphi^1_j\in (\pi,\pi + 2\epsilon),\\
 \tau_j(p^2)&=(\cos \varphi^2_j,\sin \varphi^2_j)\text{ with } \varphi^2_j\in (-2\epsilon,0),
\end{align*}
i.e.~the tangents are near $e_1$ and $-e_1$ at the points. We also have $0<\gaj^2(p^1)\leq \frac{\pi}{2} + \epsilon$ and $0>\gaj^2(p^2)\geq -\frac{\pi}{2} - \epsilon$ for $j\geq j_2\geq j_1$. We now consider the blowdown of the curves $\gaj$ for $j\geq j_2$, that means we consider
 \begin{align*}
  \gaj\mapsto \frac{1}{Q_j}\gaj  + c(p_j,t_j)
		= c(\psi_j,t_j).
 \end{align*}
 \begin{figure}
	  \begin{center}
 \scalebox{0.8}{ \begin{tikzpicture}
 \draw (-1,0) to (12,0);
 \draw (0.2,2) to (-0.4,-2.5);
  \draw (-0.5,1) to (11.5,1.3);
  \draw (-0.5,-1.1) to (11.5,-1.5);
  \draw (11,2) to (11,-2);
     \draw[out=60,in=-90] (10.8,-0.4) to (11,0.4);
         \draw[out=90,in=-3] (11,0.4) to (10,1.00);
       \draw[out=177,in=0] (10,1.00) to (2,1.05);
     \draw[out=180,in=83] (2,1.05) to (-0.051,0);
   \draw[out=263,in=179] (-0.051,0) to (2,-1.175);
      \draw[out=-1,in=180]  (2,-1.175) to (9,-1.3);
            \draw  (10.8,-0.4) to (9,-1.3);
	    
                    
                   \fill (10.8,-0.4) circle (2pt);
                \coordinate[label = -30: ${c(a,t_j)}$] (C) at (10.8,-0.4);
                  \fill (1.8,1.05) circle (2pt);
            \coordinate[label = 30: ${c(\psi_j(p^1),t_j)}$] (D) at (1.3,1.05);
         \fill (-0.054,0) circle (2pt);
        \coordinate[label = -80: ${c(\psi_j(0),t_j)}$] (B) at (-2,0);
        \fill (2,-1.175) circle (2pt);
     \coordinate[label =-30: ${c(\psi_j(p^2),t_j)}$] (D) at (1.5,-1.175);
                  \fill (9,-1.3) circle (2pt);
              \coordinate[label = -30:  ${c(b,t_j)}$] (H) at (9,-1.3);
            \coordinate[label =30: ${D_{t_j}}$] (D) at (3,-1);
              \coordinate[label =below: ${T_j}$] (D) at (11.5,-1.5);
              \draw[thick,decorate,decoration={brace,amplitude=5pt}]
(-2,-1.2) -- (-2,1.2);
 \coordinate[label = left: ${\sim\frac{\pi}{ Q_j}} $] (G) at (-2.2,0);
                                     
 \end{tikzpicture}
 }
 
	  \end{center}
	   \caption{The almost-trapezoid $T_j$ in the proof of Theorem~\ref{typeIIinnen}} \label{Bild13} 
\end{figure}
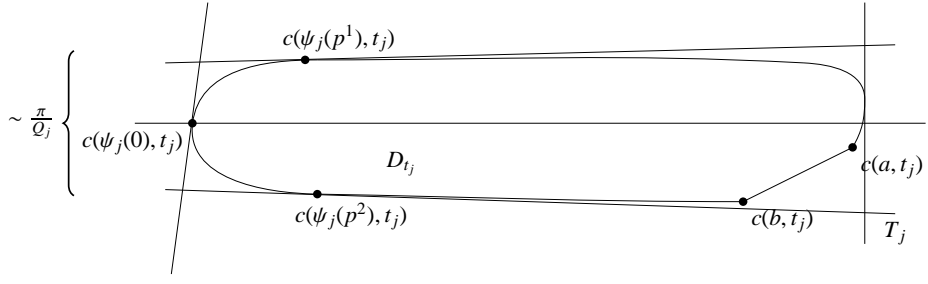
As we only rescaled and translated the curves $\gaj$, the properties of the tangents (they are near $-e_2,e_1,-e_1$ at some points) are still the same. The only difference is that we have $0<\gaj^2(p^1)\leq \frac{\frac{\pi}{2} + \epsilon}{Q_j}$ and $0>\gaj^2(p^2)\geq \frac{-\frac{\pi}{2} - \epsilon}{Q_j}$ by rescaling. We now use the fact that a convex domain always lies in the half space on the side of the tangent line, into which the inner normal is pointing. We use this here to ``frame'' the domain $D_{t_j}$. With the three tangent lines at $p^1,0,p^2$ we get an unbounded almost-trapezoid $\tilde T_j$ (still open on the right) with $D_{t_j}\subset \tilde T_j$. We cut the almost-trapezoid at the point, where $\max c^1(\psi_j,t_j)$ is attained and get an almost-trapezoid $T_j$ as in Figure~\ref{Bild13} (if $\tau(\psi_j(0),t_j)=-e_2$, it is a trapezoid). Because of $L(c(\psi_j,t_j))\leq L_0$ we get that the length of the trapezoid (in direction of $e_1$) is bounded by $L_0$. The area of $T_j$ is almost the area of the rectangle which has length $L_0$ and height $\frac{\pi}{Q_j}$. A simple estimate would be $A\left(T_{j}\right)\leq 2 L_0 \frac{2\pi}{Q_j}$ for $j\geq j_2(\epsilon)$, $\epsilon$ small. We then have for $j\geq j_3\geq j_2$ the contradiction
$$A_0\leq A\left(D_{t_j}\right)\leq A\left(T_{j}\right)\leq 4L_0 \frac{\pi}{Q_j} <A_0.$$
\end{proof}

%
%

\begin{theorem} \label{typeIIrandneu}
  Let $c:[a,b]\times[0,T)\to \Rzwei $ be a solution of the area preserving curve shortening problem with Neumann free boundary conditions, where the initial curve $c_0$ satisfies the following four conditions: 
  \begin{itemize}
   \item The curve $c_0$ has positive curvature, $\kappa_0>0$ on $[a,b]$,
   \item it has no self-intersection
   \item and it is contained in the outer domain created by the convex support curve $\Sigma$.
  \end{itemize}
  Consider the quotient $\frac{A_0}{L_0^2}\defi c_I$. There exists a constant $C=C(c_I,\Sigma)$ such that: If $L_0<C$ then the flow cannot develop a singularity in finite time, i.e.~$T=\infty$. 
  It suffices to choose $C(\Sigma,c_I)=\frac{4}{5 \sigmax}\arcsin\left( c_I\right)$ but this choice need not be optimal.
\end{theorem}

\begin{remark}
\begin{enumerate}
 \item  There is an isoperimetric inequality for the case that the initial curve is outside of the convex domain $G\subset \Rzwei$ traced out by $\Sigma$ and under the condition that $c_0(a)\not = c_0(b)$. One defines as in \cite[Chapter 4]{Andre} a multiplicity function $i$ on the union of the disc-type domains traced out by the closed curve $c_0-\tildega_0$. Using the results in \cite{Choe} it can be shown that if $i\equiv 0$ on $G$ (what is true because $c_0(a)\not = c_0(b)$) then $i\in BV(\Rzwei)$ and \begin{align*}
c_{\mathbb H}\,|A(c_0,\tildega_0)|\leq L(c_0)^2,
\end{align*}
where $c_{\mathbb H}= 2\pi$ is the isoperimetric constant of the half-circle at a straigth line. This implies $\frac{A_0}{L_0^2}\leq \frac{1}{2\pi}$. Since $\frac{4}{5}\arcsin\left(\frac{1}{2\pi}\right)<\frac{1}{2}$, the condition $L(c_0)< \frac{4}{5 \sigmax}\arcsin\left( c_I\right)$ implies $L_0< \frac{1}{2\sigmax}$, the condition in Theorem~\ref{typeIIinnen}.
 \item As the length $L_0$ appears in both sides of the inequality $L_0< \frac{4}{5 \sigmax}\arcsin\left( \frac{A_0}{L_0^2}\right)$ it makes sense to ask if such a condition is satisfiable. But the quotient $\frac{A_0}{L_0^2}$ has no unit, it descibes the ``shape'' of the domain $G_0$ that is traced out by $c_0$ and a part of $\Sigma$ and that satisfies $A(G_0)=A_0$. The condition $L_0< \frac{4}{5 \sigmax}\arcsin\left( \frac{A_0}{L_0^2}\right)$ then means that the curve $c_0$ has to be short compared to the term $\frac{1}{\sigmax}$ in an appropriate way according to the shape of $G_0$.
\end{enumerate}
\end{remark}

\begin{proof}
 The case of a finite type I singularity is ruled out by combining Proposition~\ref{Laengenachuntenabsch}, Theorem~\ref{Linftyabschaetzung} and Theorem~\ref{nottypeI} (notice $L_0< C(c_I,\Sigma)\Rightarrow L_0<\frac{1}{2\sigmax}\Rightarrow L_0<\Sig d$ with the remark above).
 By Theorem~\ref{typeIIinnen}, a singularity of type II must form at the boundary, i.e.~the reflected limit flow is the ``grim reaper'' which implies that original limit flow of Proposition~\ref{limitflow} $\gainfty$ is ``half the grim reaper'' by symmetry, i.e.~the curves given by $x=-\log\cos y + \tau$ on $(-\frac{\pi}{2},0]$ (if the singularity forms near $b$) or on $[0,-\frac{\pi}{2})$ (if the singularity forms near $a$). We treat the case that $\tilde I =(-\infty,\tilde b]$, where $\tilde b\geq 0$, which means that the singularity shows up near $b$. We have $\tilde\kappa_\infty(\tilde b,\tau)=1$ by symmetry.
 
  We want to use a similar argument as in the proof of Theorem~\ref{typeIIinnen}. But this only works up to a certain point. We get one almost vertical and one almost horizontal tangent line to bound the area of the convex domain which comes up in Theorem~\ref{convexDom}. The bound $L\left(\ct\right)\leq L_0$ gives the second vertical line. One almost horizontal line is still missing. But if the curves are short enough with respect to $c_I$, we can also estimate the rest of the area of the convex domain.
  
 After rotation we are in the situation that $\nix^\infty\Sigma=\R\times \{0\}$ and the asymptotic line to half the grim reaper is the line in direction of $e_1$ with height $\frac{\pi}{2}$. We consider $\tau=0$ and omit the notation of $\tau$ from now on. For $\epsilon>0$ we choose $p_1\in \tilde I$ such that 
 \begin{align*}
  \tilde\tau_\infty(p_1) = (\cos\varphi_1,\sin\varphi_1) \text{ with } \varphi_1\in (\pi,\pi + \epsilon).
 \end{align*}
 By local smooth convergence of $\gaj$ (the rescaled and parametrized curves as in Proposition~\ref{limitflow}) to $\gainfty$ we then find $j_0(\epsilon)\in\N$ such that 
 \begin{align*}
  \tau_j(p_1) = (\cos \varphi_j,\sin\varphi_j) \text{ with } \varphi_j\in (\pi,\pi + 2\epsilon) \text{ for all } j\geq j_0.
 \end{align*}
We also have 
\begin{align*}
 \tau_j(\tilde b) =  (\cos \varphi^0_j,\sin\varphi^0_j)\text{ with } \varphi^0_j\in (\tfrac{3}{2}\pi - \epsilon,\tfrac{3}{2}\pi + \epsilon)\text{ for all } j\geq j_1\geq j_0
 \end{align*}
 because $\tilde\tau_\infty(\tilde b)=-e_2$. 
%
\begin{figure}
	  \begin{center}
  \scalebox{0.8}{\begin{tikzpicture}
 \draw (-1,0.02) to (12,-0.35);
 \draw (0.05,2) to (-0.15,-1.5);
   \draw (-0.054,0) to (10,-3);
    \draw (12,0) to (10,-3);
     \draw (-0.5,1) to (11.5,1.3);
    \fill[gray!30] (11.755,-0.352)-- (10,-3) -- (-0.054,0);
 \draw[out=100,in=-35] (11,-5) to (10,-3);
 \draw[out=55,in=-90] (10,-3) to (10.6,-1);
\draw[out=90,in=-7] (10.6,-1) to (8,0.6);
    \draw[out=173,in=0] (8,0.6) to (2,1.05) ;
     \draw[out=180,in=89] (2,1.05) to (-0.051,0);
          \draw[out=-1,in=145] (-0.051,0) to (10,-3);
       \draw[out=15,in=179] (-2,-0.3) to (-0.051,0);
       \coordinate[label = below: ${\Sigma}$] (F) at (-2,-0.3);
                     \coordinate[label = right: ${T_{c(b,t_j)}\Sigma}$] (E) at (12.1,-0.4);
                   \fill (10,-3) circle (2pt);
                \coordinate[label = right: ${c(a,t_j)}$] (C) at (10,-3);
                \coordinate[label = above: ${\alpha_j+\frac{\pi}{2}}$] (G) at (9.6,-2.6);
                  \fill (1.8,1.05) circle (2pt);
            \coordinate[label = 30: ${c(\psi_j(p^1),t_j)}$] (D) at (1.3,1.05);
         \fill (-0.054,0) circle (2pt);
        \coordinate[label = -80: ${c(b,t_j)}$] (B) at (-1.5,0.6);
          \coordinate[label=right: ${\beta_j}$] (G) at (2,-0.4);
            \coordinate[label =30: ${D_{t_j}}$] (D) at (3,0.2);
             \draw[thick,decorate,decoration={brace,amplitude=5pt}]
(-1.8,0) -- (-1.8,1.1);
 \coordinate[label = left: ${\sim\frac{\pi}{2 Q_j}} $] (G) at (-2.1,0.55);
 \end{tikzpicture}
 }
 
	  \end{center}
	  \caption{The situation in the proof of Theorem~\ref{typeIIrandneu}}
	  \label{Bild15} 
\end{figure}
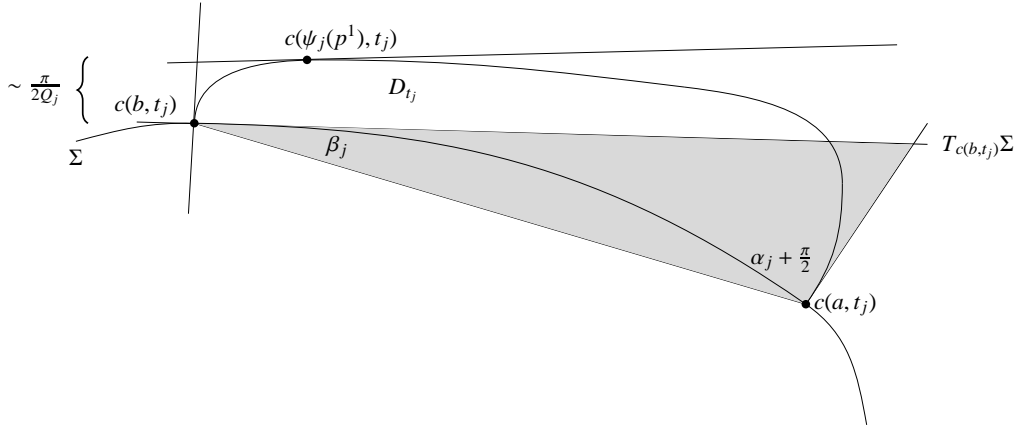
 Since $I_j=[a_j,b_j]\to(-\infty,\tilde b]$ and $$\gainfty(\tilde b) \in  \, \nix^\infty\Sigma, \ \ \ \langle \tilde\nu_\infty(\tilde b), ^{ ^\infty\Sigma}\vec\nu(\gainfty(\tilde b))\rangle = 0$$ we also have $\tau_j(b_j) =  (\cos \varphi^2_j,\sin\varphi^2_j)$ with $\varphi^2_j\in (\frac{3}{2}\pi - 2\epsilon,\frac{3}{2}\pi + 2\epsilon)$ for all $j\geq j_2\geq j_1$. 
 We consider the blowdown
 \begin{align*}
    \gaj\mapsto \frac{1}{Q_j}\gaj  + c(p_j,t_j)
		= c(\psi_j,t_j).
 \end{align*}
The directions of the tangents have not changed, only the height of the tangent line in $c(\psi_j(p^1),t_j)$ over the tangent line of $\Sigma$ in $c(b,t_j)$ is now approximately $\frac{\pi}{2 Q_j}$, see Figure~\ref{Bild15}.
 As you can see there, the domain $D_{t_j}$, which is the convex domain traced out by $c(\psi_j,t_j)$ and the line segment from $c(b,t_j)$ to $c(a,t_j)$, is divided by $T_{c(b,t_j)}\Sigma$ into two parts. The ``upper'' part becomes small, when $j$ is big because of the height of the asymptotic line. But the part ``under'' $T_{c(b,t_j)}\Sigma$ (within the gray triangle) could still be big. To bound the area of this part we introduce the following angles:
\begin{align*}
 \alpha_j&\defi \winkel\left(\Sig\vec\tau(c(a,t_j)),\frac{c(b,t_j)-c(a,t_j)}{|c(b,t_j)-c(a,t_j)|}\right)\geq 0,\\
 \beta_j&\defi  \winkel\left(\frac{c(b,t_j)-c(a,t_j)}{|c(b,t_j)-c(a,t_j)|},\Sig\vec\tau(c(b,t_j))\right)\geq 0.
\end{align*}
It follows (see also the proof of Theorem~\ref{noself3}) that
\begin{align}\label{re22} \begin{split}
 \alpha_j + \beta_j &= \winkel\left(\Sig\vec\tau(c(a,t_j)),\Sig\vec\tau(c(b,t_j))  \right) \\
 &=\int_{a(t_j)}^{b(t_j)}\Sig\tilde \kappa \de s_{\ftilde}\leq(b(t_j)-a(t_j)) \Sig\kappa_{\text{max}}\leq \frac{5}{4}L_0\sigmax,
 \end{split}
\end{align}
where we used $L_0\leq \frac{1}{2\sigmax}$ for the last inequality. Now, we want to get the angles $\alpha_j,\beta_j$ so small, that ``the rest'' of the domain $D_{t_j}$ has small area: Consider the triangle $R_j$ which is given by the line segment from $c(b,t_j)$
 and $c(a,t_j)$ and by the angles $\beta_j$ in $c(b,t_j)$ and $\alpha_j + \frac{\pi}{2}$ in $c(a,t_j)$, see Figure \ref{Bild15}.
We get by a simple calculation (using $|c(b,t_j)-c(a,t_j)|\leq L_0$)
\begin{align*}
 A\left(R_j\right)& \leq \frac{1}{2}\left(L_0 \cos \beta_j + L_0 \sin\beta_j\tan(\alpha_j + \beta_j)\right)L_0\sin\beta_j \\ 
  &=  \frac{1}{2}L_0^2\sin\beta_j \left(\cos \beta_j +  \sin\beta_j\tan(\alpha_j + \beta_j)\right). 
\end{align*}
We define $T_j\defi D_{t_j}\setminus R_j$ and note $A\left(T_j\right)\leq \frac{2\pi}{Q_j}L_0$ for $j\geq j_3\geq j_2$, where we used that the ``length'' of $T_j$ cannot be bigger then $L_0$. The estimate $L_0<\frac{1}{2\sigmax}<\frac{\pi}{5\sigmax}$ and (\ref{re22}) imply $ \alpha_j + \beta_j\leq \frac{\pi}{4},$
which yields $\tan\left(\alpha_j + \beta_j\right)\leq 1$. Therefore, we have for $j\geq j_3$
\begin{align} \label{re29}
 \begin{split} A_0&\leq A\left(D_{t_j}\right)\leq A\left(T_j\right) + A\left(R_j\right)\\
  &\leq \frac{2\pi}{Q_j}L_0 + \frac{1}{2}L_0^2\sin\beta_j \left(\cos \beta_j +  \sin\beta_j\tan(\alpha_j + \beta_j)\right)\\
    &\leq \frac{2\pi}{Q_j}L_0  + L_0^2\sin\beta_j.  \end{split}
\end{align}
By assumption and by the monotonicity of $\arcsin$ on $[0,\frac{1}{2\pi}]$ 
there is an $\epsilon_0>0$ such that $L_0\leq \frac{4}{5 \sigmax} \arcsin\left(c_I-\epsilon_0\right)$. Inequality (\ref{re22}) implies that $\sin\beta_j\leq \sin\left(\frac{5}{4}\sigmax L_0\right)\leq c_I-\epsilon_0.$
By definition of $c_I$ and with (\ref{re29}), we then have
\begin{align*}
 A_0\leq \frac{2\pi}{Q_j}L_0  + A_0 - L_0^2\epsilon_0.
\end{align*}
For $j\geq j_4(\epsilon_0,L_0)\geq j_3$ it follows that $\frac{2\pi}{Q_j}< \epsilon_0L_0$ because of $Q_j\to \infty$, which yields the contradiction $A_0<A_0$.
\end{proof}

\section{Integral estimates and subconvergence} \label{S6}
In Theorem~\ref{typeIIrandneu} we found conditions under which the maximal time of extistence of our flow is infinity. 
Unfortunately, this does not automatically imply that the curvature is uniformly bounded (in $C^0$) in space and in $[0,\infty)$ because one could imagine a ``singularity at infinity'', that is $\max_{[a,b]}|\kappa(\cdot,t)|\to\infty$ as $t\to\infty$. 
In order to prove a convergence result we want to use the gradient estimates of the graph representation of Stahl \cite[Chapter~7]{StahlDiss}. But there, the condition was a uniform bound on the curvature in space and time, which we do not have a priori. 
We show integral estimates to overcome this problem.\\
In this section, we use the following notation: Let $f:I\to\Rzwei$ be a smooth, regular curve and $w:I\to\R$ a measurable function, then $\|w\|_p\defi \left(\int_f |w|^p\de s_f\right)^{\frac{1}{p}}$ denotes the $L^p$-norm of $w$ on $f$.
%
%

\begin{lemma}[Gagliardo-Nirenberg interpolation inequalities I]\label{gagli1a}
 Let $f:I\to\Rzwei$ be a smooth, regular curve with finite length $L$ and  $v:I\to\R$ a smooth function with $\int v\de s=0$, then there is a constant $c$ such that
  \begin{align*}
  \|\ps^n v\|_{p}\leq c\|\ps^m v\|_2^\sigma\|v\|_2^{1-\sigma},
 \end{align*}
 where $n \in\{0,\dots,m-1\}$ and $\sigma=\frac{n+\frac{1}{2}-\frac{1}{p}}{m}$ and $c=c(n,m,p,L)$.
\end{lemma}

\begin{proof}
See \cite[Theorem 3.70]{Aubin}. 
\end{proof}

\begin{corollary}[Gagliardo-Nirenberg interpolation inequalities II]\label{gagli2}
 Let $f:I\to\Rzwei$ be a smooth, regular curve with finite length $L$ and  $v:I\to\R$ a smooth function, then there are constants $c$ and $\tilde c$ such that
  \begin{align*}
  \|\ps^n v\|_{p}\leq c\|\ps^m v\|_2^\sigma\|v\|_2^{1-\sigma} + \frac{\tilde c}{L^{m\sigma}} \|v\|_2,
 \end{align*}
 where $n \in\{0,\dots,m-1\}$ and $\sigma=\frac{n+\frac{1}{2}-\frac{1}{p}}{m}$. The constants $c$ and $\tilde c$ only depend  on $n,m,p$ and $L$.
\end{corollary}

\begin{proof}
 Consider the function $\tilde v\defi v-\frac{\int v\de s}{L}$ and use Lemma~\ref{gagli1a}. 
\end{proof}

%

\begin{lemma} \label{intk2}
 Let $c:[a,b]\times[0,T)\to \Rzwei $ ($T\leq\infty$) be a solution of (\ref{flow}).
 Then we have for $t>0$
 \begin{align*}
  \dt \int(\kappa-\bar\kappa)^2\de s= 2 &\left[(\kappa-\bar\kappa)\ps\kappa\right]^b_a-2\int(\ps\kappa)^2\de s \\
	&+ \int(\kappa-\bar\kappa)^4\de s + 3\bar\kappa  \int(\kappa-\bar\kappa)^3\de s + 2\bar\kappa^2 \int(\kappa-\bar\kappa)^2\de s
 \end{align*}

\end{lemma}

\begin{proof}
We use the evolution equations from Lemma~\ref{evolution} and integration by parts and get
\begin{align*}
 \dt \int(\kappa-\bar\kappa)^2\de s &= 2\int(\kappa-\bar\kappa)\pt\diff\de s - \int\kappa \diff^3\de s\\
&=2\int\diff\left(\ps^2\kappa + \kappa^2\diff\right)\de s-2\pt\bar\kappa\int\diff\de s -  \int\kappa \diff^3\de s\\
&=2\left[(\kappa-\bar\kappa)\ps\kappa\right]^b_a-2\int(\ps\kappa)^2\de s + 2\int\kappa^2\diff^2\de s- \int\kappa \diff^3\de s.
\end{align*}
Now, we write
\begin{align*}
 \int\kappa^2\diff^2\de s&= \int\kappa^2\diff^2\de s -\bar\kappa \int\kappa\diff^2\de s +\bar\kappa \int\kappa\diff^2\de s\\
 &= \int\diff^4\de s + \bar\kappa\int\diff^3\de s + \bar\kappa \int \diff^3\de s + \bar\kappa^2\int\diff^2\de s,
\end{align*}
and, with the same trick,
\begin{align*}
 -\int\kappa\diff^3\de s = - \int \diff^4 \de s - \bar\kappa \int\diff^3\de s.
\end{align*}
Together with the calculation above we get the result.
\end{proof}

\begin{corollary} \label{cor4}
 Let $c:[a,b]\times[0,T)\to \Rzwei $ be a solution of (\ref{flow})
 with $|\bar\kappa(t)|\leq c_2<\infty$ uniformly in $t$. Then there are constants $C_1,C_2,C_3$, depending only on $c_2$ and $L_0$ such that for every $t>0$ we have
    \begin{align*}
     \dt\int(\kappa-\bar\kappa)^2\de s\leq C_1\left(\int(\kappa-\bar\kappa)^2\de s\right)^{3} + C_2 \left(\int(\kappa-\bar\kappa)^2\de s\right)^{\frac{5}{3}} + C_3 \int(\kappa-\bar\kappa)^2\de s
    \end{align*}

\end{corollary}

\begin{proof}
 Due to Lemma~\ref{lemmastahl5} we have for $t>0$
 \begin{align*}
  \left[(\kappa-\bar\kappa)\ps\kappa\right]^b_a = -(\kappa(b)-\bar\kappa)^2 \ \Sig\kappa\circ\fsig\circ c(b)  -(\kappa(a)-\bar\kappa)^2\ \Sig\kappa\circ\fsig\circ c(a) \leq 0.
 \end{align*}
We only have to handle the integrals $\int(\kappa-\bar\kappa)^4\de s $ and $\int(\kappa-\bar\kappa)^3\de s $ from the previous lemma. By Lemma~\ref{gagli1a} and by the Young inequality 
we have
\begin{align*}
 \|\kappa-\bar\kappa\|^4_4 &\leq c\|\ps\kappa\|_2\|\kappa-\bar\kappa\|_2^3
  \leq\frac{1}{2}\|\ps\kappa\|_2^2 + c\|\kappa-\bar\kappa\|_2^6 \\
  \|\kappa-\bar\kappa\|^3_3&\leq c\|\ps\kappa\|_2^\frac{1}{2}\|\kappa-\bar\kappa\|_2^{\frac{5}{2}}
  \leq\frac{1}{2}\|\ps\kappa\|_2^2 + c\|\kappa-\bar\kappa\|_2^{\frac{2\cdot 5}{3}} ,
\end{align*}
which implies the result with constants also depending on $L_0$.
\end{proof}

\begin{corollary} \label{intto0}
 Let $c:[a,b]\times[0,\infty)\to \Rzwei $ be a solution of (\ref{flow}) with
 $|\bar\kappa|\leq c_2 <\infty$, then 
 \begin{align*}
  \int(\kappa-\bar\kappa)^2\de s\to 0 \ \  (t\to\infty).
 \end{align*}

\end{corollary}

\begin{proof}
The proof works as in \cite[Section~8]{Mantepaper}:
 We argue by contradiction. We assume there is a sequence $t_j\to\infty$ and a $\delta>0$ such that $\int\diff^2 \de s|_{t=t_j}>\delta>0$ for all $j\in\N$. By Corollary~\ref{cor4}, we get $ \dt \int\diff^2\de s\leq \left(1 + \int\diff^2\de s\right)^3$
with a constant only depending on $c_2$ and $L_0$. With $f(t)\defi \int\diff^2\de s $ we have
\begin{align*}
 -\dt\left(\frac{1}{(1 + f(t))^2}\right) \leq C.
\end{align*}
We integrate from $t<t_j$ to $t_j$ and get $  \frac{1}{(1+ f(t))^2}< C(t_j-t) +  \frac{1}{(1+ \delta)^2}$
and thus
\begin{align*}
 (1+f(t))^2\geq \frac{1}{C(t_j-t) +  \frac{1}{(1+ \delta)^2}}.
\end{align*}
Then we find an $\epsilon>0$ small and not depending on $j$ such that for all $t\in [\tj-\epsilon,\tj]$ 
\begin{align}\label{wid}
 \int\diff^2\de s = f(t)\geq \frac{\delta}{2}.
\end{align}
But on the other hand we have that $\dt L(\ct) = -\int\diff^2\de s$
and therefore $ \int_0^\infty\int\diff^2\de s\de t\leq L_0,$
which contradicts (\ref{wid}).
\end{proof}

\begin{lemma}
 Let $c:[a,b]\times[0,T)\to \Rzwei $ ($T\leq\infty$) be a solution of (\ref{flow}).
 For a fix $\tau>0$ we define $\mu\defi t-\tau$. Then we have for $t>0$
 \begin{align*}
  \dt\left(\mu\int(\ps\kappa)^2 \de s\right)= \int(\ps\kappa)^2\de s &+ 2\mu\left[\ps\kappa\,\ps^2\kappa\right]_a^b - 2\mu \int (\ps^2\kappa)^2\de s\\
 & + 5\mu\int(\ps\kappa)^2\kappa\diff\de s + 2\mu\int\kappa^2(\ps\kappa)^2\de s.
 \end{align*}
\end{lemma}

\begin{proof}
 With the evolution equation from Lemma~\ref{evolution} we get (note $\dt\mu =1$)
 \begin{align*}
  \dt &\left(\mu\int(\ps\kappa)^2 \de s\right)= \int(\ps\kappa)^2\de s + 2\mu\int\ps\kappa\,\pt\ps\kappa\de s - \mu\int(\ps\kappa)^2\kappa\diff\de s\\
  &= \int(\ps\kappa)^2\de s + 2\mu\int\ps\kappa\,(\ps\pt\kappa + \kappa\diff\ps\kappa)\de s   - \mu\int(\ps\kappa)^2\kappa\diff\de s\\
  &= \int(\ps\kappa)^2\de s + 2\mu\int\ps\kappa\,\ps(\ps^2\kappa + \kappa^2\diff)\de s \\
      &\qquad+ 2\mu \int(\ps\kappa)^2\kappa\diff\de s - \mu\int(\ps\kappa)^2\kappa\diff\de s\\
      &= \int(\ps\kappa)^2\de s + 2\mu\int\ps\kappa\,\ps^3\kappa \de s + 2\mu\int \ps\kappa\, (2\kappa\ps\kappa\,\diff + \kappa^2\ps\kappa)\de s \\
      &\qquad+ 2\mu \int(\ps\kappa)^2\kappa\diff\de s - \mu\int(\ps\kappa)^2\kappa\diff\de s\\
      &= \int(\ps\kappa)^2\de s + 2\mu\int\ps\kappa\,\ps^3\kappa \de s + 5\mu\int (\ps\kappa)^2 \kappa\diff \de s 
       + 2\mu\int(\ps\kappa)^2\kappa^2\de s.
 \end{align*}
 Integration by parts for the second integral yields the result.
\end{proof}
In the next step, we want to have a similar formula as above also for the time-derivative of the $L^2-$norm of the higher space-derivatives. It is useful to introduce some notation, which we learned in \cite{Mantepaper}.  

\begin{notations}
 For $j,\sigma\in\N_0$ we denote by $p_\sigma(\ps^j\kappa)$ a polynomial in $\kappa,\dots,\ps^j\kappa$ with constant coefficients such that every monomial that is contained in $p_\sigma(\ps^j\kappa)$ is of the form 
 \begin{align*}
  c\prod_{l=0}^{j}\left(\ps^l\kappa\right)^{\alpha_l}  \ \ \ \text{ with } \sum_{l=0}^j\alpha_l(l+1)=\sigma, \ \ \ \ \alpha_l\in\N_0.
 \end{align*}
 We use the same notation, if somewhere in the zero-order term arises $\diff$ instead of $\kappa$. That means $\kappa\diff^2\ps^2\kappa=p_6(\ps^2\kappa)$ for example.\\
 We also use this notation for the time derivative, but we want to respect the parabolicity of our equations. Let $\sigma, j, h\in\N_0$. We denote by $p_\sigma(\pt^j(\kappa-\bar\kappa),\ps^h\kappa)$ a polynomial in $\diff,\dots,\pt^j\diff$ and $(\kappa-\bar\kappa),\ps\kappa,\dots,\ps^h\kappa$ with constant coefficients such that every monomial that is contained in $p_\sigma(\pt^j\diff,\ps^h\kappa)$ is of the form 
 \begin{align*}
  c\prod_{l=0}^{j}\left(\pt^l\diff\right)^{\alpha_l} \cdot \prod_{l=0}^{h}\left(\ps^l\diff\right)^{\beta_l} \ \ \ \text{ with } \sum_{l=0}^j\alpha_l(2l+1) + \sum_{l=0}^h \beta_l(l+1)=\sigma, \ \ \ \ \alpha_l\in\N_0.
 \end{align*}
 This is consistent with the notation above because $p_\sigma(\pt^0\diff,\ps^h\kappa)= p_\sigma(\ps^h\kappa)$. In general, one can prove
 \begin{align*} 
   \pt^l p_\alpha (\pt^j\diff,\ps^h\kappa) = p_{\alpha + 2l}(\pt^{j + l}\diff,\ps^{h + 2l}\kappa).
 \end{align*}

\end{notations}

\begin{lemma}\label{tausch}
 With the previous notation we have for $m\in\N_0$
 \begin{align*}
  \pt\ps^m\kappa = \ps^{m+2}\kappa + p_{m+3}(\ps^m\kappa).
 \end{align*}
\end{lemma}

\begin{proof}
 For $m=0$ we have $\pt\kappa = \ps^2\kappa + \kappa^2\diff = \ps^2\kappa + p_3(\kappa)$. For $m\geq 1$ we proof the claim by induction. We have by the rule for exchanging $\pt$ and $\ps$
 \begin{align*}
  \pt\ps^{m+1}\kappa&= \pt\ps\ps^m\kappa = \ps\pt\ps^m\kappa + \kappa\diff\ps^{m+1}\kappa\\
    &= \ps\left(\ps^{m+2}\kappa + p_{m+3}(\ps^m\kappa)\right) + p_{m+4}(\ps^{m+1}\kappa)\\
    &= \ps^{m+3}\kappa + p_{m+4}(\ps^{m+1}\kappa),
 \end{align*}
which is the claim for $m+1$. In the second line we used the induction hypotheses and in the third line we used $\ps p_\sigma(\ps^j\kappa)=p_{\sigma +1}(\ps^{j+1}\kappa)$, which is true because by differentiating we lower the order $\alpha_l$ by one and increase $\alpha_{l+1}$ by one. This has to be done for every $l\in \{0,\dots,j\}$. In the sum we always get $\sigma + 1$. 
\end{proof}

\begin{lemma} \label{abl}
 Let $c:[a,b]\times[0,T)\to \Rzwei $ ($T\leq\infty$) be a solution of (\ref{flow}).
 For a fix $\tau>0$ we define $\mu\defi t-\tau$. Then we have for $t>0$ and $m\in\N_0$
 \begin{align*}
  \dt\left(\frac{\mu^m}{m!}\int(\ps^m\kappa)^2 \de s\right)= \frac{\mu^{m-1}}{(m-1)!}\int(\ps^m\kappa)^2\de s &+ \frac{2\mu^m}{m!}\left[\ps^m\kappa\,\ps^{m+1}\kappa\right]_a^b -  \frac{2\mu^m}{m!} \int (\ps^{m+1}\kappa)^2\de s\\
 & + \mu^m\int p_{2m+ 4}\left(\ps^m\kappa\right)\de s.
 \end{align*}
\end{lemma}

\begin{proof}
We use the previous lemma  and integration by parts and get
\begin{align*}
 \dt\left(\frac{\mu^m}{m!}\int(\ps^m\kappa)^2 \de s\right)&=\frac{\mu^{m-1}}{(m-1)!}\int(\ps^m\kappa)^2\de s  + 2\frac{\mu^m}{m!} \int\ps^m\kappa\,\pt\ps^m\kappa\de s \\
 &\hspace*{6cm}- \frac{\mu^m}{m!}\int (\ps^m\kappa)^2\kappa\diff\de s\\
 &= \frac{\mu^{m-1}}{(m-1)!}\int(\ps^m\kappa)^2\de s  + 2\frac{\mu^m}{m!} \int\ps^m\kappa \,\ps^{m + 2}\kappa\de s  \\
  &\qquad + 2\frac{\mu^m}{m!} \int\ps^m\kappa\, p_{m+3}(\ps^m\kappa)\de s + \mu^m\int p_{2m+4 }(\ps^m\kappa)\de s\\
  &= \frac{\mu^{m-1}}{(m-1)!}\int(\ps^m\kappa)^2\de s + \frac{2\mu^m}{m!}\left[\ps^m\kappa\,\ps^{m+1}\kappa\right]_a^b -  \frac{2\mu^m}{m!} \int (\ps^{m+1}\kappa)^2\de s\\
 & \qquad + \mu^m\int p_{2m+ 4}\left(\ps^m\kappa\right)\de s.
\end{align*}
\end{proof}

We want to get a bound on $\|\ps^l\kappa\|_2$ at least for $l=1$. When we differentiate $\|\ps^l\kappa\|_2$ in $t$ we get boundary terms via integration by parts. Precisely, terms like $\ps^{l +1}\kappa(a,t)$ and $\ps^{l +1}\kappa(b,t)$ occur on the right hand side. We saw in Lemma~\ref{lemmastahl5} that we can express $\ps\kappa$ at the boundary in terms of $\diff$ and $\Sig\kappa$: 
\begin{align}\begin{split} \label{cro8}
 \partial_s \kappa (a,t) &= \left(\kappa(a,t) - \bar\kappa(t)\right) \nix^\Sigma\kappa\left(\nix^\Sigma f^{-1}(c(a,t))\right)\quad \mbox{ and }\\
 \partial_s \kappa (b,t) &= -\left(\kappa(b,t) - \bar\kappa(t)\right) \nix^\Sigma\kappa\left(\nix^\Sigma f^{-1}(c(b,t))\right).\end{split}                                                                                                                               
\end{align}
This also works for $l$ bigger: We can express the $(l+1)$-th derivative of $\kappa$ at the boundary by the $l$-th derivative at the boundary times something like $\ps^j\Sig\kappa$, which is uniformly bounded because $\Sigma$ is smooth and fixed. Unfortunately, this only works for $l$ even. We use the Gagliardo-Nirenberg interpolation inequalities and the Young inequality to absorb the $l$-the derivative of $\kappa$ at the boundary. We do the calculations for $l=2$. At first, we need a formula for $\ps^3\kappa$ at the boundary.

\begin{lemma}\label{rand2}
  Let $c:[a,b]\times[0,T)\to \Rzwei $ ($T\leq\infty$) be a solution of (\ref{flow}).
  For any $t\in(0,T)$ we have at the point $(a,t)$
  \begin{align*}
   \ps^3\kappa = (\ps^2\kappa - 3\kappa\diff^2 - \pt\bar\kappa)\left(\Sig\kappa \circ\fsig^{-1}\circ c\right) + \diff^2\left(\ps\Sig\kappa\circ\fsig^{-1}\circ c\right),  
  \end{align*}
and at the point $(b,t)$
  \begin{align*}
 \ps^3\kappa = - (\ps^2\kappa - 3\kappa\diff^2 - \pt\bar\kappa)\left(\Sig\kappa \circ\fsig^{-1}\circ c\right) + \diff^2\left(\ps\Sig\kappa\circ\fsig^{-1}\circ c\right).
  \end{align*}
\end{lemma}

\begin{proof}
We differentiate (\ref{cro8}) in $t$. Let $\ftilde$ be the periodic extension of $\fsig$. 
Since $c(a,t)=\ftilde(a(t))$ $\forall t\in [0,T)$ (Lemma~\ref{lift}) we have $\partial_t c(a,t)=(\partial_p\ftilde) (a(t))\partial_t a(t)$ and thus
\begin{align*}
 \langle \partial_t c(a,t),\partial_p\ftilde (a(t))\rangle=\partial_t a(t)|\partial_p\ftilde(a(t))|^2 = \partial_t a(t),
\end{align*}
where we used that $\ftilde$ is parametrized by arclength. We have that $\partial_p\ftilde(a(t))=  \Sig\vec\tau\,\left(\ftilde(a(t))\right) = \Sig\vec\tau(c(a,t)) $.
We now use the equation $\partial_t c = (\kappa -\bar\kappa)\nu$ for the boundary points
\begin{align*}
 \partial_t a(t)= \langle \partial_t c(a,t),\Sig\vec\tau(c(a,t))\rangle = \left(\kappa(a,t) - \bar\kappa(t)\right) \langle \nu(a,t),\Sig\vec\tau(c(a,t))\rangle= \kappa(a,t)-\bar\kappa(t).
\end{align*}
We hence get
\begin{align}\label{dt2}
 \dt \Sig\kappa\circ\fsig^{-1}\circ c(a,t) = (\kappa(a,t)-\bar\kappa(t))\, \ps\Sig\kappa\circ\fsig^{-1}\circ c (a,t).
\end{align}
We now differentiate (\ref{cro8}), use (\ref{dt2}), the evolution equation for $\kappa$ (equation (\ref{evolution5})), the rule for changing $\ps$ and $\pt$ (equation (\ref{evolution2})) and get at the point $(a,t)$
  \begin{align*}
   \ps^3\kappa = &(\ps^2\kappa + \kappa^2\diff -\pt\bar\kappa)\left(\Sig\kappa \circ\fsig^{-1}\circ c\right)\\
   &+ \diff^2\left(\ps\Sig\kappa\circ\fsig^{-1}\circ c\right)  - 3\kappa\diff\,\ps\kappa - \kappa^2\ps\kappa,
  \end{align*}
and at the point $(b,t)$
  \begin{align*}
   \ps^3\kappa = &-(\ps^2\kappa + \kappa^2\diff - \pt\bar\kappa)\left(\Sig\kappa \circ\fsig^{-1}\circ c\right) \\
   &+ \diff^2\left(\ps\Sig\kappa\circ\fsig^{-1}\circ c\right)  - 3\kappa\diff\,\ps\kappa - \kappa^2\ps\kappa.
  \end{align*}
  Formula (\ref{cro8}) yields the result.
  \end{proof}

  \begin{lemma}\label{lem1}
   Let $c:[a,b]\times[0,T)\to \Rzwei $ ($T\leq\infty$) be a solution of (\ref{flow}).
   Then we have for $t\in(0,T)$ and with $\mu\defi t-\tau$, $\tau$ fixed,
   \begin{align*}
    \dt\left(\int\diff^2 + \mu(\ps\kappa)^2 +  \frac{\mu^2}{2}(\ps^2\kappa)^2\de s\right) \leq &-\int (\ps\kappa)^2 + \mu (\ps^2\kappa)^2 + \mu^2 (\ps^3\kappa)^2\de s\\
    &  +  4\mu\|\ps\kappa\|_\infty\|\ps^2\kappa\|_\infty + 2\mu^2\|\ps^2\kappa\|_\infty\|\ps^3\kappa\|_\infty \\
    & + \int p_4(\kappa) + \mu p_6(\ps\kappa) + \mu^2 p_8(\ps^2\kappa)\de s.
   \end{align*}

  \end{lemma}

  \begin{proof}
   By  Lemma~\ref{abl}, we get
     \begin{align*}
    \dt&\left(\int\diff^2 + \mu(\ps\kappa)^2 +  \frac{\mu^2}{2}(\ps^2\kappa)^2\de s\right) \\
    &=-\int (\ps\kappa)^2 + \mu (\ps^2\kappa)^2 + \mu^2 (\ps^3\kappa)^2\de s
     + 2\left[\diff\ps\kappa\right]_a^b + 2\mu\left[\ps\kappa\ps^2\kappa\right]_a^b \\ &+ \mu^2\left[\ps^2\kappa\ps^3\kappa\right]_a^b  + \int p_4(\kappa) + \mu p_6(\ps\kappa) + \mu^2 p_8(\ps^2\kappa)\de s.
   \end{align*}
   As in Corollary~\ref{cor4} we have $\left[\diff\ps\kappa\right]_a^b\leq 0$.
  \end{proof}
  
  \begin{remark}
   We note here, that the idea to differentiate the term $$\int\diff^2 + \mu(\ps\kappa)^2 +  \frac{\mu^2}{2}(\ps^2\kappa)^2\de s$$ in $t$ and use the Gagliardo-Nirenberg interpolation inequalities comes from \cite{Mantepaper}.
  \end{remark}

  \begin{lemma} \label{restint}
    Let $c:[a,b]\times[0,T)\to \Rzwei $ ($T\leq\infty$) be a solution of (\ref{flow})
    with uniform bounds $|\bar\kappa(t)|\leq c_2<\infty$ and $L(\ct)\geq c_1 >0$ for all $t\in[0,T)$. Then we have for $j\in\N_0$: For every $\epsilon>0$ there are constants $C_1,C_2$ only depending on $\epsilon$, $j$, $L_0$, $c_1$ and $c_2$ such that for all $t\in(0,T)$ we have
    \begin{align*}
     \int p_{2j + 4}(\ps^j\kappa)\de s\leq \epsilon \int|\ps^{j + 1}\kappa|^2 \de s + C_1\left(\int\diff^2\de s\right)^{2j + 3} + C_2.
    \end{align*}

  \end{lemma}
  \begin{proof}
   The proof works as in \cite[After Proposition~3.11]{Mantepaper}, we only have to substitute the zero order term $\kappa$ by $\kappa-\bar\kappa$. The Gagliardo-Nirenberg interpolation inequalities are used, thus the constants depend on the bounds on $L(\ct)$. Since we possibly transform an integral involving $\kappa^i$ as a factor of the integrand into an integral with $\diff^k$ as an factor in the integrand, we use the bound on $|\bar\kappa|$. 
  \end{proof}

  \begin{lemma} \label{lemboundary}
    Let $c:[a,b]\times[0,\infty)\to \Rzwei $ be a solution of 
    with $|\bar\kappa(t)|\leq c_2 <\infty$ and $L(\ct)\geq c_1 >0$ for all $t\in[0,\infty)$. Then for every $\epsilon>0$ there are constants $C_1, C_2$ only depending on $\epsilon$, $L_0$, $c_1$, $c_2$, $\|\Sig\kappa\|_{\infty}$ and  $\|\kappa-\bar\kappa\|_2$ (which is uniformly bounded cf. Corollary~\ref{intto0}) such that for all $t\in(0,\infty)$ we have
    \begin{align}\label{claim1}
    \mu \left|\left[\ps\kappa\ps^2\kappa\right]_a^b\right|\leq \epsilon \mu^2\|\ps^3\kappa\|_2^2 + \epsilon \mu \|\ps^2\kappa\|_2^2 + C_1 \mu + C_2.
    \end{align}
    Furthermore, there is a constant $C_3$ depending on $\epsilon$, $L_0$, $c_1$, $c_2$, $\|\kappa-\bar\kappa\|_2$ and $\|\Sig\kappa\|_{\infty} + \|\ps\Sig\kappa\|_{\infty}$ such that
    \begin{align}\label{claim2}
    \mu^2 \left|\left[\ps^2\kappa\ps^3\kappa\right]_a^b\right|\leq \epsilon \mu^2\|\ps^3\kappa\|_2^2 +  C_3 \mu^2,
    \end{align}
    where $\mu=t-\tau$, $\tau$ fixed.
  \end{lemma}

  \begin{proof}
   We use $\|\kappa-\bar\kappa\|_2\leq C$ and $\|\kappa\|_2\leq C$. By Lemma~\ref{gagli1a} and Corollary~\ref{gagli2}, we get $  \|\kappa-\bar\kappa\|_\infty\leq c\|\ps^3\kappa\|_2^\frac{1}{6}$ and $ \|\ps^2\kappa\|_\infty \leq c\|\ps^3\kappa\|_2^\frac{1}{2}\|\ps^2\kappa\|_2^{\frac{1}{2}}  + c\|\ps^2\kappa\|_2$.
These two inequalities imply together with (\ref{cro8}) and the young inequality 
\begin{align*}
  \mu \left|\left[\ps\kappa\ps^2\kappa\right]_a^b\right|& \leq c\mu\|\kappa-\bar\kappa\|_\infty\|\ps^2\kappa\|_\infty
  \leq  c\mu\left(\|\ps^3\kappa\|_2^{\frac{2}{3}} \|\ps^2\kappa\|_2^{\frac{1}{2}}  + \|\ps^3\kappa\|_2^{\frac{1}{6}} \|\ps^2\kappa\|_2 \right)\\
  &= c\left( \mu^{\frac{2}{3} + \frac{1}{3}}\|\ps^3\kappa\|_2^{\frac{2}{3}} \|\ps^2\kappa\|_2^{\frac{1}{2}}  + \mu^{\frac{1}{6} + \frac{5}{6}}\|\ps^3\kappa\|_2^{\frac{1}{6}} \|\ps^2\kappa\|_2 \right)\\
  &\leq 2 \epsilon \mu^2 \|\ps^3\kappa\|_2^{2} + c(\epsilon)\mu^{\frac{1}{2}}\|\ps^2\kappa\|_2^{\frac{3}{4}} +  c(\epsilon)\mu^{\frac{10}{11}}\|\ps^2\kappa\|_2^{\frac{12}{11}} \\
  &\leq 2\epsilon \mu^2 \|\ps^3\kappa\|_2^{2} + 2\epsilon \mu \|\ps^2\kappa\|_2^{2} + c(\epsilon)\left(\mu^{\frac{1}{8}\cdot\frac{8}{5} }  + \mu^{\frac{4}{11}\cdot\frac{11}{5} }\right)\\
  &\leq 2\epsilon \mu^2 \|\ps^3\kappa\|_2^{2} + 2\epsilon \mu \|\ps^2\kappa\|_2^{2} + c(\epsilon)\mu + c(\epsilon),
\end{align*}
and this is (\ref{claim1}). For (\ref{claim2}), we show $\left|\left[\ps^2\kappa\ps^3\kappa\right]_a^b\right|\leq \epsilon \|\ps^3\kappa\|_2^2 +  C_1. $
Due to Lemma~\ref{rand2}, we have
\begin{align}\begin{split}
 \left|\left[\ps^2\kappa\ps^3\kappa\right]_a^b\right|&\leq 2\|\ps^2\kappa\|_\infty\|\ps^3\kappa\|_\infty\\
 &\leq \tilde c\, \|\ps^2\kappa\|_\infty\Big(\|\ps^2\kappa\|_\infty + \|\kappa\diff^2\|_\infty + |\dt\bar\kappa| + \|\kappa-\bar\kappa\|_\infty^2\Big),
\end{split}\label{re27}\end{align}
where $\tilde c$ also depends on $\|\Sig\kappa\|_{\infty} + \|\ps\Sig\kappa\|_{\infty}$.
We use Lemma~\ref{gagli1a} and Corollary~\ref{gagli2} and get
\begin{align*}
 \|\kappa-\bar\kappa\|_\infty&\leq c\|\ps^3\kappa\|_2^\frac{1}{6}, \qquad \|\kappa\|_\infty\leq c\|\ps^3\kappa\|_2^\frac{1}{6} + c,\\
 \|\ps\kappa\|_\infty&\leq c\|\ps^3\kappa\|_2^\frac{1}{2}, \qquad \|\ps^2\kappa\|_\infty\leq c\|\ps^3\kappa\|_2^\frac{5}{6}.
\end{align*}
We compute 
\begin{align*}
 \begin{split}
 \frac{\de}{\de t} \bar\kappa &= \frac{1}{L(c_t)} \frac{\de}{\de t} \int_a^b\kappa\de s - \frac{1}{L(c_t)^2} \frac{\de}{\de t} L(c_t) \int_a^b\kappa\de s\\
    & =  \frac{1}{L(c_t)} \int_a^b \partial_t\kappa \de s + \frac{1}{L(c_t)}\int_a^b \kappa \partial_t\left(\de s\right)  + \frac{\bar\kappa}{L(c_t)}\int_a^b (\kappa - \bar\kappa)^2\de s\\
     & =  \frac{1}{L(c_t)} \int_a^b \partial^2_s \kappa \de s  + \frac{\bar\kappa}{L(c_t)}\int_a^b (\kappa - \bar\kappa)^2\de s\\
     &= \frac{1}{L(\ct)}\left[\ps\kappa\right]_a^b + \frac{\bar\kappa}{L(\ct)}\intab\diff^2\de s.
\end{split} 
\end{align*}
 This implies $\left|\dt\bar\kappa\right|\leq c\|\ps\kappa\|_\infty + c\leq c\|\ps^3\kappa\|_2^\frac{1}{2} + c. $
We combine (\ref{re27}) and the interpolation inequalities above and get
\begin{align*}
 \left|\left[\ps^2\kappa\ps^3\kappa\right]_a^b\right|\leq c_\Sigma \left(\|\ps^3\kappa\|_2^{\frac{5}{3}} + \|\ps^3\kappa\|_2^{\frac{5}{6}} + \|\ps^3\kappa\|_2^{\frac{4}{3}} + \|\ps^3\kappa\|_2^{\frac{7}{6}}\right).
\end{align*}
All the exponents are smaller then $2$, so by the Young inequality, we get the result.
  \end{proof}

  \begin{lemma}\label{dtL2}
   Let $c:[a,b]\times[0,\infty)\to \Rzwei $ be a solution of (\ref{flow}) with
   $|\bar\kappa(t)|\leq c_2<\infty$ and $L(\ct)\geq  c_1 >0$ for all $t\in[0,\infty)$.  Then there is a constant $C>0$ such that for all $t\in(0,\infty)$ we have
   \begin{align*}
     \dt\left(\int\diff^2 + \mu(\ps\kappa)^2 +  \frac{\mu^2}{2}(\ps^2\kappa)^2\de s\right) \leq C(\mu^2 + \mu + 1)
   \end{align*}
where $\mu=t-\tau$, $\tau$ fixed. The constant depends on $c_1$, $c_2$, $L_0$, $\|\kappa-\bar\kappa\|_2$ and $\|\Sig\kappa\|_{\infty} + \|\ps\Sig\kappa\|_{\infty}$.
  \end{lemma}
  
  \begin{proof}
   Combine Lemma~\ref{lem1}, Lemma~\ref{restint} and Lemma~\ref{lemboundary}. 
  \end{proof}

  \begin{corollary} \label{kappaabsch}
    Let $c:[a,b]\times[0,\infty)\to \Rzwei $ be a solution of (\ref{flow}) with 
    $|\bar\kappa(t)|\leq c_2<\infty$ and $L(\ct)\geq  c_1 >0$ for all $t\in[0,\infty)$. Then there is a constant $C>0$ depending on $c_1$,$c_2$, $L_0$, $\|\kappa-\bar\kappa\|_2$ and $\|\Sig\kappa\|_{\infty} + \|\ps\Sig\kappa\|_{\infty}$ such that for all $t\in[1,\infty)$ we have
    \begin{align*}
     \|\kappa\|_\infty + \|\ps\kappa\|_\infty\leq C.
    \end{align*}

  \end{corollary}
  
  \begin{proof}
 The proof is analogous to that of \cite[Proposition~4.8]{Magni}:
   Consider any sequence $t_l\to\infty$ and define $\mu\defi t-t_l$. With the previous lemma, we get 
 \begin{align*}
   \dt\left(\int\diff^2 + \mu(\ps\kappa)^2 +  \frac{\mu^2}{2}(\ps^2\kappa)^2\de s\right) \leq C(\mu^2 + \mu + 1).
 \end{align*}
Integration from $t_l$ untill $t\in(t_l,t_l + 2\delta]$ ($\delta>0$ arbitrary) yields
\begin{align*}
 \int\diff^2 + \mu (\ps\kappa)^2  +  \frac{\mu^2}{2}(\ps^2\kappa)^2\de s\leq \int\diff^2\de s\big|_{t=t_l} + c \left(\delta^{3} + \delta^2 + \delta \right)\leq c \left(\delta^{3} + \delta^2 + \delta + 1 \right).
\end{align*}
  Considering $t \in [t_l + \delta, t_l+2\delta]$, we have
  \begin{align*}
  \int\diff^2 + \delta (\ps\kappa)^2 + \frac{\delta^2}{2}(\ps^2\kappa)^2\de s &\leq 
   \int\diff^2 + \mu (\ps\kappa)^2  + \frac{\mu^2}{2}(\ps^2\kappa)^2\de s \\&\leq  c \left(\delta^{3} + \delta^2  + \delta + 1\right).
  \end{align*}
It follows for $t \in [t_l + \delta, t_l+2\delta]$ that
\begin{align*}
\|\kappa-\bar\kappa\|_{2}+\|\ps\kappa\|_{2}+ \|\ps^2\kappa\|_{2}\leq c(\delta).
\end{align*}
The constant $\delta$ was arbitrary and the estimate above holds for any sequence $t_l\to\infty$. So we choose $t_l$ and $\delta$ such that for every $t\in[1,\infty)$ we have $t\in[t_l + \delta,t_l + 2\delta]$ for an $l\in\N$. For $t\in[1,\infty]$, it follows that
\begin{align*}
 \|\kappa-\bar\kappa\|_{2}+\|\ps\kappa\|_{2}+ \|\ps^2\kappa\|_{2}\leq c.
\end{align*}
The integral $\int\kappa^2\de s$ is bounded because of Corollary~\ref{intto0} and the bounds on $\bar\kappa$. We get the result by interpolation.
  \end{proof}

\begin{theorem}
 Let $c:[a,b]\times[0,\infty)\to \Rzwei $ be a solution of the area preserving curve shortening problem with Neumann free boundary conditions with
 \begin{align*}
  0<\bar c\leq|\bar\kappa(t)|\leq c_2<\infty \ \ \ \text{ and }\ \ \ L(\ct)\geq  c_1 >0  \ \  \ \text{ for all } t\in[0,\infty).
  \end{align*}
Then for every sequence $t_j\to\infty$ the curves $c(\cdot,\tj)$ sub\-con\-verge (after reparametrization) smooth\-ly to a (possibly multiply covered) arc of a circle. The two boundary points of this arc are in $\Sigma$ and the arc is perpendicular to $\Sigma$ at these points. The arc ``starts'' into the outer region with respect to $\Sigma$, and it meets $\Sigma$ from the outside at the ``endpoint''.
\end{theorem}

\begin{proof}
For any sequence $\tau_l\to\infty$, reparametrize the curves $c(\cdot,\tau_l)$ by constant speed: Define $L_l\defi L(c_{\tau_l})$ and note $L_l\in[c_1,L_0]$. Define $\varphi_l(p)\defi\frac{1}{L_l}\int_a^p|c'(\cdot,\tau_l)|$
and denote the inverse function by $\psi_l:[0,1]\to[a,b]$. Define
\begin{align*}
 \gamma_l\defi c(\psi_l,\cdot):[0,1]\times [0,\infty)\to\Rzwei.
\end{align*}
Then we have $|\gamma_l'|= L_l$ at the times $\tau_l$. By Corollary~\ref{kappaabsch} we have $$\sup_{(p,t)\in[0,1]\times[1,\infty)}|\kappa_l(p,t)|\leq C$$ for a constant not depending on $l$. As in Proposition~\ref{ParabolicGrenzuebergangII} 
we get a constant $c$ depending on $m,\Sigma,C, c_1, L_0$ and $\delta$ such that $$|\pp^m\gamma_l|\leq c \text{ on }[0,1]\times [\tau_l,\tau_l + \delta].$$It is easy to find $\tau_l\to\infty$ and $\delta>0$ such that $\bigcup_{l\in\N}[\tau_l,\tau_l + \delta) = [1,\infty)$. It follows that 
\begin{align}\label{re28}
 |\pp^m\gamma_l|\leq c \text{ on }[0,1]\times [1,\infty).
\end{align}
For every sequence $t_l\to\infty$ we consider $\gamma_l(\cdot,t_l)$. By inequality (\ref{re28}) and the theorem of Arzela-Ascoli we have subconvergence to a smooth curve $\gamma_\infty:[0,1]\to\Rzwei$ in every $C^{m}$ on $[0,1]$, $m\in\N_0$. 
This implies
\begin{align*}
   \lim_{l\to\infty}\bar\kappa(t_l) = \lim_{l\to\infty}\bar\kappa(\gamma_l) = \lim_{l\to\infty}\frac{\int\kappa_{\gamma_l}\de s_{\gamma_l}}{\int\de s_{\gamma_l}}=\bar\kappa(\gamma_\infty)\in[\bar c,c_2].
 \end{align*}
It follows that
\begin{align*}
 \lim_{l\to\infty}\int(\kappa_{\gamma_l}-\bar\kappa(\gamma_\infty))^2\de s_{\gamma_l} = \lim_{l\to\infty}\int(\kappa_{\gamma_l}-\bar\kappa(\gamma_l))^2\de s_{\gamma_l} = 0,
\end{align*}
which implies $\kappa_{\gamma_\infty}\equiv \bar\kappa(\gamma_\infty)>0$. Thus, $\gamma_\infty$ is an arc of a circle. Since $\Sigma$ is closed, the points $\gamma_\infty(0)= \lim_{l\to\infty}c(a,t_l)$ and $\gamma_\infty(1)= \lim_{l\to\infty}c(b,t_l)$ lie on $\Sigma$. The facts about the contact angles follow from continuity.
%
\end{proof}

\begin{remark}
 In the proof of the previous theorem, one could have used the estimates on the higher derivatives of $\kappa$ as in Lemma~\ref{lem2} instead of the estimates that came from the graph representation.
\end{remark}

%
%

 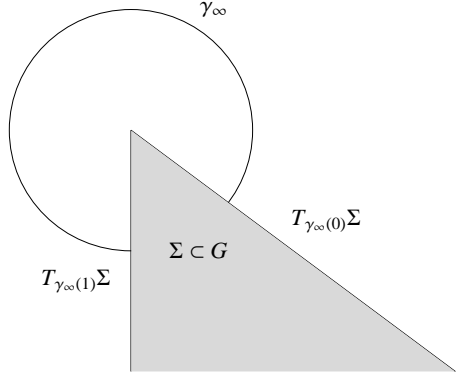
\begin{figure}
	  \begin{center}
\scalebox{0.8}{ \begin{tikzpicture}
 
 \draw (0,0) circle (2cm);
 \draw (-0.00,-4) to (-0.00,0); 
  \draw (5.333,-4) to (0,0);
\fill[gray!30] (0.005,-4)-- (0.005,-0.005) -- (5.33,-4);
 \coordinate[label = right: ${\Sigma\subset G}$] (C) at (0.5,-2);
  \coordinate[label = right: ${T_{\gamma_\infty(0)}\Sigma}$] (D) at (2.5,-1.5);
    \coordinate[label = right: ${\gamma_\infty}$] (E) at (1,2);
      \coordinate[label = left: ${T_{\gamma_\infty(1)}\Sigma}$] (F) at (-0.2,-2.5);

 \end{tikzpicture}}
 
	  \end{center}
	   \caption{The situation in the limit: The domain $G$ (gray), where the support curve is contained.} \label{Bild17} 
\end{figure}

{\it Proof of Theorem~\ref{mainthm}.}
Theorem~\ref{typeIIrandneu} gives $T=\infty$. The bounds on $\bar\kappa$ and $L(\ct)$ can be found in  Proposition~\ref{Laengenachuntenabsch} and Theorem~\ref{Linftyabschaetzung}. By Theorem~\ref{mitdSig}, we have $\int\kappa\de s\in[\pi,2\pi)$ for all $t\in[0,\infty)$. It follows that $\int\kappa_\infty\de s_{\gamma_\infty}\in[\pi,2\pi]$. But $\int\kappa_\infty\de s_{\gamma_\infty}=2\pi$ would imply that $\gamma_\infty$ is a full circle. 
This is a contradiction to the boundary behaviour of $\gamma_\infty$. Hence, $\gamma_\infty$ is embedded.\\
 It is easy to see that $\gamma_\infty$ is outside of the domain created by $\Sigma$: The tangent $T_{\gamma_\infty(0)}\Sigma$ is perpendicular to the tangent line of $\gamma_\infty(0)$. We even know 
 $$\winkel\left(\tau_{\gamma_\infty}(0),\Sig\vec\tau(\gamma_\infty(0))\right)=\frac{\pi}{2} \text{ and } \winkel\left(\tau_{\gamma_\infty}(1),\Sig\vec\tau(\gamma_\infty(1))\right)=-\frac{\pi}{2}.$$
 Since a smooth convex domain always lies on the side of the half space created by the tangent line that goes into the direction of the inner normal, we get a domain $G$, where $\Sigma$ must be contained. In Figure \ref{Bild17}, this domain is colored in gray. Since we are in the special situation that $\gamma_\infty$ is a part of a circle, the two generating lines of this domain (namely $T_{\gamma_\infty(0)}\Sigma$ and $T_{\gamma_\infty(1)}\Sigma$) meet in the center of the circle. But this implies that there cannot be another intersection point of $\gamma_\infty$ with $\Sigma$ except the two boundary points.


\begin{appendix}
\section{Further calculations} \label{App}

\begin{lemma}[due to Stone \cite{Stone}]  \label{Parabolicstone}
  Let $c:[a,b]\times[0,T)\to \Rzwei $ be a solution of the area preserving curve shortening problem with Neumann free boundary conditions with singularity at $T<\infty$,
where $x_0\in\Sigma$ is a blowup point in the support curve. 
Assume further $|\bar\kappa|\leq c_2<\infty.$
After parabolic rescaling as in Definition~\ref{Parabolicrescaling} and with the notation $\tildec^\tau\defi \tilde c_j(\cdot,\tau)$ we have the following properties:
 \begin{enumerate} 
  \item For any $\tau\in(-\infty,0)$ there is a $j_0=j_0(\tau,T) \in \mathbb N$ and a constant \\$\hat C_3 = \hat C_3 (\tau,L(c_0), c_2,T)$ such that for all $j\geq j_0$ \label{Paraboliclemmaa}
	\begin{align*}
	 \int_{\tildec^{\tau}} e^{-\frac{|x|}{-2\tau}}\de \tilde s_\tau\leq \hat C_3.
	\end{align*}
  \item For any $\tau\in(-\infty,0)$ and $\epsilon>0$ there is a $j_0=j_0(\tau,T) \in \mathbb N$ and a radius \\$R_0=R_0(\tau,\epsilon,L(c_0), c_2,T)$ such that for all $j\geq j_0$ \label{Paraboliclemmab}
	\begin{align*}
	  \int_{\tildec^{\tau}\cap \,\left(B_{R_0}(0)\right)^c} e^{-\frac{|x|^2}{-4\tau}}\de \tilde s_\tau\leq  \epsilon.
	\end{align*}
 \end{enumerate}
\end{lemma}

\begin{proof}
   The proof is almost the same as the original proof of Stone. It only has to be adapted to our flow equation. Indeed, the crucial part of the proof is the inequality 
    \begin{align*}
     \frac{\de}{\de \tau}\left( \tilde f_j \int_{\tildec^{\tau}}g\de \tilde s_\tau\right)\leq \frac{1}{2\tau^2}\tilde f_j  \int_{\tildec^{\tau}}g \left(\frac{1}{2}-\tau-|x|\right)\de \tilde s_\tau.
    \end{align*}
     with $g(x,\tau)\defi \frac{1}{\sqrt{-4\pi \tau}} e^{-\frac{|x|}{-2\tau}} $. 
     Buckland's expansion formula (\cite[Proposition~2.3]{Buckland}) can be used to derive this inequality (since this lemma can be done for arbitrary dimension, we use the notation for surfaces):  
    For functions $f,g:U\times [\tau_1,\tau_2)\to \R$ sufficiently smooth, where $\tilde M_{\tau}\subset U\subset \Rnplus \ \forall \tau\in[\tau_1,\tau_2)$, and $g>0$ we have
    \begin{align*}
     \frac{\de}{\de \tau} \int_{\tilde M_\tau} f g \de \tilde\mu_\tau = & - \int_{\tilde M_\tau} f\big| \vec H- \frac{D^{\perp}g}{g}\big|^2 g \de \tilde \mu_\tau\\
	   & + \int_{\tilde M_\tau} f \left[ Q(g) + \langle Dg,\frac{\partial x}{\partial \tau} - \vec H\rangle + \left(|\vec H|^2 + \Div_{\tilde M_\tau} \frac{\partial x}{\partial \tau}\right) g \right]\de \tilde\mu_\tau\\
	  & + \int_{\tilde M_\tau} g \left(\frac{\de}{\de \tau} - \Delta_{ \tilde M_\tau}\right) f\de \tilde \mu_\tau\\
	  & + \int_{\partial\tilde M_\tau} \left( g \langle D f,\mu\rangle - f\langle Dg,\mu\rangle \right)  \de\tilde \sigma_\tau,
    \end{align*}
    where $\mu$ is the outer unit conormal to $\partial\tilde M_{\tau}$ and $Q(g)= \frac{\partial g}{\partial \tau} + \Div_{\tilde M_\tau} Dg + \frac{|D^\perp g|^2}{g}$. We use this formula for $f = \tilde f_j$, $g = \frac{1}{\sqrt{-4\pi \tau}} e^{-\frac{|x|}{-2\tau}}$ and $\tilde M_\tau =  \tildec^{\tau}$. 
    In this situation we have $\frac{\partial x}{\partial \tau} = -(\kappa-\bar\kappa)\hat\nu$ and $\Div_{\tilde M_\tau}\frac{\partial x}{\partial \tau} = \langle (\kappa-\bar\kappa) \hat\nu,-\kappa\hat\nu\rangle = -\kappa(\kappa-\bar\kappa)$.\footnote{Since the calculations can also be done for the case $n\geq 2$ with codimension one, the ``outer unit normal`` $\hat\nu = -\nu$ were used. This is just a matter of notation.} 
    We omit the index $j$ in the calculations, which yield
    \begin{align*}
     Dg &= g \frac{x}{2\tau|x|} \\
     \frac{D^\perp g}{g}&= \frac{x^\perp}{2\tau|x|}\\
    \frac{|D^\perp g|^2}{g}& = g \frac{|x^\perp|^2}{4\tau^2|x|^2} \\
    \frac{\partial g}{\partial \tau} &= g\left(-\frac{1}{2\tau} - \frac{|x|}{2\tau^2}\right)\\
    \Div_{\tilde M_\tau} Dg &= g \left(- \frac{|x^\top|^2}{2\tau|x|^3} + \frac{|x^\top|^2}{4\tau^2|x|^2} + \frac{1}{2\tau|x|}\right)\\
     Q(g) &= g\left( \frac{1}{4\tau^2} -\frac{1}{2\tau} -   \frac{|x^\top|^2}{2\tau|x|^3} + \frac{1}{2\tau|x|}-\frac{|x|}{2\tau^2}\right)\\
      \langle Dg,\frac{\partial x}{\partial \tau} - \vec H\rangle &= g\bar\kappa \frac{\langle x,\hat\nu\rangle}{2\tau|x|} \\
      |\vec H|^2 + \Div_{\tilde M_\tau}\frac{\partial x}{\partial \tau} &= \bar\kappa\kappa\\
        \big| \vec H -  \frac{D^\perp g}{g}\big|^2 &= \big|\kappa+\frac{\langle x,\hat\nu\rangle}{2\tau|x|}\big|^2.
    \end{align*}
    Therefore, we get
    \begin{align*}
      \frac{\de}{\de \tau} \left(\tilde f \int_{\tilde M_\tau}  g\de \tilde s_\tau \right) = & - \tilde  f \int_{\tilde M_\tau}\big| \kappa  +  \frac{\langle x,\hat\nu\rangle}{2\tau|x|}\big|^2 g \de\tilde s_\tau\\
	   & + \tilde f \int_{\tilde M_\tau} g \left[ \frac{1}{4\tau^2} -\frac{1}{2\tau} -   \frac{|x^\top|^2}{2\tau|x|^3} + \frac{1}{2\tau|x|}-\frac{|x|}{2\tau^2} + \bar\kappa \frac{\langle x,\hat\nu\rangle}{2\tau|x|} + \bar\kappa\kappa  \right]\de \tilde s_\tau\\
	  & + \int_{\tilde M_\tau} g \frac{\de}{\de \tau}  \tilde f \de \tilde s_\tau\\
	  & + \tilde f \int_{\partial\tilde M_\tau} g \langle \frac{x}{-2\tau|x|}    ,\mu\rangle  \de\tilde \sigma_\tau,
    \end{align*}
    where integration over $\partial \tilde M_\tau =\{\tildec^\tau(\tilde a),\tildec^\tau(\tilde b)\}$ of $\psi$ with respect to $\de\tilde\sigma_\tau$  means $\psi(\tilde b)-\psi(\tilde a)$.
    By convexity of $\Sigma$ we get $ \int_{\partial\tilde M_\tau} g \langle \frac{x}{-2\tau|x|}    ,\mu\rangle  \de \tilde\sigma_\tau\leq 0$ as in the proof of Proposition~\ref{monocurve}. We use $-a^2 + \bar\kappa a =   - \frac{1}{2}a^2 - \frac{1}{2}(a-\bar\kappa)^2 + \frac{1}{2}\bar\kappa^2$ to get
    \begin{align*}
      \frac{\de}{\de \tau} &\left(\tilde f \int_{\tilde M_\tau}  g\de \tilde \mu_\tau \right)  \leq  \tilde  f \int_{\tilde M_\tau}g \left[ - \frac{1}{2} \big|\kappa-\frac{\langle x,\hat\nu\rangle}{2\tau|x|}\big|^2 - \frac{1}{2}  \big|(\kappa-\bar\kappa)-\frac{\langle x,\hat\nu\rangle}{2\tau|x|}\big|^2  + \frac{1}{2}\bar\kappa^2  \right. \\
      & \hspace*{3.5cm}\left.  +\frac{1}{4\tau^2} -\frac{1}{2\tau} -   \frac{|x^\top|^2}{2\tau|x|^3} + \frac{1}{2\tau|x|}-\frac{|x|}{2\tau^2}   \right]      \de\tilde \mu_\tau + \int_{\tilde M_\tau} g \left(-\frac{1}{2} \bar\kappa^2 \tilde f\right)\de \tilde\mu_\tau\\
   & \leq \tilde  f \int_{\tilde M_\tau}g \left[ - \frac{1}{2} \big|\kappa-\frac{\langle x,\hat\nu\rangle}{2\tau|x|}\big|^2 - \frac{1}{2}  \big|(\kappa-\bar\kappa)-\frac{\langle x,\hat\nu\rangle}{2\tau|x|}\big|^2  +  \frac{1}{4\tau^2}   -\frac{1}{2\tau}  -\frac{|x|}{2\tau^2} \right]      \de\tilde \mu_\tau\\
    & \leq\frac{1}{2\tau^2}\tilde  f \int_{\tilde M_\tau}g \left( \frac{1}{2}-\tau    -|x|\right)\de\tilde \mu_\tau
    \end{align*}
    using $\frac{|x^\top|^2}{-2\tau|x|^3} \leq \frac{1}{-2\tau|x|}$. 
\end{proof}

\begin{lemma} \label{lem2}
Let $c:\ab\times[0,\infty)\to\Rzwei$ be a solution of the area preserving curve shortening flow with Neumann free boundary conditions.
 Under the conditions $L(c_t)\geq c_1>0$ and $|\bar\kappa(t)|\leq c_2<\infty$ $\forall t\in[0,\infty)$ there is a constant $c= c(m, c_1, c_2, L_0, \|\kappa-\bar\kappa\|_2,\|\Sig \kappa\|_\infty,...,\|\ps^{\frac{m}{2}}\Sig\kappa\|_\infty)$ such that 
\begin{align*}
\|\kappa\|_{\infty}+\|\ps\kappa\|_{\infty}+ \dots + \|\ps^{m-1}\kappa\|_{\infty}\leq c
\end{align*}
\end{lemma}

\begin{proof}
 The proof is based on the approach of Lemma~\ref{dtL2} for higher derivatives. Due to Lemma~\ref{abl} and \ref{restint}, we have for $m = 2k$, $k\in\N_0$,
 \begin{align}\begin{split}\label{re30}
  \dt&\left(\int\diff^2 + \mu (\ps\kappa)^2 + \frac{\mu^2}{2}(\ps^2\kappa)^2 + \dots + \frac{\mu^m}{m!}(\ps^m\kappa)^2\de s\right)\\
  &\leq -\frac{1}{2}\int (\ps\kappa)^2 + \mu (\ps^2\kappa)^2 + \dots + \frac{\mu^m}{m!}(\ps^{m+1}\kappa)^2\de s\\
  &+2\mu\left[\ps\kappa\ps^2\kappa\right]_a^b + \dots + 2\frac{\mu^m}{m!}\left[\ps^m\kappa\ps^{m+1}\kappa\right]_a^b +  c\left(\mu^m + \mu^{m-1} + \dots + \mu + 1\right),\end{split}
 \end{align}
so that the only problem is to get an estimate of the boundary terms. Since $m +1$ is odd, we have a formula for $\ps^{m+1}\kappa$ at the boundary in terms of $\ps^m\kappa$ at the boundary and in terms of $\ps^l\Sig\kappa$. Namely, one can proof by induction (as in Lemma~\ref{rand2}): We have at the point $(a,t)$ for $n$ odd 
  \begin{align*}
   \ps^n\kappa=\pt^{\frac{n-1}{2}}&\diff\, \Sig\kappa\circ\fsig^{-1}\circ c\\
   &+ p_{n-1}\left(\pt^{\frac{n-1}{2}-1}(\kappa-\bar\kappa)\right)\, \ps\Sig\kappa\circ\fsig^{-1}\circ c\\
   &+ p_{n-2}\left(\pt^{\frac{n-1}{2}-2}(\kappa-\bar\kappa)\right)\, \ps^2\Sig\kappa\circ\fsig^{-1}\circ c\\
   &+ \dots\\
   & + p_{\frac{n-1}{2} + 2}\left(\pt\diff\right)\, \ps^{\frac{n-1}{2}-1}\Sig\kappa\circ\fsig^{-1}\circ c\\
    & + \diff^{\frac{n-1}{2} +1} \,  \ps^{\frac{n-1}{2}}\Sig\kappa\circ\fsig^{-1}\circ c\\
   & + p_{n+1}(\ps^{n-2}\kappa,\pt^{\frac{n-1}{2}-1}\diff)
  \end{align*}
  and at the point $(b,t)$
    \begin{align*}
   \ps^n\kappa=-\pt^{\frac{n-1}{2}}&\diff\, \Sig\kappa\circ\fsig^{-1}\circ c\\
   &+ p_{n-1}\left(\pt^{\frac{n-1}{2}-1}(\kappa-\bar\kappa)\right)\, \ps\Sig\kappa\circ\fsig^{-1}\circ c\\
   &+ p_{n-2}\left(\pt^{\frac{n-1}{2}-2}(\kappa-\bar\kappa)\right)\, \ps^2\Sig\kappa\circ\fsig^{-1}\circ c\\
   &+ \dots\\
   & + p_{\frac{n-1}{2} + 2}\left(\pt\diff\right)\, \ps^{\frac{n-1}{2}-1}\Sig\kappa\circ\fsig^{-1}\circ c\\
    & +(-1)^{\frac{n-1}{2} + 1} \diff^{\frac{n-1}{2} +1}\,  \ps^{\frac{n-1}{2}}\Sig\kappa\circ\fsig^{-1}\circ c\\
   & + p_{n+1}(\ps^{n-2}\kappa,\pt^{\frac{n-1}{2}-1}\diff).
  \end{align*}  
When we use these formulas for $n=m+1$ and $\pt^{\frac{m}{2}}\kappa = p_{m+1}\left(\pt^{\frac{m}{2}-1}\bar\kappa,\ps^m\kappa\right)$ then we see that the highest order of the derivatives of $\kappa$ is the order $m$.\\
Now, one can use the Gagliardo-Nirenberg inequalities as in Lemma~\ref{lemboundary} to estimate the boundary terms in (\ref{re30}) by ''good`` small terms (which can be absorbed) and by $c(\mu^m + \mu^{m-1} + \dots + \mu + 1)$. That implies 
\begin{align*}
  \dt\left(\int\diff^2 \hspace*{-0.1cm}+ \mu (\ps\kappa)^2\hspace*{-0.1cm} + \frac{\mu^2}{2}(\ps^2\kappa)^2 + \dots + \frac{\mu^m}{m!}(\ps^m\kappa)^2\de s\right)\leq c\left(\mu^m + \mu^{m-1} + \dots + \mu + 1\right).
 \end{align*}
 Then the proof of Corollary~\ref{kappaabsch} can be used to get for even $m$ and for all $t\in[1,\infty)$
\begin{align*}
\|\kappa\|_{\infty}+\|\ps\kappa\|_{\infty}+ \dots + \|\ps^{m-1}\kappa\|_{\infty}\leq c(m).
\end{align*}
\end{proof}

\end{appendix}


\bibliographystyle{plain}
\bibliography{Lit.bib}

\begin{thebibliography}{10}

\bibitem{AbreschLanger}
U.~Abresch and J.~Langer.
\newblock The normalized curve shortening flow and homothetic solutions.
\newblock {\em J. Differential Geom.}, 23(2):175--196, 1986.

\bibitem{Athanassenas1997}
M.~Athanassenas.
\newblock Volume-preserving mean curvature flow of rotationally symmetric
  surfaces.
\newblock {\em Comment. Math. Helv.}, 72(1):52--66, 1997.

\bibitem{Athanassenas2003}
M.~Athanassenas.
\newblock Behaviour of singularities of the rotationally symmetric,
  volume-preserving mean curvature flow.
\newblock {\em Calc. Var. Partial Differential Equations}, 17(1):1--16, 2003.

\bibitem{Athanassenas2012}
M.~Athanassenas and S.~Kandanaarachchi.
\newblock Convergence of axially symmetric volume-preserving mean curvature
  flow.
\newblock {\em Pacific J. Math.}, 259(1):41--54, 2012.

\bibitem{Aubin}
T.~Aubin.
\newblock {\em Some nonlinear problems in {R}iemannian geometry}.
\newblock Springer Monographs in Mathematics. Springer-Verlag, Berlin, 1998.

\bibitem{Brakke}
K.~A. Brakke.
\newblock {\em The motion of a surface by its mean curvature}, volume~20 of
  {\em Mathematical Notes}.
\newblock Princeton University Press, Princeton, N.J., 1978.

\bibitem{Buckland}
J.~A. Buckland.
\newblock Mean curvature flow with free boundary on smooth hypersurfaces.
\newblock {\em J. Reine Angew. Math.}, 586:71--90, 2005.

\bibitem{CabezasMiquel2009}
E.~Cabezas-Rivas and V.~Miquel.
\newblock Volume-preserving mean curvature flow of revolution hypersurfaces in
  a rotationally symmetric space.
\newblock {\em Math. Z.}, 261(3):489--510, 2009.

\bibitem{CabezasMiquel2012}
E.~Cabezas-Rivas and V.~Miquel.
\newblock Volume preserving mean curvature flow of revolution hypersurfaces
  between two equidistants.
\newblock {\em Calc. Var. Partial Differential Equations}, 43(1-2):185--210,
  2012.

\bibitem{Choe}
J.~Choe, M.~Ghomi, and M.~Ritor{\'e}.
\newblock The relative isoperimetric inequality outside convex domains in
  {${\bf R}^n$}.
\newblock {\em Calc. Var. Partial Differential Equations}, 29(4):421--429,
  2007.

\bibitem{DoCarmo}
M.~P. do~Carmo.
\newblock {\em Differential geometry of curves and surfaces}.
\newblock Prentice-Hall Inc., Englewood Cliffs, N.J., 1976.

\bibitem{DziukKuwertSch}
G.~Dziuk, E.~Kuwert, and R.~Sch{\"a}tzle.
\newblock Evolution of elastic curves in {$\mathbb R^n$}: existence and
  computation.
\newblock {\em SIAM J. Math. Anal.}, 33(5):1228--1245 (electronic), 2002.

\bibitem{EckerBuch}
K.~Ecker.
\newblock {\em Regularity theory for mean curvature flow}.
\newblock Progress in Nonlinear Differential Equations and their Applications,
  57. Birkh\"auser Boston Inc., Boston, MA, 2004.

\bibitem{Friedman}
A.~Friedman.
\newblock {\em Partial differential equations of parabolic type}.
\newblock Prentice-Hall Inc., Englewood Cliffs, N.J., 1964.

\bibitem{Gage}
M.~Gage.
\newblock On an area-preserving evolution equation for plane curves.
\newblock In {\em Nonlinear problems in geometry ({M}obile, {A}la., 1985)},
  volume~51 of {\em Contemp. Math.}, pages 51--62. Amer. Math. Soc.,
  Providence, RI, 1986.

\bibitem{Halldorsson}
H.~P. Halldorsson.
\newblock Self-similar solutions to the curve shortening flow.
\newblock {\em Trans. Amer. Math. Soc.}, 364(10):5285--5309, 2012.

\bibitem{HamTrick}
R.~S. Hamilton.
\newblock Four-manifolds with positive curvature operator.
\newblock {\em J. Differential Geom.}, 24(2):153--179, 1986.

\bibitem{Hamilton2}
R.~S. Hamilton.
\newblock The formation of singularities in the {R}icci flow.
\newblock In {\em Surveys in differential geometry, {V}ol.\ {II} ({C}ambridge,
  {MA}, 1993)}, pages 7--136. Int. Press, Cambridge, MA, 1995.

\bibitem{Hamilton}
R.~S. Hamilton.
\newblock Harnack estimate for the mean curvature flow.
\newblock {\em J. Differential Geom.}, 41(1):215--226, 1995.

\bibitem{Huisken87}
G.~Huisken.
\newblock The volume preserving mean curvature flow.
\newblock {\em J. Reine Angew. Math.}, 382:35--48, 1987.

\bibitem{Huisken90}
G.~Huisken.
\newblock Asymptotic behavior for singularities of the mean curvature flow.
\newblock {\em J. Differential Geom.}, 31(1):285--299, 1990.

\bibitem{HuiskenSinestrari}
G.~Huisken and C.~Sinestrari.
\newblock Mean curvature flow singularities for mean convex surfaces.
\newblock {\em Calc. Var. Partial Differential Equations}, 8(1):1--14, 1999.

\bibitem{Klingenberg}
W.~Klingenberg.
\newblock {\em Eine {V}orlesung \"uber {D}ifferentialgeometrie}.
\newblock Springer-Verlag, Berlin, 1973.
\newblock Heidelberger Taschenb{\"u}cher, Band 107.

\bibitem{Andre}
A.~Ludwig.
\newblock {\em A Relaxed Partitioning Disk for Strictly Convex Domains}.
\newblock PhD thesis, {Freiburg: Univ. Freiburg, Fac. of Math. and Phys. 94
  S.}, 2013.

\bibitem{MeineDiss}
E.~M{\"a}der-Baumdicker.
\newblock {\em The area preserving curve shortening flow with Neumann free
  boundary conditions}.
\newblock PhD thesis, {Freiburg: Univ. Freiburg, Fac. of Math. and Phys. 155
  S.}, 2014.

\bibitem{Magni}
A.~Magni and C.~Mantegazza.
\newblock A {N}ote on {G}rayson's {T}heorem.
\newblock To appear in Rend. Sem. Mat. Univ. Padova.

\bibitem{Mantepaper}
C.~Mantegazza, M.~Novaga, and V.~M. Tortorelli.
\newblock Motion by curvature of planar networks.
\newblock {\em Ann. Sc. Norm. Super. Pisa Cl. Sci. (5)}, 3(2):235--324, 2004.

\bibitem{ProtterWeinberger}
M.~H. Protter and H.~F. Weinberger.
\newblock {\em Maximum principles in differential equations}.
\newblock Prentice-Hall Inc., Englewood Cliffs, N.J., 1967.

\bibitem{RitoreSinestrari}
M.~Ritor{\'e} and C.~Sinestrari.
\newblock {\em Mean curvature flow and isoperimetric inequalities}.
\newblock Advanced Courses in Mathematics. CRM Barcelona. Birkh\"auser Verlag,
  Basel, 2010.
\newblock Edited by Vicente Miquel and Joan Porti.

\bibitem{StahlDiss}
A.~Stahl.
\newblock {\em {On the mean curvature flow with Neumann boundary values on
  smooth hypersurfaces. (\"Uber den mittleren Kr\"ummungsflu{\ss} mit
  Neumannrandwerten auf glatten Hyperfl\"achen.)}}.
\newblock PhD thesis, {T\"ubingen: Univ. T\"ubingen, Math. Fak. 129 S.}, 1994.

\bibitem{StahlPaper1}
A.~Stahl.
\newblock Convergence of solutions to the mean curvature flow with a {N}eumann
  boundary condition.
\newblock {\em Calc. Var. Partial Differential Equations}, 4(5):421--441, 1996.

\bibitem{StahlPaper2}
A.~Stahl.
\newblock Regularity estimates for solutions to the mean curvature flow with a
  {N}eumann boundary condition.
\newblock {\em Calc. Var. Partial Differential Equations}, 4(4):385--407, 1996.

\bibitem{Stone}
A.~Stone.
\newblock A density function and the structure of singularities of the mean
  curvature flow.
\newblock {\em Calc. Var. Partial Differential Equations}, 2(4):443--480, 1994.

\bibitem{AchimDipl}
A.~Windel.
\newblock {\"U}ber den volumenerhaltenden mittleren {K}r{\"u}m\-mungs\-flu{\ss}
  mit einer {N}eumann-{R}andbedingung auf einer konvexen {H}yperfl{\"a}che.
\newblock Diplom thesis, Univ. Freiburg, Fac. of Math. and Phys., 100 S., 2006.

\end{thebibliography}
\vspace*{2cm}
\begin{center}

\texttt{Elena M\"ader-Baumdicker\\
 Karlsruhe Insitute of Technology (KIT)\\
 Department of Mathematics\\
Kaiserstr. 89-93, D-76133 Karlsruhe, Germany\\
Email: elena.maeder-baumdicker$@$kit.edu}
 \end{center}
 
\end{document}